\documentclass[11pt]{amsart}

\pdfoutput=1

\usepackage[utf8]{inputenc}
\usepackage[T1]{fontenc}
\usepackage[english]{babel}
\usepackage{bigints}
\usepackage{amsmath}
\usepackage{amsfonts}
\usepackage{amssymb}
\usepackage{amsthm}
\usepackage{mathtools}
\usepackage{xcolor}
\usepackage{centernot}
\usepackage{graphicx}
\usepackage{tikz-cd}
\usepackage{mathabx}
\usepackage{microtype}
\usepackage{enumitem}
\usepackage{hyperref}
\usepackage[nameinlink,capitalize]{cleveref}

\graphicspath{{Fig/}}

\hypersetup{
  colorlinks=true,
  linkcolor=blue,
  citecolor=blue,
  urlcolor=blue
}

\theoremstyle{plain}
\newtheorem{theorem}{Theorem}[section]
\newtheorem{corollary}[theorem]{Corollary}
\newtheorem{lemma}[theorem]{Lemma}
\newtheorem{proposition}[theorem]{Proposition}
\theoremstyle{definition}
\newtheorem{definition}[theorem]{Definition}

\theoremstyle{remark}

\DeclareMathOperator{\Ima}{Im}

\DeclareRobustCommand{\looongrightarrow}{%
\DOTSB\relbar\joinrel\relbar\joinrel\relbar\joinrel\rightarrow
}
\DeclareRobustCommand{\looongmapsto}{\DOTSB\mapstochar\looongrightarrow}

\title[Systolic inequalities for odd-symplectic forms]{On the $C^2$-local systolic optimality of Zoll odd-symplectic forms}

\author{Samanyu Sanjay}
\address{Chair for Geometry and Analysis, RWTH Aachen University, Germany}
\email{sanjay@mathga.rwth-aachen.de}


\begin{document}

\begin{abstract}
An odd-symplectic form is a closed and maximally non-degenerate $2$-form on a compact odd-dimensional manifold. It describes the dynamics of an autonomous Hamiltonian system on a regular energy level. It is called Zoll if the induced dynamics is a free circle action, up to a global time reparameterization. This article establishes a normal form theorem for odd-symplectic forms close to a Zoll one and cohomologous to it, which is then used to prove that Zoll odd-symplectic forms are the local maximizers of the associated systolic ratio. This generalizes the known systolic optimality of Zoll contact forms in the $C^3$-topology. As an application, local systolic inequalities are established in symplectic manifolds for hypersurfaces close to Zoll ones. In particular, this applies to certain non-exact twisted cotangent bundles of manifolds of dimension greater than two.
\end{abstract}

\maketitle

\section{Introduction}
\subsection{Odd-symplectic forms}
Given a connected symplectic manifold $(Q,\omega)$, a smooth real-valued function $H \in C^{\infty}(Q)$ is called an \emph{autonomous Hamiltonian} on $W$, and it defines a Hamiltonian vector field $X_H$ on $Q$. This Hamiltonian vector field is defined to be the unique solution to the following differential equation:
\begin{align}
    \iota_{X_H}\omega =dH
\end{align}
When the dimension of $Q$ is at least four, upon choosing a regular value $r \in \mathbb{R}$ of $H$, the restriction $\varphi^*\omega$ is a closed and maximally non-degenerate $2$-form on the hypersurface $\varphi:H^{-1}(r) \hookrightarrow W$, i.e, it is an instance of an odd-symplectic form. A precise definition of an odd-symplectic form is the following:
\begin{definition}\label{odd-sympelctic_form_def}
    A closed $2$-form $\Omega$ on smooth manifold $\Sigma$ of dimension $2n-1$ for $n\geq 2$ is said to be an {\bf odd-symplectic form} if it has a $1$-dimensional kernel at every point. 
\end{definition}
Every odd-symplectic manifold $(\Sigma,\Omega)$ canonically determines a distribution $\ker\Omega \subset T\Sigma$ called its \emph{characteristic distribution}. Being one-dimensional, the characteristic distribution of an odd-symplectic form is always involutive and determines a foliation of $\Sigma$ by embedded $1$-dimensional submanifolds. This foliation is called the \emph{characteristic foliation} of $\Omega$.\\
\\
If $\Sigma$ is orientable, a choice of orientation on $\Sigma$ and the $2n-2$ form $\Omega^{n-1}$ determine an orientation on the characteristic distribution of $\Omega$ and on the leaves of its characteristic foliation. In such a setting, given an oriented odd-symplectic manifold $(\Sigma,\Omega)$, a $1$-form $\alpha \in \Omega^1(\Sigma)$ is said to be \emph{positively transverse} to $\Omega$ if $\ker\alpha \pitchfork \ker\Omega$ and $\alpha$ evaluates positively on any vector field spanning the distribution $\ker\Omega$ with the correct orientation. The set of $1$-forms positively transverse to an odd-symplectic manifold $(\Sigma,\Omega)$ determine a convex cone in $\Omega^1(\Sigma)$ and will be denoted by $\mathcal{C}(\Sigma,\Omega)$. Given an $\alpha \in \mathcal{C}(\Sigma,\Omega)$, the pair $(\alpha,\Omega)$ is usually called a \emph{framed Hamiltonian} pair on the manifold $\Sigma$ (see \cite{Fish-Hofer},\cite{rohil}). The choice of a framed Hamiltonian pair $(\alpha,\Omega)$ can be used to show that odd-symplectic forms always arise as restrictions of symplectic forms to regular energy levels of autonomous Hamiltonians. Indeed, observe that there exists $\epsilon>0$ such that on the manifold $[-\epsilon,\epsilon] \times \Sigma$, the following $2$-form is a symplectic form: 
\begin{align}
    \tilde \omega := \pi^*_{\Sigma}\Omega + d(a\alpha)
\end{align}
where $a \in [-\epsilon,\epsilon]$ denotes the coordinate on the first factor and $\pi_{\Sigma}$ is the projection to $\Sigma$. This observation motivates the definition of the \emph{Hamiltonian vector field} of the framed Hamiltonian pair $(\alpha,\Omega)$, which is defined to be the unique vector field $R$ on $\Sigma$ satisfying the following equations:
\begin{align}
    \alpha(R) \equiv 1 \qquad \iota_{R}\Omega \equiv 0
\end{align}
The main objects of study in this work are Zoll odd-symplectic forms, which are defined in the following way:
\begin{definition}\label{zoll_odd_symp_form_def}
    A closed, connected, and oriented odd-symplectic manifold $(\Sigma,\Omega)$ is called \textbf{Zoll} if the oriented leaves of its characteristic foliation correspond to the oriented fibers of an $S^1$-bundle with total space $\Sigma$ under an orientation-preserving diffeomorphism. In such a situation, $\Omega$ will be called a Zoll odd-symplectic form.
\end{definition}
It follows from the definition of a Zoll odd-symplectic form that any oriented Zoll odd-symplectic manifold $(\Sigma,\Omega_0)$ determines an isomorphism class of oriented $S^1$-bundles. We will call any representative of this isomorphism class an $S^1$-bundle \emph{associated with} $\Omega_0$. Examples of Zoll odd-symplectic forms are constructed in Section \ref{Examples_section}. 
\subsection{Local systolic inequalities for odd-symplectic forms}
The main result of this work is the following. Its proof can be found in Section \ref{Proof_of_main_thm}. Below, we will make use of the following notational convention: given a smooth manifold $\Sigma$ and a cohomology class $B \in H^k(\Sigma,\mathbb{R})$, we will denote by $\Xi^k_{B}(\Sigma) \subset \Omega^k(\Sigma)$ the set of representatives of this class.
\begin{theorem}\label{systolic_ineq}
Consider a closed, connected and oriented $2n-1$-dimensional smooth manifold $\Sigma$ for $n\geq 2$ admitting a cohomology class $C_0 \in H^2(\Sigma,\mathbb{R})$ satisfying the condition $C_0^{n-1} =0$ that can be represented by a Zoll odd-symplectic form $\Omega_0$. Then, there exists a $C^2$-neighbourhood $\mathcal{U} \subset \Xi^2_{C_0}(\Sigma)$ of $\Omega_0$ such that for any odd-symplectic form $\Omega \in \mathcal{U}$ the following inequalities hold:
\begin{align}\label{systolic_ineq_eqn}
    P(\inf \mathcal{A}_{\Omega}) \leq Vol_{\Omega_0}(\Omega) \leq  P(\sup \mathcal{A}_{\Omega})
\end{align}
with equality holding if and only if $\Omega$ is Zoll. In particular, if $Vol_{\Omega_0}(\Omega) = 0$ then, the following inequality holds:
\begin{align}\label{sys_ineq_second_eqn}
    \inf \mathcal{A}_{\Omega} \leq 0 \leq \sup \mathcal{A}_{\Omega}
\end{align}
with equality holding if and only if $\Omega$ is Zoll.
\end{theorem}
In Theorem \ref{systolic_ineq} above, $\mathcal{A}_{\Omega}$ denotes the Hamiltonian action of the odd-symplectic form $\Omega$ (it is introduced in Section \ref{action_section} below), and $P$ is a polynomial called the Zoll polynomial; it is monotonic close to zero. This polynomial can be thought of as being motivated by the polynomial expression appearing in Weyl's tube formula \cite{weyl}; we introduce it in Section \ref{zoll_poly_section} below. $Vol_{\Omega_0}(\Omega)$ is the Chern-Simons functional for $S^1$-bundles \cite[Remark 2.1]{B-K-Odd_Symp}, which we refer to as the volume function; this volume function is introduced in Section \ref{volume_section} below. The hypothesis $C^{n-1}_0 =0$ in Theorem \ref{systolic_ineq} is equivalent to the non-vanishing of the Euler class of the isomorphism class of $S^1$-bundles associated with the Zoll odd-symplectic manifold $(\Sigma,\Omega_0)$ \cite[Lemma 4.5]{B-K-Odd_Symp}; a precise definition of the Euler class can be found in Section \ref{tech_results_section}. The statement of Theorem \ref{systolic_ineq} when the Euler class vanishes is not new and was proven in \cite{B-K-Odd_Symp}. Applications of Theorem \ref{systolic_ineq} to the study of autonomous Hamiltonian systems are considered in Section \ref{Applications_of_sys_ineq}.\\
\\
It was shown in \cite{B-K-Odd_Symp}, that the inequalities  \eqref{systolic_ineq_eqn} and \eqref{sys_ineq_second_eqn} reduce to the $C^3$-local systolic optimality of Zoll contact forms, in the contact case; a known result proven in \cite{Abbondandolo-Benedetti}. Recall that an oriented odd-symplectic manifold $(\Sigma,\Omega)$ satisfies the contact type condition if $\Omega$ is exact that is, if $\Omega = d\alpha$ for some $\alpha \in \Omega^1(\Sigma)$ and if the $2n-1$-form $\alpha \wedge d\alpha^{n-1}$ is everywhere positive on $\Sigma$ \cite{WEINSTEIN1979353}. In such a situation, the pair $(\Sigma,\alpha)$ is called a strict contract manifold \cite{Geiges_book}. The Reeb vector field of $\alpha$ is the unique vector field $R_{\alpha}$ on $\Sigma$ satisfying the following conditions:
\begin{align}\label{Reeb_odd-symp_eqn}
    \alpha(R_{\alpha}) \equiv 1 \;\;\; \text{and} \;\;\; \iota_{R_{\alpha}} d\alpha \equiv 0
\end{align}
A flow line of $R_{\alpha}$ on $\Sigma$ is called a Reeb orbit of the contact form $\alpha$. An odd-symplectic manifold $(\Sigma,d\alpha)$ of contact type is Zoll if and only if its Reeb flow induces a free $S^1$-action on $\Sigma$.\\
\\
Given a strict contact manifold $(\Sigma,\alpha)$, the associated contact systolic ratio is the following ( \cite{abhs18}, \cite{APB14}):
\begin{align}\label{contat_sys_ratio_dim_3}
    \rho_{\text{sys}}(\alpha):=  \frac{T_{\text{min}}(\alpha)^{n}}{\text{Vol} (\alpha)}
\end{align}
where $T_{\text{min}}(\alpha)$ is the minimal period of a periodic orbit of the Reeb vector field associated of $\alpha$ and $\mathrm{Vol}(\alpha): = \int_{\Sigma} \alpha \wedge d\alpha^{n-1}$ is the contact volume. If a contact form $\alpha \in \Omega^1(\Sigma)$ does not admit any closed Reeb orbits we set $ \rho_{\text{sys}}(\alpha) = + \infty$.\\
\\
The motivation for the definition of the contact systolic ratio comes from the $\pi_1$-systolic ratio in metric geometry, defined in the following way: given a closed, connected smooth manifold $W$ of dimension $n$. If $g$ is a smooth Riemannian metric on $W$, the associated $\pi_1$-systolic ratio is
\begin{align}\label{metric_sys_ratio}
    \rho_{\text{sys}}(g):= \frac{\text{sys}^n(g)}{\text{Vol}(g)}
\end{align}
where $\text{sys}(g)$ is the $g$-length of the shortest non-constant closed $g$-geodesic and $\text{Vol}(g)$ is the Riemannian volume of $W$.  The key question in metric systolic geometry is to find an upper bound on the systolic ratio of metrics $g$ on $W$ that depends only on $W$ and not on the specific metric $g$. Such a uniform bound was first established on non-simply connected surfaces \cite{Pu}, \cite{Bavard}, \cite{Katz_book}, and \cite[Section 7.2]{berger_book}. And in \cite{croke}, it was shown that the ratio \eqref{metric_sys_ratio} was bounded from above on $S^2$. Interestingly, however, the round metric is not the maximizer of the ratio, contradicting what would be expected from what is known about the $2$-torus and $\mathbb{RP}^2$. It was then shown in \cite{balacheff2} that the round metric was a local maximizer in an appropriate topology. It was later conjectured (see \cite{ABHS17}) that the local maximizers of the ratio \eqref{metric_sys_ratio} on $S^2$ in the $C^3$-topology could be characterized as Zoll metrics, i.e., those metrics all of whose geodesics are closed with the same prime period.\\
\\
Up to multiplication by a universal dimensional constant proportional to the volume of an $n$-dimensional Euclidean ball, the ratio \eqref{metric_sys_ratio} agrees with the contact systolic ratio of the strict contact manifold $(S^*_gW,\lambda)$, where $S^*_gW$ is the unit cotangent bundle of $(W,g)$ and $\lambda$ is the restriction of the Liouville $1$-form. Since the dynamics of a tight contact form on the lens space $S_g^*S^2$ is \emph{uniquely} determined by its lift to the double cover $S^3$ (for any smooth Riemannian metric $g$ on $S^2$), motivated by the above-mentioned conjecture, Abbondandolo, Bramham, Hryniewicz, and Salom\~{a}o showed in \cite{abhs18} that Zoll contact forms on $S^3$ were the $C^3$-local maximizes of the contact systolic ratio on the set of contact forms defining the standard contact structure on $S^3$. The same authors were also able to construct counterexamples to this statement in the $C^0_{\text{loc}}$-topology. In \cite{BKcontact} Benedetti and Kang showed that Zoll contact forms maximized the contact systolic ratio in dimension 3. Abbondandolo and Benedetti proved the higher-dimensional analog of this result in \cite{Abbondandolo-Benedetti}. Finally, in \cite{ABE_systolic}, Abbondandolo, Benedetti, and Edtmair showed that Zoll contact forms were the $C^2$-local maximizes of the contact systolic ratio on any contact manifold. The same authors also construct counterexamples to the $C^1$-local systolic optimality of Zoll contact forms.
\subsection{Main technical results}\label{tech_results_section}
In this section, we present the main technical contributions of this work.\\
\\
We preface the statement of these results with a brief description of the geometric setup: given a connected Zoll odd-symplectic manifold $(\Sigma,\Omega_0)$ and an $S^1$-bundle $\Sigma \xrightarrow{p} M$ associated with $\Omega_0$, a $1$-form $\alpha_0 \in \Omega^1(\Sigma)$ is called a \emph{connection $1$-form} if the following holds for some vector field $R$ generating the relevant $S^1$-action on $\Sigma$:
\begin{align}
    \mathcal{L}_R\alpha_0 \equiv 0 \qquad \alpha_0(R)\equiv 1
\end{align}
In particular, fixing a connection $1$-form for an $S^1$-bundle is equivalent to fixing a vector field generating the relevant $S^1$-action on the total space. Upon fixing a connection $1$-form $\alpha_0$ for the $S^1$-bundle $p: \Sigma \rightarrow M$, the real Euler class is defined to be the real cohomology class represented by a closed $2$-form $\nu \in \Omega^2(M)$ such that the following hold:
\begin{align}
    p^*\nu = d\alpha_0
\end{align}
We will denote by $e_0 \in H^2(M,\mathbb{R})$, $(-1)$ times the real Euler class of the chosen $S^1$-bundle.\\
\\
In the remainder of this work, when working with a Zoll odd-symplectic manifold $(\Sigma,\Omega_0)$, we will always assume that the cohomology class $C_0 \in H^2(\Sigma,\mathbb{R})$ represented by $\Omega_0$ satisfies the condition $C^{n-1}_0=0$.
\subsubsection{Normal form theorem}
The following normal form theorem was motivated by \cite[Theorem 2]{Abbondandolo-Benedetti}, its proof can be found in Section \ref{normal_form_section}.
\begin{theorem}\label{normalformthm}
Let $(\Sigma, \Omega_0)$ be a connected Zoll odd-symplectic manifold and fix an $S^1$-bundle $p: \Sigma \rightarrow M$ associated with $\Omega_0$ with connection $1$-form $\alpha_0$. Let $C_0 \in H^2(\Sigma,\mathbb{R})$ be the cohomology class represented by $\Omega_0$. Then, there exists a $C^1$-neighbourhood $\mathcal{U} \subset \Xi_{C_0}^2(\Sigma)$ of $\Omega_0$ with the property that for any odd-symplectic form $\Omega \in \mathcal{U}$, there exists a diffeomorphism $u: \Sigma \rightarrow \Sigma$ isotopic to the identity such that:
\begin{align}\label{normal_form_odd_symp}
    u^* \Omega = \Omega_0 + d(S\alpha_0 + \eta )
\end{align}
where:
\begin{enumerate}
    \item $S\in C^{\infty}(\Sigma)$ is a function that is invariant under the $S^1$-action on $\Sigma$. It coincides with a special case of the Hamiltonian action of $\Omega$ and is defined in \eqref{def_of_S}.
    
    \item $\eta \in \Omega^1(\Sigma)$ satisfies $\iota _{R_0} \eta =0$ for the generating vector field $R_0$ of the $S^1$-action on $\Sigma$ fixed by $\alpha_0$, and,
    
    \item $\iota _{R_0} d\eta = \mathcal{F}(dS)$ where $\mathcal{F}: T^*\Sigma \rightarrow T^*\Sigma$ is a bundle automorphism lifting the identity. 
\end{enumerate}
Moreover, for each integer $k\geq 1$ there exists a monotonically increasing continuous function $\sigma_k: \mathbb{R} \rightarrow \mathbb{R}_{\geq 0}$ such that the following bounds hold:
\begin{equation}\label{bounds_odd_symp}
    \Big\{ dist_{_{C^{k+1}}}(u, Id),  \vert\vert S \vert\vert_{C^{k}}, \vert\vert dS\vert\vert_{C^k}, \vert\vert\eta\vert\vert_{C^{k}},  \vert\vert d\eta\vert\vert_{C^{k}}, \vert\vert\mathcal{F}\vert\vert_{C^k} \Big\} \leq \sigma_{k+1}(\vert\vert\Omega -\Omega_0\vert\vert_{C^{k+1}})
\end{equation}
\end{theorem}
The following is an immediate corollary of the above normal form theorem and is a qualitative refinement of \cite[Theorem 3.10]{Kerman} in the Zoll case; its proof can be found in Section \ref{proof_of_var_principle}. In the sequel, unless otherwise specified, when talking about a cohomology functor, we will always mean the singular cohomology functor computed with the ring of real numbers as coefficients.
\begin{corollary}\label{variationalprinciple}
Given a connected Zoll odd-symplectic manifold $(\Sigma,\Omega_0)$ with associated $S^1$-bundle $\Sigma \xrightarrow{p} M$ and connection $1$-form $\alpha_0$, there exists a $C^1$-neighborhood $\mathcal{U} \subset \Xi_{C_0}^2(\Sigma)$ of $\Omega_0$ such that any odd-symplectic form $\Omega \in \mathcal{U}$ has the property that any vector field spanning its characteristic distribution has at least as many periodic orbits as the cup-length of $M$. Moreover, the following chain of inequalities holds:
\begin{align}\label{var_principle_eqns}
    \inf\mathcal{A}_{\Omega} \leq \min S \leq \max S \leq \sup\mathcal{A}_{\Omega} 
\end{align}
with equality holding if and only if $\Omega$ is Zoll; where $S$ is the function appearing in the statement of Theorem \ref{normalformthm}.
\end{corollary}
\subsubsection{A formula for the volume}
The following theorem furnishes a convenient expression for the Chern-Simons functional close to Zoll odd-symplectic forms. Its proof can be found in Section \ref{vol_formula_section}.
\begin{theorem}\label{Volume_formula}
Let $(\Sigma,\Omega_0)$ be a connected Zoll odd-symplectic manifold and $C_0 \in H^2(\Sigma,\mathbb{R})$ be the cohomology class represented by $\Omega_0$. Then, there exists a $C^2$-neighbourhood $\mathcal{U} \subset \Xi^2_{C_0}(\Sigma)$ such that for any odd-symplectic form $\Omega \in \mathcal{U}$, the volume $Vol_{\Omega_0}(\Omega)$ can be written in the form:
\begin{align}\label{vol_of_omega}
Vol_{\Omega_0}(\Omega) =   \int^1_0 \int_{\Sigma}  S\alpha_0 \wedge (\Omega_0 +rSd\alpha_0)^{n-1} \; dr + \int^1_0  \int_{\Sigma} \frac{1}{r} D(x,rS(x)) \; \alpha_0 \wedge \Omega_0^{n-1} \; dr
\end{align}
where $D: \Sigma \times \mathbb{R}  \rightarrow \mathbb{R}$ is a smooth function such that $D(x,0) = 0$ for every $x \in \Sigma$ and $S$ is the function in the statement of Theorem \ref{normalformthm}. For every fixed $s \in \mathbb{R}$ the function $D(\cdot, s)$ integrates to $0$ over the $\Sigma$. In addition, given an $\epsilon>0$ there exists a 
$C^2$-neighbourhood $\mathcal{U}_{\epsilon} \subset \mathcal{U}$ and a $\delta >0$ such that if $\Omega \in \mathcal{U}_{\epsilon}$ and $s \in [-\delta,\delta]$ then:
\begin{align}\label{volume_formula_bounds}
 \Big|\Big|\frac{\partial}{\partial s}D \Big|\Big|_{C^0(\Sigma \times [-\delta, \delta])} < \epsilon
\end{align}
\end{theorem}
\subsection{Preliminaries}\label{prelm_section}
Throughout this section, we will fix a connected Zoll odd-symplectic manifold $(\Sigma,\Omega_0)$ and an $S^1$-bundle $p: \Sigma \rightarrow M$ associated with the Zoll odd-symplectic form $\Omega_0$; we will also denote by $C_0 \in H^2(\Sigma,\mathbb{R})$ the cohomology class represented by $\Omega_0$.
\subsubsection{The Action}\label{action_section}
In this section, we outline a variational principle that allows us to look for closed leaves of the characteristic foliation of an odd-symplectic form $\Omega \in \Xi^2_{C_0}(\Sigma)$ that is $C^1$-close to $\Omega_0$. By Lemma \ref{Reeb_dist_lemma} and \cite[Proposition B.2]{Abbondandolo-Benedetti}, we expect to find such leaves in a small neighborhood of the fibers of the $S^1$-bundle $\Sigma \xrightarrow{p} M$ associated with $\Omega_0$ in the loop space $\Lambda(\Sigma)$ i.e, the set of smooth ($C^{\infty}$), embedded loops in $\Sigma$. This fact motivates the following definition: we say that an open cover $\{\tilde B_n\}_{n \in \mathbb{N}}$ of $M$ is a \emph{good cover} if pair wise intersections of the form $\tilde B_m \cap \tilde B_k$ are contractible for all $m,k \in \mathbb{N}$. An open cover $\{B_n\}_{n \in \mathbb{N}}$ of $\Sigma$ is called a \emph{good cover} if there exists a corresponding good cover $\{\tilde B_n\}_{n \in \mathbb{N}}$ of $M$ such that the following holds:
\begin{align}
    B_n \subset p^{-1}(\tilde B_n) \qquad \forall\; n\in \mathbb{N}
\end{align}
Given a loop $\gamma_x \subset \Sigma$ passing through a point $x \in \Sigma$, we say that it is a \emph{short loop} if there exists a good neighborhood $B_x$ of the fiber $p_x$ through $x$ such that $\gamma_x \subset B_x$. Upon denoting by $\mathfrak{h} \in \pi_1(\Sigma)$ the homotopy class represented by the fibers of the fixed $S^1$-bundle associated with $\Omega_0$, we will denote by $\Lambda_{\mathfrak{h}}(\Sigma)$ the set of short loops in $\Sigma$. By the inverse function theorem, we know this is an open subset of $\Lambda(\Sigma)$ \cite{Ginzburg-Kerman}. In the sequel, we will always suppress the base point $x \in \Sigma$ of the loop $\gamma_x$ in the notation.\\
\\
Upon fixing an arbitary representative $\Omega_* \in \Xi^2_{C_0}(\Sigma)$, for any odd-symplectic form $\Omega \in \Xi^2_{C_0}(\Sigma)$ there exists an $\alpha \in \Omega^1(\Sigma)$ such that $\Omega = \Omega_*+d\alpha$, we define the action $\mathcal{A}_{\Omega}$ of $\Omega$ in the following way:
\begin{align}\label{action_functional_eq}
    \mathcal{A}_{\Omega}(\gamma):\;&{\Lambda}_{\mathfrak{h}}(\Sigma) \rightarrow \mathbb{R}\\\nonumber \gamma \mapsto \int_{\gamma}&\alpha + \int_{\Gamma_{\gamma}} \Omega_*+d\alpha
\end{align}
where $\Gamma_{\gamma}$ is an embedded cylinder such that $\Gamma_{\gamma} \subset \mathcal{B}_x$ that is obtained parameterizing a homotopy from $\gamma$ to an $S^1$-fiber. We will call such cylinders \emph{short} cylinders.\\
\\
In \cite{B-K-Odd_Symp}, it is shown that the above function is well defined and that it is invariant when pulled back.\\
\\
The relevance of the action functional defined in \eqref{action_functional_eq} is that a loop $\gamma \in  {\Lambda}_{\mathfrak{h}}(\Sigma)$ is a critical point of $\mathcal{A}_{\Omega}$ if and only if it is a leaf of the characteristic foliation of $\Omega$. This was proven in \cite[Corollary 6.5]{B-K-Odd_Symp}. In Corollary \ref{variationalprinciple}, we will use this functional to show that for any odd-symplectic form $\Omega \in \Xi^2_{C_0}(\Sigma)$ that is sufficiently close to $\Omega_0$, $\mathcal{A}_{\Omega}$ has a strictly positive number of critical points in $\Lambda_{\mathfrak{h}}(\Sigma)$.\\
\\
We denote by $\chi(\Omega)$ the set of closed leaves of the characteristic foliation of an odd-symplectic form $\Omega \in \Xi^2_{C_0}(\Sigma)$ and make the following definitions:
\begin{align}
    \inf \mathcal{A}_{\Omega}:= \inf_{\gamma \in \chi(\Omega) \cap \Lambda_{\mathfrak{h}}(\Sigma)}
\mathcal{A}_{\Omega}(\gamma)  \\\nonumber \sup \mathcal{A}_{\Omega}:= \sup_{\gamma \in \chi(\Omega)\cap \Lambda_{\mathfrak{h}}(\Sigma)} \mathcal{A}_{\Omega}(\gamma)
\end{align}
$\inf \mathcal{A}_{\Omega}$ and $\sup \mathcal{A}_{\Omega} $ are finite and close to zero when $\Omega$ is $C^1$-close to $\Omega_0$, see \cite{B-K-Odd_Symp} for a proof.
\subsubsection{The Volume}\label{volume_section}
Here, we introduce an approach of Benedetti and Kang in \cite{B-K-Odd_Symp} to computing a volume-like average of closed $2$-form $\Omega \in \Xi^2_{C_0}(\Sigma)$ on $\Sigma$ with respect to the fixed reference form $\Omega_* \in \Xi^2_{C_0}(\Sigma)$. This average reduces to the contact volume when the $2$-form being considered is an odd-symplectic form of contact type. In addition, as in the contact case, this quantity coincides with the helicity when $\Omega$ is odd-symplectic and exact and $\Sigma$ is of dimension three \cite[Pg. 331]{B-K-Odd_Symp}.\\
\\
We define the volume functional to be:
\begin{align}\label{Vol_function_def_eqn}
    Vol_{\Omega_*}: \; &\Xi^2_{C_0}(\Sigma) \rightarrow \mathbb{R}\\\nonumber
   &\Omega \mapsto \int^1_0 \int_{\Sigma} \alpha \wedge (\Omega_*+rd\alpha)^{n-1} dr
\end{align}
In \cite{B-K-Odd_Symp}, it is shown that the above functional is well defined and invariant under diffeomorphisms isotopic to the identity.\\
\\
The relevance of the volume functional here is that, as shown in \cite[Theorem 6.14]{B-K-Odd_Symp}, if $\Omega$ is a Zoll odd-symplectic form, the functional $Vol_{\Omega_*}(\Omega)$ defined in \eqref{Vol_function_def_eqn} is a polynomial. More precisely, since $\Omega_0$ is a Zoll odd-symplectic form, there exists a symplectic form $\omega_0 \in \Omega^2(M)$ such that $p^*\omega_0 = \Omega_0$. Upon denoting by $c_0 \in H^2(M,\mathbb{R})$ the cohomology class represented by $\omega_0$, the following holds:
\begin{align}
    p^*c_0 = C_0
\end{align}
In this setting, it is shown in \cite[Theorem 6.14]{B-K-Odd_Symp} that the following holds when $\Omega \in \Xi^2_{C_0}(\Sigma)$ is a Zoll odd-symplectic form:
\begin{align}
    Vol_{\Omega_*}(\Omega) = \int^{\mathcal{A}_{\Omega}}_0 \langle (c_0+te_0)^{n-2},[M] \rangle \; dt \;\;\; \text{for} \;\;\; \text{where } {\mathcal{A}_{\Omega}} \in \mathbb{R}\;\;\; \text{is the action of $\Omega$}
\end{align}
In addition, the above polynomial is monotonic close to zero because the class $c_0$ can be represented by a symplectic form on $M$. Indeed, the following is true:
\begin{align}
     \frac{ Vol_{\Omega_0}(\Omega)}{d A}\Big|_{A=0} = \langle c_0^{n-2}, [M] \rangle >0
\end{align}
This observation motivates the definition of the Zoll polynomial defined in the following section.
\subsubsection{The Zoll polynomial}\label{zoll_poly_section}
The Zoll polynomial is defined in the following way:
\begin{align}\label{Zoll_poly_def_eqn}
     P(A) = \int^A_0 \langle (c_0+te_0)^{n-2},[M] \rangle \; dt \;\;\; \text{for} \;\;\; A \in \mathbb{R}
\end{align}
where $e_0 \in H^2(M,\mathbb{R})$ is $(-1)$ times the real Euler class, $[M]$ is the fundamental class of $M$ and $c_0 \in H^2(M,\mathbb{R})$ is such that $p^*c_0 = C_0 \in H^2(\Sigma, \mathbb{R})$. It follows from \eqref{Zoll_poly_def_eqn} that
\begin{align}\label{montonicity_of_zoll_poly}
    \frac{dP(A)}{d A}\Big|_{A=0} = \langle c_0^{n-2}, [M] \rangle >0 \;\;\; 
\end{align}
where the last inequality follows from the fact that $c_0$ can be represented by a symplectic form.
\subsection{Applications}\label{Applications_of_sys_ineq}
\subsubsection{Systolic inequalities for hypersurfaces in symplectic manifolds}
In this section, we illustrate an application Theorem \ref{systolic_ineq} by using the relative complexity of the dynamics on hypersurfaces in a symplectic manifold to compare the growth of volumes of tubular neighborhoods of such hypersurfaces. We preface the statement of the main results of this section with a brief outline of the geometric setup.\\
\\
Given a connected symplectic manfiold $(W,\omega)$ and a closed, connected and embedded hypersurface $\tilde \Sigma \hookrightarrow W$, by slightly abusing terminology, we will refer to the characteristic foliation of the odd-symplectic form $\omega\vert_{\Sigma}$ as the characteristic foliation of $\tilde \Sigma$. Since any embedded hypersurface of a symplectic manifold $(W,\omega)$ naturally inherits an orientation form the symplectic manifold, we define a Zoll hypersurface to be one whose oriented characteristic foliation corresponds with the foliation induced by the fibers of an oriented $S^1$-bundle under an orientation preserving diffeomorphism. Given another closed, connected and oriented hypersurface $\Sigma$ that is $C^2$-close to a Zoll one, we pick a closed leaf of the characteristic foliation of $\Sigma$ and define its action to be the following:
\begin{align}
    \mathcal{A}_{\Sigma}(\gamma):= \int_{\Gamma_{\gamma}} \omega
\end{align}
where $\Gamma_{\gamma} \hookrightarrow W$ is an embedded cylinder interpolating between $\gamma$ and a leaf $\gamma_0$ of the characteristic foliation of the Zoll hypersurface that is $C^1$-close to $\gamma$. It is shown in Section \ref{proof_of_hyp_sys_ineq_section} that the action is well defined and bounded when $\Sigma$ is sufficiently close to a Zoll hypersurface. We will also denote by $\text{sys}(\Sigma)$ the infimum of the action of $\Sigma$ over closed leaves of the characteristic foliation of $\Sigma$ that are $C^1$-close to those of a fixed Zoll hypersurface. We will denote by $\tilde D(\Sigma)$ the region in $W$ between $\Sigma_0$ and $\Sigma$.\\
\\
If $\Sigma_0$ is Zoll, then there exists an $\epsilon>0$ such that a tubular neighborhood of $\Sigma_0$ of radius $\epsilon>0$ is foliated by Zoll hypersurfaces. We will denote the leaves of this foliation by $\hat{\Sigma}_t$ for $t \in (-\epsilon,\epsilon)$ and we will assume that the parameter $t$ was chosen so that $\text{sys}(\hat{\Sigma}_t) = t$ (see Section \ref{proof_of_hyp_sys_ineq_section} for more details). We observe here that it is a priori unclear if Zoll hypersurfaces exist in an arbitrary symplectic manifold, especially if one is interested in hypersurfaces that lie outside the paradigm of Weinstein's contact type condition, as we are in this work. We present the construction of such examples in Theorem \ref{Examples_of_Zoll_odd-symplectic_forms}.\\
\\
The following is the first main result of this section; its proof can be found in Section \ref{proof_of_hyp_sys_ineq_section}:
\begin{theorem}\label{systolic_ineq_for_hyp_sur}
    Let $(W,\omega)$ be a connected symplectic manifold of dimension $2n$ for $n\geq 2$ and $\Sigma_0 \hookrightarrow W$ a connected and embedded Zoll hypersurface. Then, for any closed, connected, oriented and embedded non-intersecting hypersurface $\Sigma \hookrightarrow W$ that is sufficiently $C^2$-close to $\Sigma_0$, the following holds:
    \begin{align}
        Vol\big(\tilde D(\hat{\Sigma}_{\text{sys}(\Sigma)})\big)\leq nVol\big(\tilde D(\Sigma)\big)
    \end{align}
    with equality holding if and only if the hypersurface $\Sigma \hookrightarrow W$ is Zoll, where the volumes above are computed with respect to $\omega$.
\end{theorem}

\subsubsection{Systolic inequalities for magnetic systems}\label{mag_sys_sec_intro}
Given a Riemannian manifold $(W,g)$ of dimension at least two, the metric Hamiltonian (or kinetic energy) of $g$ is the following:
\begin{align}
    H_g: T^*W \rightarrow \mathbb{R}\\\nonumber
    (q,p)\mapsto \frac{\vert\vert p\vert\vert^2_g}{2}
\end{align}
Upon picking a closed $2$-form $\sigma \in \Omega^2(\Sigma)$, the following is a symplectic form on $T^*W$:
\begin{align}\label{mag_form_def_intro}
    \Omega_{g,\sigma}:= d\lambda - \pi^*\sigma 
\end{align}
where $\lambda \in \Omega^1(T^*W)$ is the Liouville $1$-form and $\pi:T^*W \rightarrow W$ is the usual foot-point projection. The Hamiltonian vector field determined by $H_g$ and $\Omega_{g,\sigma}$, that is, the unique solution to the following differential equation, is called the magnetic (or twisted) geodesic flow of the pair $(g,\sigma)$:
\begin{align}
    \iota_{X_{g,\sigma}}\Omega_{g,\sigma} = dH_{g}
\end{align}
The Hamiltonian dynamical system determined on the manifold $T^*W$ by the flow $\Phi_{g,\sigma}$ of the vector field $X_{g,\sigma}$ is called the magnetic system on $(W,g)$ determined by the pair $(g,\sigma)$; we will refer to such pairs as magnetic pairs. The projections $\pi(\Phi_{g,\sigma})$ to $W$ are called the magnetic geodesics of the (magnetic) pair $(g,\sigma)$ when they are parametrized by arc length. Observe that if $\sigma$ was chosen to be the $2$-form that is constantly zero everywhere on $W$ then the magnetic geodesic flow of the pair $(g,\sigma)$ coincides with the geodesic flow of the Riemannian manifold $(W,g)$ and the magnetic geodesics on $W$ are just the geodesics of the metric $g$. The study of magnetic systems was initiated by V. Arnol'd \cite{Arnold1988} and S.P Novikov \cite{Novikov} and has since grown into a vibrant area of study in symplectic dynamics; see \cite{CieliebakFrauenfelderPaternain2010} for a nice introduction to the subject.\\
\\
From the point of view of theoretical physics, magnetic dynamical systems can be thought of as modeling the motion of a charged particle in a magnetic field represented by $\sigma$. (cf. \cite{CieliebakFrauenfelderPaternain2010}). \\
\\
Given a $k \in \mathbb{R}$, the flow of the vector field $X_{g,\sigma}$ on a regular energy level $H_{\text{kin}}^{-1}(k)$ can be conjugated to that of the Hamiltonian vector field determined by the metric Hamiltonian of $g$ and the following symplectic form on $S_g^*W$ \cite{paternain_helicity}:
\begin{align}
    \Omega^s_{g,\sigma}:= d\lambda -s\pi^*\sigma \;\;\; \text{for $s = \frac{1}{\sqrt{2k}}$}
\end{align}
The constant $s$ is called the \emph{strength} of the magnetic system determined by the pair $(g,\sigma)$.\\
\\
A magnetic pair $(g,\sigma) \in \mathcal{M}(W) \times \Omega^2(W)$ will be called \textbf{Zoll} if for some $s \in \mathbb{R}$, the magnetic geodesic flow of the pair $(g,\sigma)$ of strength $s$ induces a free $S^1$-action on $S^*_gW$. In such a setting, given another pair $(\tilde g,\tilde\sigma) \in \mathcal{M}(W) \times \Omega^2(W)$, we will denote by $\chi(\tilde g, \tilde \sigma)$ the set of closed magnetic geodesics of the pair $(\tilde g,\tilde\sigma)$ of strength $s$ that are $C^2$-close to those of the pair $(g,\sigma)$ of the same strength. In this section, since we are interested in applications of Theorem \ref{systolic_ineq}, we will always fix a Zoll magnetic pair and consider the dynamics of another magnetic pair relative to this fixed Zoll magnetic pair. In such a setting, we will always that both magnetic systems are being considered at a fixed strength where the fixed (Zoll) pair is Zoll. \\
\\
It is also worth observing here that there are examples of magnetic systems that are Zoll at some energies but not at others. One such class of examples was constructed by Bimmerman in \cite{j_bimmerman} (see also \cite{paternain_horocycle}), we recall the relevant version of Bimmerman's statement here as Theorem \ref{zollness_of_mag_systems_kahler}.\\
\\
The first result of this section is the following; its proof can be found in Section \ref{mag_sys_ineq_general_proof}: 
\begin{theorem}\label{mag_sys_ineq_general}
Let $(W,g)$ be a closed, connected Riemannian manifold of dimension at least three with the property that there exists a cohomology class $C_0 \in H^2(W,\mathbb{R})$ admitting a representative $\sigma \in \Xi^2_{C_0}(W)$ such that the magnetic pair $(g,\sigma)$ is Zoll. Then, there exists a $C^3$-neighborhood $\mathcal{U} \subset \mathcal{M}(W) \times \Xi^2_{C_0}(W)$ of the pair $(g,\sigma)$ with that property that for any pair $(\tilde g,\tilde \sigma) \in \mathcal{U}$ satisfying the condition: $\text{Vol}_{\tilde g}(W) = \text{Vol}_{g}(W)$, the following holds:
\begin{align}\label{mag_sys_ineq_general_eqn}
\inf_{\gamma \in \chi(\tilde g,\tilde \sigma)}\Big(\text{length}_{\tilde g}(\gamma)-\int_{\pi_*\Gamma_{\gamma}} \tilde \sigma\Big) \leq \text{length}_{g}(\gamma_0) - \int_{\gamma_0}\eta
\end{align}
with equality holding if and only if the pair $(g,\sigma)$ is Zoll;
where $\eta \in \Omega^1(W)$ is any $1$-form such that $\sigma = \sigma_0+d\eta$, $\gamma \in \chi(\tilde g,\tilde \sigma)$ and $\gamma_0 \in \chi(g,\sigma)$ are such that their lifts to $S^*_gW$ can be connected by a short homotopy represented by the cylinder $\Gamma_{\gamma} \subset S^*_gW$.
\end{theorem}
As observed by Benedetti and Kang in \cite[Lemma 4.5]{B-K-Odd_Symp}, the integral homology class, i.e., the class in $H_1(W,\mathbb{Z})$ represented by the magnetic geodesics of a Zoll magnetic pair, is torsion. If this class vanishes then by for every Zoll magnetic geodesic $\gamma$ there exists an immersed surface $\mathbb{D}_{\gamma} \hookrightarrow W$ whose boundary coincides with $\gamma$ \cite[Proposition 3.4.8]{Geiges_book}. We will call such a surface a capping surface for $\gamma$.\\
\\
We now fix a Zoll magnetic pair $(g,\sigma)$ such that the integral homology class represented by its magnetic geodesics is the zero class. Given another magnetic pair $(\tilde g,\tilde \sigma) \in \mathcal{M}(W)\times \Omega^2(W)$ such that the set $\chi(\tilde g,\tilde \sigma)$ is non-empty, we define the magnetic length of the magnetic geodesics in $\chi(\tilde g,\tilde \sigma)$ in the following way:

\begin{align}\label{mag_length_functional_def}
    &l_{\text{mag}}^{\tilde g,\tilde \sigma}: \chi(\tilde g,\tilde \sigma) \rightarrow \mathbb{R}\\\nonumber
    \gamma &\mapsto \text{length}_{\tilde g}(\gamma) - \int_{\mathbb{D}_{\gamma}} \tilde\sigma
\end{align}
where $\mathbb{D}_{\gamma}$ is a capping surface for the magnetic geodesic $\gamma$ obtained by concatenating to a capping surface $\mathbb{D}_{\gamma_0}$ of a magnetic geodesic $\gamma_0 \in \chi(g,\sigma)$, the projection to $W$ of a cylinder $\Gamma_{\gamma} \hookrightarrow S^*_gW$ parameterizing a short homotopy between the lifts of $\gamma$ and $\gamma_0$ to $S^*_gW$; a canonical lift exists for each such magnetic geodesic, see Section \ref{mag_dyn_riem_section} for more details. In this setting, the inequality in \eqref{mag_sys_ineq_general_eqn} can now be expressed in the following succinct form:
\begin{align}
    l_{\text{min}}^{\tilde g,\tilde\sigma} \leq c(g,\sigma)  
\end{align}
where the constant $c(g,\sigma)$ is the value of the magnetic length functional $l^{g,\sigma}_{\text{mag}}$ evaluated on a geodesic of the $(g,\sigma)$-magnetic system and $l_{\text{min}}^{\tilde g,\tilde\sigma}$ is defined in the following way:
\begin{align}\label{l_min_def_equation}
    l_{\text{min}}^{\tilde g,\tilde\sigma}:= \inf_{\gamma\in \chi(\tilde g,\tilde\sigma)} l_{\text{mag}}^{\tilde g,\tilde\sigma}(\gamma) 
\end{align}
When the manifold $W$ in Theorem \ref{mag_sys_ineq_general} is a K\"ahler manifold where the associated K\"ahler metric $g_0$ has constant holomorphic sectional curvature denoted by $\kappa$ (see Section \ref{mag_dynamics_on_kahler_section} for precise definitions), it was shown by Bimmermann (see Theorem \ref{zollness_of_mag_systems_kahler} below) that the magnetic pair $(g_0,\sigma_0)$ where $\sigma_0$ is the K\"ahler form, is Zoll at strengths satisfying: $s^2+\kappa >0$. A consequence of this result of Bimmermann's is also the fact that the integral homology class represented by Zoll magnetic geodesics is always zero. In this setting, we also have enough topological information to compute the Zoll polynomial explicitly, as a consequence, we obtain the next main result of this section; in it we will always assume that all the magnetic systems are being considered at a strength satisfying the condition $s^2+\kappa>0$. This result can be seen as a generalization of the main result in \cite{BK-mag_surfaces} and its proof can be found in Section \ref{mag_sys_ineq_proof}:
\begin{theorem}\label{mag_sys_ineq}
Consider a closed K\"ahler manifold $(M^{2n},\sigma_0,J)$ with the property that the associated K\"ahler metric $g_0$ has constant holomorphic sectional curvature $\kappa$ and denote by $C_0 \in H^2(M,\mathbb{R})$ the cohomology class represented by $\sigma_0$. Then, there exists a $C^3$-neighborhood $\mathcal{U} \subset \mathcal{M}(M) \times \Xi^2_{C_0}(M) $ of the pair $(g_0,\sigma_0)$ with the property that the magnetic system determined by any pair $(\tilde g,\tilde \sigma) \in \mathcal{U}$ satisfies the following local systolic inequality:
\begin{align}
    \Bigg(
        l_{\text{min}}^{\tilde g,\tilde \sigma}
        \Big(
            \kappa l_{\text{min}}^{\tilde g,\tilde \sigma}
            -&4\pi^2 a^2(1)
        \Big)
    \Bigg)^{n-1}
    \Bigg(
        \kappa l_{\text{min}}^{\tilde g,\tilde \sigma}
        \big(na^2(1)-1\big)
        -4\pi  a^2(1)
    \Bigg) \\\nonumber \leq& \frac{2^{2n+1}\pi^{2n}a^2(1)}{n!(2n-1)!} \Bigg(\frac{\text{Vol}_{\tilde g}(M)}{\text{Vol}_{g_0}(M)}-1\Bigg)
\end{align}
\begin{align}
     \Big(l_{\text{min}}^{\tilde g,\tilde \sigma}\Big)^{n-1} \leq \frac{\pi^n}{n!\;(2n-1)! }\Bigg(\frac{\text{Vol}_{\tilde g}(M)}{\text{Vol}_{g_0}(M)}-1\Bigg)
\end{align}
with equality holding if and only if the magnetic system determined by the pair $(\tilde g,\tilde \sigma)$ is also Zoll. If in addition, $\text{Vol}_{g_0}(M) = \text{Vol}_{\tilde g}(M)$ then, the above inequalities simplify to the following:
   \begin{align}\label{mag_sys_ineq_2}
  l_{\text{min}}^{g, \sigma} \leq a^2(1)\pi
   \end{align}
   once again with equality if and only if the magnetic pair $(\tilde g,\tilde \sigma)$ is Zoll; where $l_{\text{min}}^{g,\sigma}$ is as defined in \eqref{l_min_def_equation}. The constant $a^2(1)$ has the following expression:
   \begin{align}
    a^2(1)= \frac{2}{\kappa}\Big(\sqrt{s^2+\kappa} -s\Big) 
\end{align}
   \end{theorem}
When the K\"ahler manifold $M$ above is of dimension two, Theorem \ref{mag_sys_ineq} is a generalization of the main result in \cite{benedetti2019local}, and when $\text{Vol}_{\tilde g}(M) = \text{Vol}_{g_0}(M)$, it reduces to this result. This is because the following is true:
\begin{align}\label{bk_mag_3}
    a^2(1)\pi = \frac{2\pi}{\sqrt{s^2+\kappa}-s}
\end{align}
Combining \eqref{mag_sys_ineq_2} with the above equation and accounting for the sign change due to the assumption \cite[Equation (1.1)]{benedetti2019local}, we see that \eqref{mag_sys_ineq_2} reduces to the first inequality in \cite[Theorem 1.9]{benedetti2019local} as required.
\section*{Acknowledgments}
I am very grateful to my advisors, to Gabriele Benedetti for suggesting the problem to me and for his valuable help and guidance while working on the project and writing it down, and to Umberto Hryniewicz for numerous helpful discussions, for his advice on writing, and for the environment at MathGA. I am also very grateful to B.Albach, L.Asselle, V.Assenza, D.Bechara, J.Bimmermann, L.Dahinden U.Fuchs, R.Loiola, F.Morabito, and M.Voguel for several useful discussions. I am supported by the DFG SFB/TRR 191 “Symplectic Structures in Geometry,
Algebra and Dynamics”, Projektnummer 281071066-TRR 191.
\section{Proof of Theorem \ref{systolic_ineq}}\label{Proof_of_main_thm}
We begin by fixing an $S^1$-bundle $p: \Sigma \rightarrow M$ associated with $\Omega_0$ and a connection $1$-form $\alpha_0$. We now denote by $\mathcal{U} \subset \Xi^2_{C_0}(\Sigma)$ the neighborhood of $\Omega_0$ provided by Theorem \ref{Volume_formula}. Upon denoting by $\tilde{\mathcal{U}} \subset \Xi^2_{C_0}(\Sigma)$ the neighborhood of $\Omega_0$ provided by Theorem \ref{normalformthm}, we can assume without any loss of generality that $\mathcal{U}$ was chosen so that the following holds:
\begin{align}
   \mathcal{U} \subset \tilde{\mathcal{U}}
\end{align}
We now pick an arbitrary odd-symplectic form $\Omega \in \mathcal{U}$ and suppose that the following holds: $\text{Vol}_{\Omega_0}(\Omega) =0$. In this case, the required local systolic inequality is the following:
\begin{align}\label{sys_ineq_proof_1}
    \inf \mathcal{A}_{\Omega} \leq 0 \leq \sup\mathcal{A}_{\Omega}
\end{align}
It is shown in \cite[Proposition 7.4]{B-K-Odd_Symp} that if $\Omega$ is sufficiently close to $\Omega_0$ then both $\inf \mathcal{A}_{\Omega}$ and $\sup\mathcal{A}_{\Omega}$ lie in a neighborhood of zero and that the size of the neighborhood of zero can be controlled by the distance between $\Omega$ and $\Omega_0$. As a consequence, we can assume without any loss of generality that the neighborhood $\mathcal{U}$ was chosen so that $\inf \mathcal{A}_{\Omega}$ and $\sup\mathcal{A}_{\Omega}$ lie in a neighborhood of zero where the Zoll polynomial is monotonic. \eqref{sys_ineq_proof_1} can then be seen to follow from \cite[Proposition 7.4]{B-K-Odd_Symp} and Corollary \ref{variationalprinciple}.\\
\\
Suppose now that $\text{Vol}_{\Omega_0}(\Omega) \neq 0$ then, by Theorem \ref{Volume_formula}, the following holds:
\begin{align}\label{sys_ineq_volume_expansion}
  Vol_{\Omega_0}(\Omega) =  \int^1_0 \int_{\Sigma}  S\alpha_0 \wedge (\Omega_0 +rSd\alpha_0)^{n-1} \; dr + \int^1_0 \frac{1}{r} \int_{\Sigma} D(x,rS(x)) \; \alpha_0 \wedge \Omega_0^{n-1} \; dr
\end{align}
where $S$ is the $S^1$-invariant function provided by Theorem \ref{normalformthm}. We claim that the above expression for the volume functional is monotonic in the variable $S$. To prove this claim, we first consider the case where $S$ is constant. In this case, by Theorem \ref{Volume_formula}, the second summand in \eqref{sys_ineq_volume_expansion} vanishes. Let $\bar S \in C^{\infty}(M)$ be the quotient of $S$ by the relevant $S^1$-action on $\Sigma$ and we now project the first summand in \eqref{sys_ineq_volume_expansion} to $M$ by integrating along the fibers (see \cite{BottandTu} or \cite[Pg. 386]{B-K-Odd_Symp}) to obtain the following:
\begin{align}\label{monotonic_in_S_1}
  \int^1_0 \int_{\Sigma}  S\alpha_0 \wedge (\Omega_0 +rSd\alpha_0)^n \; dr  = \int^1_0 \int_{M}  \bar{S}(\omega_0 +r\bar{S} \nu)^n dr
\end{align}
where $\omega_0 \in \Omega^2(M)$ is a symplectic form satisfying: $p^*\omega_0 = \Omega_0$ and $\nu \in \Omega^2(M)$ is the $2$-form satisfying the relation $p^*\nu = d\alpha_0$.\\
\\
It can now readily be seen that the integral in \eqref{monotonic_in_S_1} is the Zoll polynomial and monotonicity of the expression \eqref{sys_ineq_volume_expansion} in the variable $S$ follows by shrinking the neighborhood $\mathcal{U}$ of $\Omega_0$ if necessary to ensure that $S$ is contained in the neighborhood of zero where the Zoll polynomial is monotonic.\\
\\
On the other hand, if $S$ is not constant then there exists a smooth function $Q: \Sigma \times
\mathbb{R} \times \mathbb{R} \rightarrow \mathbb{R}$ defined by the following relation:
\begin{align}\label{sys_ineq_proof_Q}
 Q(x,S(x),r) \; \alpha_0 \wedge \Omega_0^{n-1} \; dr = S\alpha_0 \wedge (\Omega_0 +rSd\alpha_0)^{n-1} \; dr
\end{align}
Upon differentiating both sides of \eqref{sys_ineq_proof_Q} with respect to $S$, we obtain the following: 
\begin{align}
    \frac{\partial Q}{ \partial S} \big|_{S=0}  \; \alpha_0 \wedge \Omega^{n-1}_0 = \alpha_0 \wedge \Omega_0^{n-1} 
\end{align}
We can therefore conclude that the following holds:
\begin{align}\label{sys_ineq_partial_Q}
    \frac{\partial Q}{ \partial S} \Big|_{S=0} =1
\end{align}
By Theorem \ref{Volume_formula}, given an $\epsilon >0$, there exists a $\delta >0$ such that if $\vert S\vert <\delta$ then, the following holds:
\begin{align*}
    \Bigg\vert \Bigg\vert\frac{1}{r} \frac{\partial D(x,rS)}{\partial S}\Bigg\vert\Bigg\vert_{C^0(\Sigma \times [-\delta,\delta])}  < \epsilon
\end{align*}
The above inequality combined with the observation in \eqref{sys_ineq_partial_Q} shows that $Vol_{\Omega_0}(\Omega)$ is monotonic in $S$ if $\vert S\vert <\delta$, a fact we can always guarantee by shrinking the neighborhood $\mathcal{U} \subset\Xi^2_{C_0}(\Sigma)$ is necessary.\\
\\
Since the expression in \eqref{Vol_function_def_eqn} is monotonic in the variable $S$, we obtain the following chain of inequalities:
\begin{align}\label{sys_ineq_1}
     \int^1_0 \int_{\Sigma}  &S_{\text{min}}\alpha_0 \wedge (\Omega_0 +rS_{\text{min}}d\alpha_0)^{n-1} \; dr \\\nonumber \leq \;\; &Vol_{\Omega_0}(\Omega) \\\nonumber \leq \int^1_0 \int_{\Sigma}  &S_{\text{max}}\alpha_0 \wedge (\Omega_0 +rS_{\text{max}}d\alpha_0)^{n-1} \; dr
\end{align}
where $S_{\text{min}}:= \inf_{x\in \Sigma} S(x)$ and $S_{\text{max}}:= \sup_{x \in \Sigma} S(x)$ and any of these inequalities is an equality if and only if $S$ is constant.\\
\\
Since the $2$-form $\nu \in \Omega^2(M)$ was chosen so that $p^*\nu = d\alpha_0$, upon projecting to the base $M$ of the $S^1$-bundle associated with $\Omega_0$ by integrating along the $S^1$-fibers, for every constant $S_0$ we obtain:
\begin{align}\label{sys_ineq_zoll_poly_vol_relation_leq}
       \int^1_0 \int_{\Sigma}  S_{0}\alpha_0 \wedge &(\Omega_0 +rS_{0}d\alpha_0)^{n-1} \; dr =  \int^1_0 \int_{M}  \bar{S}_{0}(\omega_0 +r\bar{S}_{0} \nu)^{n-1} dr = \\\nonumber &\int^1_0 \langle (c_0+r\bar{S}_{0}e_0)^{n-1} , [M]\rangle \;dr = P(\bar{S}_{0})\end{align}
where $c_0 \in H^2(M, \mathbb{R})$ is the cohomology class represented by $\omega_0 \in \Omega^2(M)$ and $e_0 \in H^2(\Sigma, \mathbb{R})$ is $(-1)$ times the real Euler class of the $S^1$-bundle associated with $\Omega_0$.\\
\\
Replacing the constant $S_0$ in \eqref{sys_ineq_zoll_poly_vol_relation_leq} with $S_{\text{min}}$ and $S_{\text{max}}$ and using \eqref{sys_ineq_1} and noting that the action and the volume are invariant under diffeomorphisms isotopic to the identity, we obtain:
\begin{align}
   P(\inf \mathcal{A}_{\Omega}) \leq P(\bar{S}_{\text{min}}) \leq Vol_{\Omega_0}(\Omega) \leq P(\bar{S}_{\text{max}}) \leq P(\sup \mathcal{A}_{\Omega})
\end{align}
with equality holding if and only if $S$ is constant.\\
\\
Finally, recall from Corollary \ref{variationalprinciple} that $S$ is constant if and only if $\Omega$ is Zoll and in this case $\inf \mathcal{A}_{\Omega} = \min S = \max S = \sup \mathcal{A}_{\Omega}$.

\section{Examples of Zoll odd-symplectic forms}\label{Examples_section}
This section is dedicated to a result in which we present the construction of a large class of examples of Zoll odd-symplectic forms that are not necessarily of contact type. We preface the statement of this result by giving a brief outline of the geometric setup: let $(W,\omega)$ be a connected symplectic manifold of an arbitrary dimension and $Q \hookrightarrow W$ be a closed, connected, symplectic submanifold. Then, $Q$ can be realized as the unique Bott non-degenerate minimum of a smooth function $H_0:W \rightarrow \mathbb{R}$. In this setting, the following is the main result of this section:
\begin{theorem}\label{Examples_of_Zoll_odd-symplectic_forms}
    Let $(W,\omega)$ be a connected symplectic manifold of an arbitrary dimension. Then, for every closed, connected, embedded symplectic submanifold $Q \hookrightarrow W$ of a positive dimension then, there exists tubular neighborhood $T(Q) \subset W$ of $Q$ and a smooth function $H_0 \in C^{\infty}(T(Q))$ such that for every regular value $r \in \mathbb{R}$ of $H_0$, the following odd-symplectic form is Zoll:
    \begin{align}
        \omega\vert_{H^{-1}(r)}
    \end{align}
    Moreover, the above odd-symplectic form will be exact if and only if $\omega\vert_Q$ is exact.
\end{theorem}
\begin{proof}
We denote by $p: NQ^{\perp_{\omega}} \rightarrow Q$ the symplectic orthogonal bundle to $Q$ and observe that the pair $(NQ^{\perp_{\omega}},\bar\omega)$ is a symplectic vector bundle over $Q$, where $\bar \omega:= \omega|_{T(Q)}$.\\ 
\\
We now fix an almost complex structure $J_0$ on $NQ^{\perp_{\omega}}$ that is compatible with $\bar{\omega}$ and denote the resulting fiber metric by $\bar g$. We then pick a connection $\nabla$ on $NQ^{\perp_{\omega}}$ such that the following hold:
\begin{align}
    \nabla g_0 \equiv 0 \qquad \nabla J_0 \equiv 0
\end{align}
$\nabla$ allows us to decompose $TNQ^{\perp_{\omega}}$ into the vertical (denoted by $\mathcal{V}$) and horizontal (denoted by $\mathcal{H}$) distributions in the following way where we denote a base point in $NQ^{\perp_{\omega}}$ by $(q,p) \in NQ^{\perp_{\omega}}$:
\begin{align}
    pr_{\mathcal{H}}: &T_{(q,p)}NQ^{\perp_{\omega}} \rightarrow \mathcal{\mathcal{H}}(q,p)  \\\nonumber
   &\xi \mapsto \xi_{\mathcal{H}} := d_{p}p(\xi) 
\end{align}
\begin{align}
    pr_{\mathcal{V}}: &T_{(q,p)}NQ^{\perp_{\omega}}\rightarrow \mathcal{V}(q,p) \\\nonumber
   &\xi \mapsto \xi_\mathcal{V} := \nabla_{t}(Z(0)) 
\end{align}
where $Z:(-\epsilon, \epsilon) \rightarrow NQ^{\perp_{\omega}}$ is a curve such that $Z(0) = p$ and $\Dot{Z}(0) = \xi$ and $\nabla_t$ represents the covariant derivative with respect to $\nabla$. Therefore, we can write $TNQ^{\perp_{\omega}}$ in the following way:
\begin{align}\label{V_H_TW}
    TNQ^{\perp_{\omega}} \cong\mathcal{H} \oplus  \mathcal{V} 
\end{align}
In these coordinates, the following is called the angular $1$-form on $NQ^{\perp_{\omega}}$:
\begin{align}\label{angular_tau_def_eqn}
    \tau_{(q,p)}(\xi):= \bar g(\mathcal{P}_{\mathcal{V}}(\xi),J_0p) = \bar\omega(\mathcal{P}_{\mathcal{V}}(\xi),p) \;\;\;\text{for $q\in Q$, $p \in NQ_q^{\perp_{\omega}}$ and $\xi \in TNQ^{\perp_{\omega}}$} 
\end{align}
where $\mathcal{P}_{\mathcal{V}}: TNQ^{\perp_{\omega}} \rightarrow \mathcal{V}$ is the projection to the vertical distribution. We now claim that if the tubular neighborhood $NQ^{\perp_{\omega}}$ is chosen to be sufficiently small, the following $2$-form is a symplectic form on $NQ^{\perp_{\omega}}$:
\begin{align}
    \Omega_{\text{cpl}}:=\frac{1}{2} d\tau-p^*\omega
\end{align}
where we denote by $p: NQ^{\perp_{\omega}} \rightarrow Q$, the bundle projection and by a slight abuse of notation, we denote by $\omega$ the restriction $Q\big|_{\omega}$. Indeed, observe that from the definition of the angular $1$-form $\tau$, the following holds at the zero-section of $NQ^{\perp_{\omega}}$:
\begin{align}
    \frac{1}{2}d\tau_{(q,0)}(\xi_\mathcal{V},\zeta_{\mathcal{V}}) = \bar \omega(\xi_\mathcal{V},\zeta_{\mathcal{V}}) \neq 0
\end{align}
The final equality above follows from the fact that $(NQ^{\perp_{\omega}},\bar \omega)$ is a symplectic vector bundle. Similarly, the following also holds at the zero-section:
\begin{align}
     \frac{1}{2}d\tau_{(q,0)}(\xi_\mathcal{H},\zeta_{\mathcal{H}}) = 0
\end{align}
And we obtain the following:
\begin{align}
    -p^*\omega_{(q,0)}(\xi_\mathcal{H},\zeta_{\mathcal{H}}) = -\omega_x(\xi_\mathcal{H},\zeta_{\mathcal{H}}) \neq 0
\end{align}
As a consequence of the above observations, at the zero-section, $\Omega_{\text{cpl}}$ admits the following block-diagonal decomposition with respect to the orthogonal splitting \eqref{V_H_TW}:
\begin{align}\label{coupling_form_diagonal}
    \Omega_{\text{cpl}}(x,0)= \begin{pmatrix}
        -\omega_Q &0 \\
        0& d\tau_{(x,0)}
    \end{pmatrix}
\end{align}
Also by the discussion above, each of the diagonal blocks in \eqref{coupling_form_diagonal} is non-degenerate; therefore, the following holds:
\begin{align}
    \text{det}_{(x,0)} \big(\Omega_{\text{cpl}}\big) \neq 0
\end{align}
Picking a chart $U\subset Q$ and trivializing the bundle $NQ^{\perp_{\omega}}$ over $U$, we see that by \eqref{coupling_form_diagonal}, the following function is smooth on $U$:
\begin{align}
 \text{det}\big(\Omega_{\text{cpl}}&\big): NQ^{\perp_{\omega}}\big\vert_{U} \rightarrow \mathbb{R}\\\nonumber 
 (q,p&) \mapsto \text{det}\big(\Omega_{\text{cpl}}\big)(q,p)
\end{align}
As a consequence, the following holds:
\begin{align}
    \lim_{q\rightarrow 0} \text{det}\big(\Omega_{\text{cpl}}\big)(q,p) = \text{det}\big(\Omega_{\text{cpl}}\big)(q,0)
\end{align}
Therefore, there exists an $\epsilon_U>0$ such that the following holds, as a consequence of the fact that $\text{det}\big(\Omega_{\text{cpl}}\big)(q,0) \neq 0$
\begin{align}
    \text{det}\big(\Omega_{\text{cpl}}\big)(q,p) \neq 0 \qquad \forall\; (q,p) \; \text{such that } \vert\vert p\vert\vert_{\bar g}^2 < \epsilon_U
\end{align}
Finally, since $Q$ is compact, we can cover it by finitely many charts of the form used above, and repeating the above procedure for each such chart, we see that there exists an $\epsilon>0$ such that the $2$-form $\Omega_{\text{cpl}}$ is non-degenerate on the following tubular neighborhood of the zero-section of the bundle $NQ^{\perp_{\omega}}$:
\begin{align}
    T(Q):=\big\{(q,p)\in NQ^{\perp_{\omega}}\; \vert \; \vert\vert p\vert\vert^2_{\bar g} < \epsilon \big\}
\end{align}
Consider now the following Hamiltonian on the symplectic manifold $(T(Q),\Omega_{\text{cpl}})$:
\begin{align}
    H_0: &T(Q) \longrightarrow \mathbb{R}\\\nonumber
    (q,&p)\mapsto -\pi \vert\vert p \vert\vert_{\bar g}^2 \;\;\;\; \text{where $\pi \in \mathbb{R}$ is the real constant} 
\end{align}
We now claim that the following Hamiltonian vector field is full-periodic with prime period always equal to one:
\begin{align}
    \iota_{X_0}\Omega_{\text{cpl}} =  dH_0
\end{align}
Indeed, the Hamiltonian vector field $X_0$ is the one generated by the fiber-wise rotation $e^{2\pi J_0t}:TT(Q) \rightarrow TT(Q)$. To see this, we begin by observing that the differential $dH_0$ always evaluates a zero on vectors contained in the horizontal distribution of the splitting \eqref{V_H_TW}. So, the vector field $X_0$ is completely contained in the vertical distribution of $TT(Q)$ determined by chosen connection $\nabla$. As a consequence, the following holds:
\begin{align}
    \iota_{X_0}p^*\omega=0
\end{align}
From the definition of $H_0$, the following is also clear:
\begin{align}
    dH_0(\xi)=-2\pi \bar g(p,\xi) \qquad \forall\; \xi \in \mathcal{V}(q,p)
\end{align}
We now observe that by Weinstein's neighborhood theorem \cite[Theorem 3.4.10]{mcduffsalamon2017introduction}, the two symplectic manifolds $(T(Q),\bar \omega)$ and $(T(Q),\Omega_{\text{cpl}})$ are symplectomorphic. Therefore, the following holds:
\begin{align}
    \frac{1}{2}d\tau(\cdot,J\cdot) = \bar g(\cdot,\cdot)\big\vert_{\mathcal{V}}
\end{align}
We now denote by $2\pi Jp$ the vector field generated by the fiber-wise rotation $e^{2\pi J_0t}:TT(Q) \rightarrow TT(Q)$, and compute:
\begin{align}
    \iota_{2\pi Jp}\frac{1}{2}d\tau(\xi) = \bar g(2\pi Jp,\xi) = dH_0(\xi) \qquad \forall \; \xi \in \mathcal{V}(q,p)
\end{align}
Finally, observe that the inclusion $\iota: Q \hookrightarrow T(Q)$ is a deformation retraction since the following holds:
\begin{align}
    \iota \circ p = \text{Id}_Q
\end{align}
Therefore, if $p^*\omega$ is exact in $T(Q)$ the following holds:
\begin{align}
    0=\iota^*[p^*\omega]=[\omega] \in H^2(Q,\mathbb{R})
\end{align}
Conversely, if $\omega\vert_Q = d\beta $ for some $\beta \in \Omega^1(Q)$ then,
\begin{align}
    p^*\omega=p^*d\beta =d(p^*\beta)
\end{align}
\end{proof}
It is worth observing here that in the three-dimensional setting, as observed in \cite{rohil}, the trajectories of any volume-preserving flow on a closed $3$-manifold can be realized as the characteristic foliation of an odd-symplectic form. Therefore, if such a flow induces a free $S^1$-action on the underlying manifold, the corresponding odd-symplectic form will be Zoll. Such odd-symplectic forms will not be of contact type if and only if the associated $S^1$-bundle is trivial and they have been completely classified up to orientation-preserving diffeomorphisms of the associated $S^1$-bundles, see \cite[Proposition 1.11]{B-K-Odd_Symp}.\\
\\
In Theorem \ref{Examples_of_Zoll_odd-symplectic_forms} above, if the submanifold $Q$ has exactly half the dimension of $W$, the example falls into the class of Hamiltonian dynamical systems called magnetic systems In \cite[Theorem A \& Theorem B]{johanna_negative_curvature}, Bimmermann constructs examples of Zoll odd-symplectic forms that are not of contact type within the paradigm of such Hamiltonian systems. In fact, the construction of Zoll odd-symplectic forms in Theorem \ref{Examples_of_Zoll_odd-symplectic_forms} can be thought of as a generalization of the ideas leading to the proof of Bimmermann's result to a more general geometric context.
\section{Proof of the normal form theorem}\label{normal_form_section}
Recall the setup of Theorem \ref{normalformthm}: $(\Sigma,\Omega_0)$ is a connected $2n-1$-dimensional Zoll odd-symplectic manifold; we denote by $C_0 \in H^2(\Sigma,\mathbb{R})$ the class represented by $\Omega_0$. We also fix an $S^1$-bundle $p: \Sigma \rightarrow M$ associated with $\Omega_0$ and a connection $1$-form $\alpha_0 \in \Omega^1(\Sigma)$ and we denote by $R_0$ the vector field fixed by $\alpha_0$ and generating the relevant $S^1$-action on $\Sigma$. In particular $R_0$ is the Hamiltonian vector field of the pair $(\alpha_0,\Omega_0)$. We also fix a Riemannian metric $g$ on $\Sigma$ that is invariant with respect to the flow of $R_0$.\\
\\
Since we are working with smooth (i.e $C^{\infty}$) objects, the set $\Xi^2_{C_0}(\Sigma)$ of representatives of the cohomology class $C_0$ can be given the $C^k$-topology for any $k \geq 0$. So, we fix the $C^1$-topology on $\Xi^2_{C_0}(\Sigma)$ and note that by Lemma \ref{Reeb_dist_lemma} there exists a neighborhood $\mathcal{U} \subset \Xi^2_{C_0}(\Sigma)$ of $\Omega_0$ such that $\alpha_0$ is positively transverse to $\ker(\Omega)$ for any odd-symplectic form $\Omega \in \mathcal{U}$. So, upon fixing such a neighborhood $\mathcal{U}$ of $\Omega_0$ we denote by $R_{\Omega}$ the Hamiltonian vector field of the pair $(\alpha_0,\Omega)$ for an odd-symplectic form $\Omega \in \mathcal{U}$. Lemma \ref{Reeb_dist_lemma} in the Appendix tells us that for each $k \geq 0$ there exists a monotonically increasing continuous function $\sigma_k: \mathbb{R} \rightarrow \mathbb{R}_{\geq 0}$ such that:
\begin{align}
    ||R_{\Omega} - R_0||_{C^k} \leq \sigma_{k}(||\Omega - \Omega_0||_{C^{k}})
\end{align}
for $\Omega \in \mathcal{U}$; we refer to such functions as moduli of continuity.  Using this fact and \cite[Theorem 2.1]{Abbondandolo-Benedetti} we can assume without loss of generality that the neighbourhood $\mathcal{U}$ is chosen such that for any $\Omega \in \mathcal{U}$ the following holds:
\begin{align}\label{bottkol_decomp_for_reeb}
     hu^*R_{\Omega} = R_{0} - \mathcal{P}_{u}[V]
\end{align}
such that:\begin{enumerate}
    \item $u:\Sigma \rightarrow \Sigma$ is a diffeomorphism isotopic to the identity.
    \item $V$ is a vector field on  $\Sigma$ satisfying $\mathcal{L}_{R_{0}} V =0$ and $g(R_{0}, V) =0 $.
    \item $h: \Sigma \rightarrow \mathbb{R}$ is a function satisfying $\mathcal{L}_{R_{0}} h =0$. 
    \item $\mathcal{P}_{u}: T\Sigma \rightarrow T\Sigma$ is a bundle isomorphism dependent on the diffeomorphism $u$ and lifting the identity. 
    \item For each integer $k \geq 0$ there exist moduli of continuity $\tilde{\sigma}_k$ satisfying:
\begin{align}\label{normal_form_bounds_1}
 \max \Big\{dist_{C^{k+1}}(u,Id_{\Sigma}), &\vert\vert h-1\vert\vert_{C^{k+1}}, \vert\vert V\vert\vert_{C^{k+1}}, \vert\vert\mathcal{P}_u -Id \vert\vert_{C^{k}} \Big\} \\\nonumber \leq \;\; &\tilde{\sigma}_{k+1}(\vert\vert R_{\Omega} - R_{0} \vert\vert_{C^{k+1}})
\end{align}
\end{enumerate}
Recall that we denote by $\mathfrak{h}$ the homotopy class of the fibers of $\Sigma \xrightarrow{p} M$ and by $\Lambda_{\mathfrak{h}}(\Sigma)$ the set of short loops in $\Sigma$ (c.f Section \ref{action_section}). Since we have fixed a neighborhood $\mathcal{U} \subset \Xi^2_{C_0}(\Sigma)$ of $\Omega_0$ such that $\alpha_0$ is positively transverse to any odd-symplectic form $\Omega \in \mathcal{U}$ and \eqref{bottkol_decomp_for_reeb} holds for the Hamiltonian vector field $R_{\Omega}$, we can now fix an arbitrary odd-symplectic form $\Omega \in \mathcal{U}$ and define $\Tilde{\Omega}:= u^*\Omega$ so that $hR_{\Tilde{\alpha_0}, \Tilde{\Omega}} = R_{0} - \mathcal{P}_{u}(V)$ where $\Tilde{\alpha_0}:= u^*\alpha_0$. By \cite[Lemma 1.1]{Abbondandolo-Benedetti} the following bounds hold:
\begin{align}\label{bounds_of_omega_tilde_minus_omega}
    &\vert \vert \tilde{\Omega} - \Omega_0\vert \vert_{C^{k}} \leq \vert \vert\tilde{\Omega} - \Omega\vert \vert_{C^{k}} + \vert \vert\Omega - \Omega_0\vert \vert_{C^k } \\ \leq \nonumber
    \vert \vert\Omega\vert \vert_{C^{k+1}} \Big(1+ &\vert \vert du\vert \vert^{k+2}_{C^{k}} \Big) dist_{C^{k+1}}(u,\text{Id}_{\Sigma})  +\vert \vert\Omega - \Omega_0\vert \vert_{C^k}  \leq \sigma_{k+1}(\vert \vert\Omega - \Omega_0\vert \vert_{C^{k+1}})
\end{align}
The last inequality above is justified by \eqref{normal_form_bounds_1} and by the fact that by Lemma \ref{Reeb_dist_lemma} for each such $k$, there exists a modulus of continuity $\sigma_{k}$ such that:
\begin{align}
     \tilde{\sigma}_{k}(\vert \vert R_{\Omega} - R_{0} \vert \vert_{C^{k}}) \leq \sigma_{k}( \vert \vert \Omega - \Omega_0\vert \vert_{C^{k}})
\end{align}
The function $S_{\tilde{\Omega}}: \Sigma \rightarrow \mathbb{R}$ is defined in the following way:
\begin{align}\label{def_of_S}
     S_{\tilde{\Omega}}(x):= \mathcal{A}_{\tilde{\Omega}}(\bar{\Gamma}_{x}) \;\;\;\text{where $\bar{\Gamma}_{x}$ is the homotopy $\bar{\Gamma}_x(\cdot,t) = p_x$ $\forall t \in [0,1]$   } x \in \Sigma 
\end{align}
where $\mathcal{A}_{\tilde{\Omega}}$ is defined in \eqref{action_functional_eq}.\\
\\
Given a representation $\tilde{\Omega} = \Omega_0 + d\alpha$ it is easy to see that $S_{\tilde{\Omega}}$ has the following convenient expression:
\begin{align}\label{def_of_S_ind}
    S_{\tilde{\Omega}}(x):= \int_{p_x} \alpha \;\;\; \text{for } x\in \Sigma
\end{align}
By \cite[Proposition 6.7]{B-K-Odd_Symp} the representation of the function $S_{\Tilde{\Omega}}$ in \eqref{def_of_S_ind} is independent of the choice of parametrizing $1$-form $\alpha$ since our assumption that $C_0^{n-1} =0$ implies that the fibers of the bundle $\Sigma \xrightarrow{p} M$ are null-homologous \cite[Lemma 4.5]{B-K-Odd_Symp}. Since we have fixed $\Tilde{\Omega}$, we suppress the subscript in $S_{\Tilde{\Omega}}$ for ease of notation. The function $S$ is invariant under the flow of $R_0$ and note that for any $1$-form $\alpha$ such that $\Tilde{\Omega} = \Omega_0 + d\alpha$, the difference $\alpha - S\alpha_0$ satisfies:
\begin{align}
    \int_{\mathfrak{p}_x} \alpha - S\alpha_0 =0 \;\; \text{forall $x \in \Sigma$}
\end{align}
\\
\\
Therefore, by \cite[Lemma 1.3]{Abbondandolo-Benedetti} the following is true:
\begin{align}\label{alpha_eta_alpha_0}
    \alpha = S\alpha_0 + \eta +df
\end{align}
where $f$ is some smooth function and $\eta \in \Omega^1(\Sigma)$ satisfies $\iota_{R_{0}} \eta =0$. Therefore $\Tilde{\Omega}$ has the following expression:
\begin{align}\label{Omega_tilde_full_expansion}
    \Tilde{\Omega} = \Omega_0 + dS \wedge \alpha_0 + Sd\alpha_0 + d\eta
\end{align}
We will denote by $\Phi^t_0$ the flow of the vector field $R_0$ and we now make the following definition:
\begin{align}\label{def_of_K}
    K_{u}: R_{0}^{\perp} &\looongrightarrow [R^{\perp}_{0}]^* \nonumber\\
X &\looongmapsto \int_{S^1}  u^*\Omega(\mathcal{P}_{u}(d\Phi^t_0(X)), \; d\Phi^t_0\;\cdot \;) \; dt
\end{align}
Since $\Omega$ is $C^1$-close to $\Omega_0$ Lemma \ref{Reeb_dist_lemma} tells us that $R_{\Omega}$ is $C^1$-close to $R_0$ therefore, by the argument in \cite[Claim 3.9]{Kerman} the map $K_{\text{Id}_{\Sigma}}$ is invertible. Since the neighborhood $\mathcal{U}$ was chosen so that $u$ is $C^1$-close to the identity on $\Sigma$, using Neumann's series argument, we know that non-degeneracy is an open condition. So, we see can conclude that the map $K_u$ is also invertible. \\
\\
We know from \cite[Pg. 33]{B-K-Odd_Symp} the function $\mathcal{A}_{\Tilde{\Omega}}$ is a primitive of the following $1$-form:
\begin{align}\label{action_form_def_eqn}
 \mathfrak{a}_{\Tilde{\Omega}}(\xi(t))|_{\gamma}:= \int_{S^1} \Tilde{\Omega}(\xi(t), \Dot{\gamma}(t)) \text{ } \forall \gamma \in \Lambda_{\mathfrak{h}}(\Sigma) \text{ and } \xi \in T_{\gamma}\Lambda_{\mathfrak{h}}(\Sigma)
\end{align}
So, it is clear that $dS$ has the following form:
\begin{align}\label{dS=K(V)}
dS|_x[w] = d\mathcal{A}_{\tilde{\Omega}}\big|_{\bar{\Gamma}_x}[d\Phi^t_0 (w)] \stackrel{\eqref{action_form_def_eqn}}{=} \int_{S^1} \tilde{\Omega}(d\Phi^t_0(w) ,\frac{ \partial p_x}{\partial t}) \; dt \\\nonumber \stackrel{\eqref{bottkol_decomp_for_reeb}}{=}  \int_{S^1} \tilde{\Omega}( \mathcal{P}_u((V \circ \Phi^t_0)) - h(\Phi^t_0) R_{\tilde{\Omega}}(u \circ \Phi^t_0), \; d\Phi^t_0(w)) \; dt \\\nonumber
= \int_{S^1} \tilde{\Omega}( \mathcal{P}_u((V \circ \Phi^t_0)), \; d\Phi^t_0 (w) \;) \; dt = K_u(V)[w]
\end{align}
where $\frac{ \partial p_x}{\partial t} = \frac{ \partial \Phi^t_0(x)}{\partial t}$ and the final equality follows from the fact that $\mathcal{L}_{R_0}(V) =0$.\\
\\
So using \eqref{Omega_tilde_full_expansion}, \eqref{def_of_K} and \eqref{dS=K(V)}, the contraction $\iota_{R_{0}} d\eta$ can be written as:
\begin{align}\label{decompositon_of_odd_symp_forms_nft_proof_2}
    \iota_{R_{0}} d\eta = \iota_{R_0} \tilde{\Omega} + \iota_{R_0} Sd\alpha_0 -\iota_{R_0} \alpha_0 \wedge dS
   \stackrel{\eqref{bottkol_decomp_for_reeb}}{=} \iota_{\mathcal{P} \circ K_{u}^{-1}(dS)} \tilde{\Omega} - dS
\end{align}
where in the second equality above, we have used the fact that since $\alpha_0$ is a connection $1$-form for the $S^1$-bundle $p: \Sigma \rightarrow M$, there exists a $2$-form $\nu \in \Omega^2(M)$ such that $p^*\nu = d\alpha_0$. As a consequence, the contraction $\iota_{R_0} d\alpha_0$ vanishes. \\
\\
We now define the endomorphism $\mathcal{F}: T^*\Sigma \rightarrow T^*\Sigma$ in the statement of Theorem \ref{normalformthm} in the following way
\begin{align}\label{def_of_F}
    \mathcal{F}(\lambda) = \iota_{\mathcal{P}\circ K^{-1}_u(\lambda|_{[R_0]^{\perp}})} \Tilde{\Omega} - \lambda|_{[R_0]^{\perp}} \;\; \text{for $\lambda \in T^*\Sigma$}
\end{align}
where $[R_0]^{\perp} \subset T\Sigma$ is the distribution that is perpendicular to the linear span of $R_0$ with respect to the fixed invariant metric $g$.\\
\\
Clearly, $\mathcal{F}$ is a linear endomorphism and since $\mathcal{P}$ and $K_u$ lift the identity on $\Sigma$, so does $\mathcal{F}$.\\
\\
Putting \eqref{alpha_eta_alpha_0}, \eqref{decompositon_of_odd_symp_forms_nft_proof_2} and \eqref{def_of_F} together we see that:
\begin{align}\label{tilda_omega_final_form}
    \Tilde{\Omega} = \Omega_0 + d(S\alpha_0 + \eta)
\end{align}
such that $\eta \in \Omega^1(\Sigma)$ satisfies:
\begin{align}
  \iota_{R_0} \eta =0 \;\;\; \& \;\;\; \iota_{R_0} d\eta = \mathcal{F}(dS)
\end{align}
So, in order to finish the proof of Theorem \ref{normalformthm}, we need to only establish \eqref{bounds_odd_symp}.\\
\\
Towards that end, we first remind the reader that by Lemma \ref{Reeb_dist_lemma} we can assume without any loss of generality that the neighborhood $\mathcal{U}$ is chosen such that for each $k \geq 1$ there exists a modulus of continuity $\sigma_k$ such that:
\begin{align}\label{normal_form_bounds_2}
    \tilde{\sigma}_{k}(\vert\vert R_{\Omega} - R_{0}\vert\vert_{C^{k}}) \leq \sigma_{k}(\vert\vert\Omega - \Omega_0\vert\vert_{C^{k}})
\end{align}
In each step of the sequel below, we will successively obtain new moduli of continuity by shrinking the previous ones but for ease of notation, we will continue to denote these by $\sigma_k$.\\
\\
By Proposition \ref{elliptic_bounds_on_1_forms} and \eqref{bounds_of_omega_tilde_minus_omega} there exists an $\Hat{\alpha} \in \Omega^1(\Sigma)$ and moduli of continuity $\sigma_k$ for $k \geq 1$ such that 
\begin{align}\label{nft_proof_alpha_hat_bound}
    &\tilde{\Omega} = \Omega_0 + d\hat{\alpha} \\\nonumber 
    \vert\vert \hat{\alpha}\vert\vert_{C^k}& \leq \sigma_{k}(\vert\vert\Omega - \Omega_0\vert\vert_{C^{k}})
\end{align}
By \cite[Lemma 4.5]{B-K-Odd_Symp}, our assumption that $C_0^{n-1}=0$ implies that the representation of $S$ in \eqref{def_of_S_ind} is independent of the $1$-form $\alpha$, therefore:
\begin{align*}
    S_{\tilde{\Omega}}(x) = \int_{p_x} \hat{\alpha}
\end{align*}
So, for the sake of continuity, we will continue to write $\alpha$ for $\hat{\alpha}$.\\
\\
We therefore obtain the following:
\begin{align}\label{bound_on_S}
    \vert\vert S_{\tilde{\Omega}}\vert\vert_{C^{k}} \leq \sigma_{k+1}(\vert\vert\Omega - \Omega_0\vert_{C^{k+1}})
\end{align}
We know from \eqref{dS=K(V)} that $dS = K_u(V)$ for $V$ as in \eqref{bottkol_decomp_for_reeb}. In addition, by \cite[Lemma 1.1]{Abbondandolo-Benedetti} the following holds:
\begin{align}\label{nft_bound_3}
    \vert\vert \tilde{\Omega}\vert\vert_{C^k} \leq \vert\vert\Omega\vert\vert_{C^k} \cdot \vert\vert du\vert\vert^2_{C^k}(1+\vert\vert du\vert\vert_{C^{k-1}})
\end{align}
Using \eqref{normal_form_bounds_1} and combining \eqref{def_of_K} with \eqref{nft_bound_3} we obtain the following:
\begin{align}\label{bound_on_dS}
    \vert\vert dS_{\tilde{\Omega}}\vert\vert_{C^k} \leq \vert\vert K_u\vert\vert_{C^k} \vert\vert V\vert\vert_{C^k} \leq \sigma_{k+1}(\vert\vert \Omega - \Omega_0\vert\vert_{C^{k+1}})
\end{align}
Using the bound \eqref{nft_proof_alpha_hat_bound} on the $1$-form we now call $\alpha$ and \cite[Lemma 1.3]{Abbondandolo-Benedetti}, the following holds:
\begin{align}\label{nft_eta_bound_1}
    \vert\vert \eta\vert\vert_{C^k} \leq \vert\vert \alpha -S\alpha_0\vert\vert_{C^k} + \vert\vert d\iota_{R_0}( \alpha -S\alpha_0)\vert\vert_{C^k}
\end{align}
In light of this fact, \cite[Lemma 1.1]{Abbondandolo-Benedetti} allows us to deduce the following bounds:
\begin{align}\label{nft_eta_bound_2}
     \vert\vert\alpha -S\alpha_0\vert\vert_{C^k} \leq &\vert\vert\alpha - u^*\alpha\vert\vert_{C^k} + \vert\vert u^*\alpha - S\alpha_0\vert\vert_{C^k} \\ \nonumber \leq \Big[\vert\vert\alpha\vert\vert_{C^{k+1}} \text{ } dist_{C^{k+1}}(u,Id_{\Sigma})&(1+\vert\vert du\vert\vert^{k+1}_{C^k}) \Big] + \vert\vert u^*\alpha\vert\vert_{C^k} + \vert\vert S\alpha_0\vert\vert_{C^k} \\\nonumber \leq &\sigma_{k+1}(\vert\vert\Omega -\Omega_0\vert\vert_{C^{k+1}})
\end{align}
Similarly, we arrive at the following bounds for $\vert\vert d(\iota_{R_0} \alpha -S\alpha_0)\vert\vert_{C^k}$:
\begin{align}\label{nft_eta_bound_3}
    \vert\vert d(\iota_{R_0} \alpha -S\alpha_0)\vert\vert_{C^k} \leq &\vert\vert d(\iota_{R_0}\alpha)\vert\vert_{C^k} + \vert\vert dS \wedge \alpha_0 \vert\vert_{C^k} + \vert\vert S d\alpha_0\vert\vert_{C^k} \\\nonumber \leq &\sigma_{k+1}(\vert\vert \Omega -\Omega_0\vert\vert_{C^{k+1}})
\end{align}
Plugging \eqref{bound_on_dS}, \eqref{nft_eta_bound_2} and \eqref{nft_eta_bound_3} into \eqref{nft_eta_bound_1} we obtain:
\begin{align}
    \vert\vert \eta\vert\vert_{C^k} \leq \sigma_{k+1}(\vert\vert \Omega -\Omega_0\vert\vert_{C^{k+1}})
\end{align}
For the required bound on $\vert\vert d\eta\vert\vert_{C^k}$ observe that:
\begin{align}\label{d_eta_bounds}
    \vert\vert d\eta\vert\vert_{C^{k}} \leq \vert\vert \tilde{\Omega} - \Omega_0\vert\vert_{C^{k}} + \vert\vert dS\wedge \alpha_0\vert\vert_{C^{k}} + \vert\vert Sd\alpha_0\vert\vert_{C^{k}}
\end{align}
Applying the bounds in \eqref{bounds_of_omega_tilde_minus_omega} and \eqref{bound_on_dS} to \eqref{d_eta_bounds} we obtain
\begin{align}\label{nft_deta_bound_final}
\vert\vert d\eta\vert\vert_{C^k}    \leq \sigma_{k+1}(\vert\vert \Omega - \Omega_0\vert\vert_{C^{k+1}})
\end{align}
We now only need to show that $\vert \vert \mathcal{F}\vert\vert_{C^k} \leq \sigma_{k+1}(\vert\vert\Omega -\Omega_0\vert\vert_{C^{k+1}})$. Towards this end, we make the following definitions:
\begin{align}
    \mathcal{F}_1:= \iota_{K^{-1}_u|_{[R_0]^{\perp}}} \Omega -Id|_{[R_0]^{\perp}} \\\nonumber
    \mathcal{F}_0:= \iota_{K^{-1}_0|_{[R_0]^{\perp}}} \Omega_0 -Id|_{[R_0]^{\perp}}
\end{align}
where $K_0: R_0^{\perp} \rightarrow [R^{\perp}_{0}]^*$ is defined by:
\begin{align}
    X \mapsto \int_{S^1} \Omega_0(d\Phi^t_0 (X),\; d\Phi^t_0 \;\cdot)
\end{align}
Now observe that:
\begin{align}\label{mathcal_F_bound_1}
  \vert\vert\mathcal{F}\vert\vert_{C^k} \leq  \vert\vert\mathcal{F} -\mathcal{F}_1\vert\vert_{C^k}   + \vert\vert\mathcal{F}_1 -\mathcal{F}_0\vert\vert_{C^k} + \vert\vert\mathcal{F}_0\vert\vert_{C^k}
\end{align}
Note that by \eqref{normal_form_bounds_1} and Lemma \ref{Reeb_dist_lemma}, $\vert\vert\mathcal{P}_u -Id\vert\vert_{C^k} \leq \sigma_k(\vert\vert\Omega- \Omega_0\vert\vert_{C^k})$; combining this observation with \eqref{bounds_of_omega_tilde_minus_omega} we see that:
\begin{align}\label{mathcal_F_bound_2}
    \vert\vert\mathcal{F} -\mathcal{F}_1\vert\vert_{C^k} \leq \omega_k(\vert\vert\Omega- \Omega_0\vert\vert_{C^{k}})  \\\nonumber \vert\vert\mathcal{F}_1 -\mathcal{F}_0\vert\vert_{C^k} \leq \omega_k(\vert\vert\Omega- \Omega_0\vert\vert_{C^k})
\end{align}
Plugging \eqref{mathcal_F_bound_2} in \eqref{mathcal_F_bound_1} and observing that $\vert\vert \mathcal{F}_0\vert\vert_{C^k} = 0$ we obtain:
\begin{align}
    \vert\vert\mathcal{F}\vert\vert_{C^k}  \leq \sigma_{k+1}(\vert\vert\Omega- \Omega_0\vert\vert_{C^{k+1}})
\end{align}
\section{Proof of Corollary \ref{variationalprinciple}}\label{proof_of_var_principle}
We begin by fixing an $S^1$-bundle $p:\Sigma \rightarrow M$ associated with the Zoll odd-symplectic form $\Omega_0$ and pick a connection $1$-form $\alpha_0$. Let $\mathcal{U} \subset \Xi^2_{C_0}(\Sigma)$ be the neighborhood of $\Omega_0$ provided by Theorem \ref{normalformthm}. Then, for any odd-symplectic form $\Omega \in \mathcal{U}$, there exists a diffeomorphoism $u:\Sigma \rightarrow \Sigma$ such that $u^*\Omega$ has the form \eqref{normal_form_odd_symp}. Denoting by $\tilde{\Omega}:= u^*\Omega$ and using the same notation as in the statement of Theorem \ref{normalformthm}, we see that \eqref{normal_form_odd_symp} implies the following chain of equalities:
\begin{align}\label{var_princip_1}
    \iota_{R_0} \Tilde{\Omega} = \iota_{R_0} d\eta - dS = \mathcal{F}[dS] - dS
\end{align}
where $R_0$ is the generating vector field for the $S^1$-action on $\Sigma$ that interests us and $S$ is the invariant function whose existence is guaranteed by Theorem \ref{normalformthm}. Denoting by $\bar{S}$ the quotient of $S$ by the $S^1$-action on $\Sigma$, we observe that by \eqref{var_princip_1}, for any $y \in \Sigma$ contained in an $S^1$-fiber over a critical point $x$ of $\Bar{S}$ the following holds:
\begin{align}
    \iota_{R_0} \Tilde{\Omega}|_{y} =0
\end{align}
So, any vector field that spans the characteristic distribution of $\Tilde{\Omega}$ at $y$ is a scalar multiple of $R_0$. As a consequence, the following is true: pick an arbitrary vector field $R$ spanning the characteristic distribution of $\tilde \Omega$ then, for any $y \in \text{Crit}(S)$ the trajectory of $R$ passing through $y$ is the image $u(p_y)$ of the $S^1$-fiber $p_y$ through $y$. From the definition of the quotient map $\bar S$, we can then conclude that any vector field spanning the characteristic distribution of $\tilde \Omega$ must have at least as many closed trajectories as the number of critical points of the function $\bar S$. Moreover, by the Lusternik–Schnirelmann theorem \cite[Theorem 11.2.9]{mcduffsalamon2017introduction}, the following chain of inequalities hold:
\begin{align}
    \text{Crit $\bar{S}$} \geq \text{cl}(M)
\end{align}
where $\text{cl}(M)$ is the cup-length of $M$ computed for the singular cohomology functor with real coefficients.\\
\\
To finish the proof, we only need to show that $\Omega$ is Zoll if and only if $S$ is constant. Toward that end, observe that if $S$ and therefore $\bar{S}$ is constant, then the characteristic foliation of $\Omega$ can be conjugated to that of $\Omega_0$ by the diffeomorphism $u$ and $\Omega$ is therefore Zoll.\\
\\
If, on the other hand, $\Omega$ is Zoll then by \cite[Proposition 7.4 (iii)]{B-K-Odd_Symp} we know that $\inf\mathcal{A}_{\Omega} = \sup\mathcal{A}_{\Omega}$. As a consequence, since function $S$ is defined in \eqref{def_of_S} below to be the following:
\begin{align*}
    S(x) = \mathcal{A}_{\Omega}(\bar{\Gamma}_x) 
\end{align*}
where $\Bar{\Gamma}: S^1 \times
 [0,1] \rightarrow \Sigma$ is just the constant homotopy $\bar{\Gamma}(\cdot, t) = p_x$ $\forall t \in [0,1]$, the inequalities in \eqref{var_principle_eqns} now follow.
\section{Proof of Theorem \ref{Volume_formula}}\label{vol_formula_section}
Recall the setup of Theorem \ref{Volume_formula}: $(\Sigma,\Omega_0)$ is a connected Zoll odd-symplectic manifold of dimension $2n-1$ and the cohomology class $C_0 \in H^2(\Sigma,\mathbb{R})$ represented by $\Omega_0$ satisfies the condition: $C_0^{n-1} =0$. The volume functional (c.f \eqref{Vol_function_def_eqn}) was defined as:
\begin{align}
  Vol_{\Omega_0}(\Omega) = \int^1_0 \int_{\Sigma} \alpha \wedge (\Omega_0+rd\alpha)^{n-1} \;dr
\end{align}
where $\alpha \in \Omega^1(\Sigma)$ is such that $\Omega = \Omega_0+d\alpha$.\\
\\
As always, denote the $S^1$-bundle associated with $\Omega_0$ by $\Sigma \xrightarrow{p} M$ and let $\alpha_0$ be a connection $1$-form of this $S^1$-bundle that is, the $1$-form satisfying: $\alpha_0 (R_0) =1$ and $\mathcal{L}_{R_0}\alpha_0 =0$ where $R_0$ is the vector field generating the relevant $S^1$-action on $\Sigma$. Then, by Theorem \ref{normalformthm} there exists a $C^2$-neighbourhood $\mathcal{U} \subset \Xi^2_{C_0}(\Sigma)$ of $\Omega_0$ such that for any odd-symplectic form $\Omega \in \mathcal{U}$ there exists a diffeomrophism $u$ of $\Sigma$ isotopic to the identity and a smooth function $\Tilde{S}: \Sigma \rightarrow \mathbb{R}$ such that:
\begin{align}
    \Tilde{\Omega}:= u^*\Omega = \Omega_0 + d(\Tilde{S}\alpha_0 + \Tilde{\eta}) 
\end{align}
where the $1$-form $\Tilde{\eta}$ is as in the statement of Theorem \ref{normalformthm}. Since $u$ is isotopic to the identity, it follows from \cite[Proposition 2.8]{B-K-Odd_Symp} that $Vol_{\Omega_0}(\Omega) = Vol_{\Omega_0}(\Tilde{\Omega})$. Therefore,
\begin{align}\label{raw_vol_integral_vol_formula_compuattion}
Vol_{\Omega_0}(\Omega) =    \int^1_0 \int_{\Sigma} (\tilde{S}&\alpha_0 + \tilde{\eta} ) \wedge (\Omega_0 + r d(\tilde{S}\alpha_0 + \tilde{\eta}))^{n-1} dr \\\nonumber = &\int^1_0 \frac{1}{r} \int_{\Sigma}  (S\alpha_0 +\eta) \wedge (\Omega_0 + d(S\alpha_0 +\eta))^{n-1} \;\; dr
\end{align}
where $S:= r \Tilde{S}$ and $\eta:= r \Tilde{\eta}$.\\
\\
We now consider the following two summands in \eqref{raw_vol_integral_vol_formula_compuattion}:
\begin{enumerate}
    \item $S\alpha_0 \wedge (\Omega_0 +  d(S\alpha_0 + \eta))^{n-1}$
    \item $\eta  \wedge (\Omega_0 + d(S\alpha_0 + \eta))^{n-1}$
\end{enumerate}
\textbf{\underline{A formula for $\int_{\Sigma} S\alpha_0 \wedge (\Omega_0 +  d(S\alpha_0 + \eta))^{n-1}$}}\\
\\
\textbf{\underline{Step 1:}}
We now claim that the following expansion follows from the binomial theorem:
 \begin{align}\label{formula_for_omega_to_the_n}
    (\Omega_0 + d(S\alpha_0 + \eta ))^{n-1} = (\Omega_0 + d\eta +Sd\alpha_0)^{n-1} + (n-1)(\Omega_0 + d\eta +Sd\alpha_0)^{n-2} \wedge (dS \wedge \alpha_0)
\end{align}
To see that this is true, we first write out the binomial expansion of $(\Omega_0 + d(S\alpha_0 + \eta ))^{n-1}$:
\begin{align}\label{expansion_of_omega_to_the_n}
     (\Omega_0 + d(S\alpha_0 + \eta ))^{n-1} = \sum^{n-1}_{k=0} \binom{n-1}{k} (\Omega_0 + d\eta +Sd\alpha_0)^{n-1-k} \wedge (dS \wedge \alpha_0)^k
\end{align}
it is now easy to see that for $k >1$ the product $(dS \wedge \alpha_0)^k =0$ by skew-symmetry of the wedge product. So, only the first two summands in \eqref{expansion_of_omega_to_the_n} remain, thereby justifying the validity of \eqref{formula_for_omega_to_the_n}. 
Since $S\alpha_0 \wedge (n-1)(\Omega_0 + d\eta +Sd\alpha_0)^{n-2} \wedge (dS \wedge \alpha_0) =0$ by skew-symmetry of the wedge product. We therefore obtain the following:
\begin{align}\label{alpha_wedge_1}
 S\alpha_0 &\wedge (\Omega_0 +  d(S\alpha_0 + \eta))^{n-1}   = S\alpha_0 \wedge (\Omega_0 +Sd\alpha_0)^{n-1} \\\nonumber + &S\Bigg[ \sum_{k=1}^{n-1} \binom{n-1}{k} \alpha_0 \wedge (\Omega_0 +Sd\alpha_0)^{n-1-k} \wedge d\eta^k \Bigg]
\end{align}
\\
\\
We now focus just on the second summand in \eqref{alpha_wedge_1} and write:
\begin{align}\label{def_beta_k}
    \beta_k := \alpha_0 \wedge (\Omega_0 +Sd\alpha_0)^{n-1-k} \wedge d\eta^k  \;\;\; for \;\; 1\leq k \leq n-1 
\end{align}
Using the binomial theorem to expand $(\Omega_0 +Sd\alpha_0)^{n-1-k}$ in $\beta_k$ we get:
\begin{align}\label{beta_k_def}
     \beta_k =  \sum^{n-1-k}_{j=0} \binom{n-1-k}{j} S^j \alpha_0 \wedge d\alpha^j_0 \wedge (\Omega_0)^{n-1-k-j} \wedge d\eta^k 
\end{align}
By definition, the exterior derivative satisfies:
\begin{align}
    &S^j \cdot d\big(\alpha_0 \wedge d\alpha^j_0 \wedge (\Omega_0)^{n-1-k-j}\wedge \eta \wedge d\eta^{k-1}\big) = \\\nonumber S^j \big[ d\alpha^{j+1}_0  &\wedge (\Omega_0)^{n-1-k-j} \wedge \eta \wedge d\eta^{k-1}\big] - S^j \alpha_0 \wedge d\alpha^j_0 \wedge (\Omega_0)^{n-1-k-j} \wedge d\eta^k 
\end{align}
The term in \eqref{beta_k_def}is:
\begin{align}\label{alpha_wedge_2}
 &\sum^{n-1-k}_{j=0}   \binom{n-1-k}{j}  S^j \big[ d\alpha^{j+1}_0  \wedge (\Omega_0)^{n-1-k-j} \wedge \eta \wedge d\eta^{k-1}\big] \\\nonumber -    \sum^{n-1-k}_{j=0} & \binom{n-1-k}{j}  S^j \cdot d\big(\alpha_0 \wedge d\alpha^j_0 \wedge (\Omega_0)^{n-1-k-j}\wedge \eta \wedge d\eta^{k-1}\big)  
\end{align}
Plugging \eqref{alpha_wedge_2} into \eqref{alpha_wedge_1} we obtain:
\begin{align}\label{alpha_wedge_3}
  &S\alpha_0 \wedge (\Omega_0 +  d(S\alpha_0 + \eta))^{n-1} =   S\alpha_0 \wedge (\Omega_0 +Sd\alpha_0)^{n-1} \\\nonumber + S \Bigg[ \sum_{k=1}^{n-1} \binom{n-1}{k}  &(\Omega_0 + Sd\alpha_0)^{n-1-k} \wedge \eta \wedge d\eta^{k-1} \wedge d\alpha_0  - \sum^{n-1}_{k=1}  \sum^{n-1-k}_{j=0} S^j \cdot \tilde{p}^k_j \alpha_0 \wedge (\Omega_0)^{n-1} \Bigg]
\end{align}
where for each $k,j$ satisfying $1\leq k \leq n$ and $0 \leq j \leq n-k$, the functions $\Tilde{p}^k_j: \Sigma \rightarrow \mathbb{R}$ are defined by:
\begin{align}\label{def_of_p_tilde}
    \tilde{p}^k_j \alpha_0 \wedge (\Omega_0)^{n-1} = \binom{n-1}{k} \binom{n-1-k}{j} d\big(\alpha_0 \wedge d\alpha^j_0 \wedge (\Omega_0)^{n-1-k-j}\wedge \eta \wedge d\eta^{k-1}\big)
\end{align}
\textbf{\underline{Step 2:}} We now focus on the first part of the second summand in \eqref{alpha_wedge_3} and write it as $\sigma_{k,j}$:
\begin{align}
    \sigma_{k,j} =  \sum^{n-1-k}_{j=0} \binom{n-1-k}{j} S^{j+1} \big[ d\alpha^{j+1}_0  \wedge (\Omega_0)^{n-1-k-j} \wedge \eta \wedge d\eta^{k-1}\big] 
\end{align}
Since $\alpha_0$ is the connection $1$-form of the $S^1$-bundle $\Sigma \xrightarrow{p} M$ associated with $\Omega_0$ and $R_0$ is the vector field generating the $S^1$-action, it is clear that the operator $[\alpha_0 \wedge \iota_{R_0}] : \Omega^{2n-1}(\Sigma) \rightarrow   \Omega^{2n-1}(\Sigma)$ acts as the identity. Therefore, on applying the operator $\alpha_0 \wedge \iota_{R_0}$ to $\sigma_{k,j}$ we obtain:
\begin{align}
 \sigma_{k,j} =  \sum^{n-1-k}_{j=0} \binom{n-1-k}{j} S^{j+1} [\alpha_0 \wedge \iota_{R_0}]  ( d\alpha^{j+1}_0 \wedge \Omega^{n-1-k-j}_0 \wedge \eta \wedge & d\eta^{k-1}) \\\nonumber = \sum^{n-1-k}_{j=0} \binom{n-1-k}{j}  S^{j+1}[\alpha_0 \wedge  d\alpha^{j+1}_0 \wedge \Omega^{n-1-k-j}_0 \wedge \eta \wedge (k-1) \mathcal{F}(dS&) \wedge d\eta^{k-2}]  
\end{align}
where the last equality follows from the fact that $\iota_{R_0} d\alpha_0 = \iota_{R_0} \eta = \iota_{R_0} \Omega_0 =0$ and $\iota_{R_0} d\eta = \mathcal{F}(dS)$ (c.f Theorem \ref{normalformthm})
So, we can write $\sigma_{k,j}$ as:
\begin{align}\label{sigma_kj_1}
\sigma_{k,j} = \sum^{n-1-k}_{j=0} \binom{n-1-k}{j} (k-1) S^{j+1} \mathcal{F}(dS) \wedge [\alpha_0 \wedge  d\alpha^{j+1}_0 \wedge \Omega^{n-1-k-j}_0 \wedge \eta \wedge d\eta^{k-2}]
\end{align}
Before going further, we recall the following observation extracted from \cite{Abbondandolo-Benedetti}: the wedge product induces an invertible map $L: \Omega^{2n-2}(\Sigma) \times \Omega^1(\Sigma) \rightarrow \Omega^{2n-1}(\Sigma)$ in the following way:
        \begin{align}
            (\omega_1,\omega_2) \mapsto \omega_1 \wedge \omega_2 \text{  }\;\;\forall\;\;  (\omega_1,\omega_2) \in \Omega^{2n-2}(\Sigma) \times \Omega^1(\Sigma)
        \end{align}
        So, the endomorphism $\mathcal{F}:\Omega^1(\Sigma) \rightarrow \Omega^1(\Sigma)$ in Theorem \ref{normalformthm} has a dual $\hat{\mathcal{F}}: \Omega^{2n-1}(\Sigma) \rightarrow \Omega^{2n-1}(\Sigma)$ defined in the following way:
        \begin{align}\label{dual_of_endo_F}
            \mathcal{F}(\omega_1)\wedge \omega_2 = \omega_1 \wedge \hat{\mathcal{F}}(\omega_2) \text{ }\;\;\forall  (\omega_1,\omega_2) \in \Omega^{2n-2}(\Sigma) \times \Omega^1(\Sigma) 
        \end{align}
Combining this observation with \eqref{sigma_kj_1} we obtain:      \begin{align}
 \sigma_{k,j} =  \sum^{n-1-k}_{j=0} \binom{n-1-k}{j}(k-1)\frac{1}{j+2} dS^{j+2} \wedge \hat{\mathcal{F}}[\alpha_0 \wedge  d\alpha^j_0 \wedge \Omega^{n-1-k-j}_0 \wedge \eta \wedge d\eta^{n-1-k-2}]
\end{align}
where $\hat{\mathcal{F}}$ is the dual to $\mathcal{F}$ defined in \eqref{dual_of_endo_F}.\\
\\
After applying Stokes' theorem, for each fixed $k$ we can now write the integral $ \int_{\Sigma} \sigma_{k,j}$ as:
\begin{align}\label{sigma_k_final_form}
 \int_{\Sigma} \sigma_{k,j} =  \sum^{n-1-k}_{j=0} \binom{n-1-k}{j}(k-1) \frac{1}{j+2} \int_{\Sigma} S^{j+2} d\hat{\mathcal{F}}(\alpha_0 \wedge d\alpha^{j+1}_0 \wedge \Omega^{n-1-k-j}_0 \wedge \eta \wedge d\eta^{k-2})
\end{align}
\textbf{\underline{Step 3:}}\\
\\
Finally, because $\sigma_{k,j}$ is a top-rank form on $\Sigma$, for each fixed $k,j$
there exist smooth functions $p^k_j: \Sigma \rightarrow \mathbb{R}$ given by:
\begin{align}\label{sigma_k_polynomial}
    p^k_j \cdot \alpha_0 \wedge (\Omega_0)^{n-1} = \binom{n-1}{k} \binom{n-1-k}{j}\frac{1}{j+2}(k-1) 
 d\hat{\mathcal{F}}(\alpha_0 \wedge d\alpha^{j+1}_0 \wedge \Omega^{n-1-k-j}_0 \wedge \eta \wedge d\eta^{k-2}) 
\end{align}
\\
Plugging \eqref{sigma_k_polynomial} into \eqref{alpha_wedge_3} we get:
\begin{align}\label{part_2_final_form}
  &\int_{\Sigma}S\alpha_0 \wedge (\Omega_0 + Sd\alpha_0 + d\eta)^{n-1} \\\nonumber  = \int_{\Sigma} S\alpha_0 \wedge (\Omega_0 +Sd\alpha_0)^{n-1}  + &\int_{\Sigma} \Big[\sum^{n-1}_{k=1} \sum^{n-1-k}_{j=0}  S^{j+2} \cdot p^k_j + \sum^{n-1}_{k=1} \sum^{n-1-k}_{j=0} S^{j+1} \cdot \tilde{p}^k_j\Big] \alpha_0 \wedge (\Omega_0)^{n-1} \\\nonumber = \int_{\Sigma} S\alpha_0 &\wedge (\Omega_0 +Sd\alpha_0)^{n-1}  + \int_{\Sigma} \tilde{P}(x,S) \; \alpha_0 \wedge (\Omega_0)^{n-1}
\end{align}
where $\tilde{P}(x,s):= \sum_{k=0}^{n-1} \sum_{j=0}^{n-1-k} \tilde{P}_{k,j}(x,s)$ and the functions $\tilde{P}_{k,j}(x,s)$ are given by
\begin{align}\label{def_of_Tilde_P}
   \tilde{P}_{k,j}(x,s):= \sum^{n-1}_{k=1} \sum^{n-1-k}_{j=0} s^{j+2} \cdot p^k_j + \sum^{n-1}_{k=1} \sum^{n-1-k}_{j=0} s^{j+1} \cdot \tilde{p}^k_j
\end{align}
\textbf{\underline{Part 3: A formula for $\int_{\Sigma} \eta  \wedge (\Omega_0 + d(S\alpha_0 + \eta ))^{n-1}$}}\\
\\
\textbf{\underline{Step 1:}}\\
First note that by \eqref{formula_for_omega_to_the_n}, we can write $\eta  \wedge (\Omega_0 + d(S\alpha_0 + \eta))^{n-1}$ as:
\begin{align}\label{eta_wedge_1}
\eta \wedge (\Omega_0 +Sd\alpha_0 + d\eta)^{n-1} + n-1 [\eta \wedge (\Omega_0 +Sd\alpha_0 + d\eta)^{n-2} \wedge (dS \wedge \alpha_0)]
\end{align}
Looking just at the second summand above, see that we can write it as:
\begin{align}\label{part3_formula_1}
    &(n-1)dS \wedge \alpha_0 \wedge \eta  \wedge (\Omega_0 +Sd\alpha_0 + d\eta)^{n-2} \nonumber\\
= \sum^{n-2}_{k=0} \sum^{n-k-2}_{j=0} &\binom{n-2}{k}  \binom{n-k-2}{j} \frac{n-1}{j+1} dS^{j+1} \wedge \alpha_0 \wedge \eta  \wedge (\Omega_0)^{n-1-k-j} \wedge (d\alpha_0)^j \wedge  (d\eta)^{k}
\end{align}
After applying Stokes' theorem, the integral of \eqref{part3_formula_1} above can be written as:
\begin{align}\label{part_3_step_1}
     - \sum^{n-2}_{k=0} \sum^{n-1-k}_{j=0} \binom{n-2}{k}  \binom{n-1-k}{j} \frac{n-1}{j+1} \int_{\Sigma} S^{j+1} d(\alpha_0 \wedge \eta  \wedge (\Omega_0)^{n-1-k-j} \wedge (d\alpha_0)^j \wedge  (d\eta)^{k})
\end{align}
For each fixed $k$ and $j$ we can write the top-rank form in \eqref{part_3_step_1} as:
\begin{align}\label{part_3_final_1}
\tilde{Q}^{k,j}_2 \cdot \alpha_0 \wedge (\Omega_0)^{n-1} =  \binom{n-1}{k}  \binom{n-1-k}{j} \frac{-(n-1)}{j+1} d(\alpha_0 \wedge \eta  \wedge (\Omega_0)^{n-1-k-j} \wedge (d\alpha_0)^j \wedge  (d\eta)^{k})
\end{align}
where $\Tilde{Q}^{k,j}_2: \Sigma \rightarrow \mathbb{R}$ are  smooth functions for each $k$ and $j$.\\
\textbf{\underline{Step 2:}}\\
Applying the operator $[\alpha_0 \wedge \iota_{R_0}]$ to the first summand $\eta  \wedge (\Omega_0 +Sd\alpha_0 + d\eta)^{n-1}$ in \eqref{eta_wedge_1} we get:
\begin{align}
\sum_{k=1}^{n-1} \sum_{j=0}^{n-1-k} \binom{n-1}{k}   \binom{n-1-k}{j} k S^j \mathcal{F}[dS] \wedge \alpha_0 \wedge \eta \wedge (\Omega_0)^{n-1-k-j} \wedge (d\alpha_0)^{j} \wedge d\eta^{k-1}
\end{align}
On applying Stokes' theorem, the integral of the above differential form can be written as the integral of the following differential form:
\begin{align}\label{part_3_step_2_eq_2}
   \int_{\Sigma} \sum_{k=1}^{n-1} \sum_{j=0}^{n-1-k} \binom{n-1}{k}  \binom{n-1-k}{j} \frac{-k}{j+1} S^{j+1} \cdot d\hat{\mathcal{F}}[\alpha_0 \wedge \eta \wedge (\Omega_0)^{n-1-k-j} \wedge (d\alpha_0)^{j} \wedge d\eta^{k-1}]
\end{align}
For each fixed $k,j$ we define smooth functions $\tilde{Q}^{k,j}_1: \Sigma \rightarrow \mathbb{R}$ satisfying:
\begin{align}\label{part_3_final_2}
 \tilde{Q}^{k,j}_1 \cdot \alpha_0 \wedge (\Omega_0)^{n-1} = \binom{n-1}{k}  \binom{n-1-k}{j} \frac{-k}{j+1} \cdot d\hat{\mathcal{F}}[\alpha_0 \wedge \eta \wedge (\Omega_0)^{n-1-k-j} \wedge (d\alpha_0)^{j} \wedge d\eta^{k-1}] \\\nonumber  
\end{align}
$\hat{\mathcal{F}}$ is again the dual to $\mathcal{F}$ defined in \eqref{dual_of_endo_F}.\\
\\
Plugging \eqref{part_3_final_1} and \eqref{part_3_final_2} into \eqref{eta_wedge_1} we see that the integral $\int_{\Sigma} \eta \wedge (\Omega_0 + d(S\alpha_0+\eta))^{n-1}$ can be written as:
\begin{align}\label{part_3_final_final}
   \int_{\Sigma}\eta \wedge (\Omega_0 + d(S\alpha_0+\eta))^{n-1} = \int_{\Sigma} Q_1(x,s) \; \alpha_0 \wedge \Omega_0^{n-1} + \int_{\Sigma} Q_2(x,s) \; \alpha_0 \wedge \Omega_0^{n-1} 
 \end{align}  
 where 
\begin{align}\label{def_of_Q_1_and_Q_2}
 Q_1: &\Sigma \times \mathbb{R} \rightarrow \mathbb{R} \\\nonumber 
 &(x,s) \mapsto \sum_{k=1}^{n-1} \sum_{j=0}^{n-1-k} s^{j+1} \cdot \tilde{Q}^{k,j}_1(x) \\\nonumber 
 Q_2: &\Sigma \times \mathbb{R} \rightarrow \mathbb{R} \\\nonumber 
 &(x,s) \mapsto \sum_{k=0}^{n-1}   \sum^{n-1-k}_{j=0} s^{j+1} \cdot \tilde{Q}^{k,j}_2(x) 
\end{align}
\textbf{\underline{Part 4: Putting the previous steps together}}\\
\\
By \eqref{raw_vol_integral_vol_formula_compuattion} $Vol_{\Omega_0}(\Omega)$ can be written as:
\begin{align}
    Vol_{\Omega_0}(\Omega) =    \int^1_0 \frac{1}{r} \int_{\Sigma} (S\alpha_0 + \eta) \wedge (\Omega_0 + r d(S\alpha_0 + \eta ))^{n-1} dr 
\end{align}
By \eqref{def_of_Tilde_P}, \eqref{def_of_Q_1_and_Q_2} we see that:
\begin{align}
    Vol_{\Omega_0}(\Omega) = \int^1_0 \frac{1}{r}\int_{\Sigma}  S\alpha_0 &\wedge (\Omega_0 +Sd\alpha_0)^{n-1} dr + \int^1_0 \frac{1}{r} \int_{\Sigma} D(x,S) \; \alpha_0 \wedge \Omega^{n-1}_0 dr
\end{align}
where:
\begin{align}\label{def_of_D}
    D:= \Tilde{P}+Q_1+Q_2
\end{align}
We now note the following:
\begin{enumerate}
    \item For each $k$ and $j$ the following is true just from the definition of the functions $\Tilde{Q}^{k,j}_1$ (c.f \eqref{def_of_Q_1_and_Q_2}) :
    \begin{align}
   &Q_1 = \sum_{k=1}^{n-1} \sum_{j=0}^{n-1-k} S^{j+1} \cdot  \tilde{Q}^{k,j}_1 \cdot \alpha_0 \wedge (\Omega_0)^{n-1}  \\\nonumber =
 \sum_{k=0}^{n-1}   \sum^{n-1-k}_{j=0} \binom{n-1}{k} &\binom{n-1-k}{j} (j+1) S^{j+1} d\hat{\mathcal{F}}[\alpha_0 \wedge \eta \wedge (\Omega_0)^{n-1-k-j} \wedge (d\alpha_0)^{j} \wedge d\eta^{k-1}]
    \end{align}
    It follows that the functions $\tilde{Q}_1$ integrates to zero over $\Sigma$ when $S$ is constant.\\
    
    \item Similarly for the functions $ \Tilde{Q}^{k,j}_2$ we see that for each $k$ and $j$ the following equality follows from their definition (c.f \eqref{def_of_Q_1_and_Q_2}) :
    \begin{align}
    &Q_2 = \sum_{k=0}^{n-1}   \sum^{n-1-k}_{j=0}   S^{j+1} \cdot \tilde{Q}^{k,j}_2 \cdot \alpha_0 \wedge (\Omega_0)^{n-1}  \\\nonumber =
 \sum_{k=0}^{n-1}   \sum^{n-1-k}_{j=0} \binom{n-1}{k} &\binom{n-1-k}{j} S^{j+1} \cdot d( \alpha_0 \wedge \eta  \wedge (\Omega_0)^{n-1-k-j} \wedge (d\alpha_0)^j \wedge  (d\eta)^{k})
    \end{align}
   The function $\Tilde{Q}_2$ integrates to $0$ over $\Sigma$ when $S$ is constant because the left-hand side of the above equation is exact.
    \item for each $k,j$ the functions $\Tilde{P}_{k,j}(x,S)$ integrates to zero over $\Sigma$ when $S$ is constant because they are defined by:
\begin{align*}
    \tilde{P}_{k,j}(x,S):= \sum^{n-1}_{k=1} \sum^{n-1-k}_{j=0} S^{j+2} \cdot p^k_j + \sum^{n-1}_{k=1} \sum^{n-1-k}_{j=0} S^{j+1} \cdot \tilde{p}^k_j
\end{align*}
    Where for each fixed $k$ and $j$ the functions $p^k_j(x)$ satisfy:
    \begin{align}
    p^k_j \cdot \alpha_0 \wedge (\Omega_0)^{n-1} = d\hat{\mathcal{F}}(\alpha_0 \wedge d\alpha^{j+1}_0 \wedge \Omega^{n-1-k-j}_0 \wedge \eta \wedge d\eta^{k-2})
\end{align}
    And the functions $\tilde{p}^k_j$ satisfy:
    \begin{align}
    \tilde{p}^k_j \cdot \alpha_0 \wedge (\Omega_0)^{n-1} = d\big(\alpha_0 \wedge d\alpha^j_0 \wedge (\Omega_0)^{n-1-k}\wedge \eta \wedge d\eta^{k-1}\big)
\end{align}
\item It also follows from the definition of $\tilde{P}$, $Q_1$ and $Q_2$ that $D(x,0) = 0$ $\forall$ $x \in \Sigma$.
\end{enumerate}
To finish the proof of Theorem \ref{Volume_formula} we need only to establish \eqref{volume_formula_bounds}. This is the context of the following result.
\begin{proposition}
Let $D: \Sigma \times\mathbb{R} \rightarrow \mathbb{R}$ be the function defined in \eqref{def_of_D} then, given an $\epsilon>0$ there exists a $\delta_0 >0$ such that the following is true when $s \in \mathbb{R}$ is close to $0$:
\begin{align}
 \vert\vert \Omega_0 - \Omega\vert\vert_{C^2} < \delta_0 \implies   \Big\vert\Big\vert\frac{\partial}{\partial s}D\Big\vert\Big\vert_{C^0} < \epsilon
\end{align}
\end{proposition}
\begin{proof}
\textbf{\underline{Step 1: The bounds on $p^k_j$}}\\
We need to show that given an $\epsilon>0$, there exist real constants $\delta, c>0$ such that $\vert\vert p^k_j\vert\vert_{C^0} < \epsilon$ if the following hold:
\begin{align}\label{est_bounds_basis_1}
  \max \bigg\{\vert\vert\eta\vert\vert_{C^0}, \vert\vert d\eta\vert\vert_{C^0}, \vert\vert\mathcal{F}\vert\vert_{C^0} \bigg\} < \delta \\\nonumber
\max \bigg\{\vert\vert\eta\vert\vert_{C^1}, \vert\vert d\eta\vert\vert_{C^1}, \vert\vert\mathcal{F}\vert\vert_{C^1} \bigg\} < c
  \end{align}
Since $\hat{\mathcal{F}}$ is the dual of $\mathcal{F}$, the bounds in \eqref{est_bounds_basis_1} imply that there exists a real constant $b_0 >0$ such that:
\begin{align}\label{F_dual_bounds}
    \vert\vert\hat{\mathcal{F}}\vert\vert_{C^0} < b_0 \delta \; \; \& \;\; \vert\vert\hat{\mathcal{F}}\vert\vert_{C^1} <  b_0 c
\end{align}
Now, recall that for each $1 \leq k \leq n-2$ and $0\leq j \leq n-1-k$ the functions $p^k_j$ satisfy:
\begin{align*}
     p^k_j \cdot \alpha_0 \wedge (\Omega_0)^{n-1} = \binom{n-2}{k} \binom{n-1-k}{j}  \Bigg[ d\hat{\mathcal{F}}(\alpha_0 \wedge d\alpha^{j+1}_0 \wedge \Omega^{n-1-k-j}_0 \wedge \eta \wedge d\eta^{k-2})\Bigg]
\end{align*}
Consider now the following differential form:
\begin{align}\label{def_of_gamma_k_j}
    \gamma_{k,j}:= \alpha_0 \wedge \eta \wedge d\alpha^j_0 \wedge \Omega^{n-1-k-j}_0 \wedge d\eta^{k-2}
\end{align}
The Libnitz formula along with \eqref{est_bounds_basis_1} implies that there exists a real number $b_1>0$ such that:
\begin{align}\label{gamma_c_0_bound}
    \vert\vert\gamma_{k,j}\vert\vert_{C^0} \leq b_1 \delta^{k-1} \; \;  \& \;\; \vert\vert\gamma_{k,j}\vert\vert_{C^1} \leq b_1 c^{k-1}
\end{align} 
Using the Leibnitz formula, \eqref{F_dual_bounds} and \eqref{gamma_c_0_bound} we see that the following is true:
\begin{align}\label{p_kj_bound_1}
    &\vert\vert d\hat{\mathcal{F}}[\alpha_0 \wedge \eta \wedge d\alpha^j_0 \wedge \Omega^{n-1-k-j}_0 \wedge d\eta^{k-2}]\vert\vert_{C^0} \\\nonumber \leq
     \vert\vert\hat{\mathcal{F}}&\vert\vert_{C^1} \cdot \vert\vert\gamma_{k,j}\vert\vert_{C^0} + \vert\vert\hat{\mathcal{F}}\vert\vert_{C^0} \vert\vert\gamma_{k,j}\vert\vert_{C^1} \leq b_0b_1 \delta^k + b_0b_1 c^k
\end{align} 
It is clear that by choosing $\delta, c$ in \eqref{est_bounds_basis_1} small enough we can make $\vert\vert p^k_{j}\vert\vert_{C^0} < \epsilon$ for any given $\epsilon >0$.\\
\\
\textbf{\underline{Step 2: The bounds on $\Tilde{p}^{k}_{j}$}}
\\
Recall from \eqref{def_of_p_tilde} that the functions $\Tilde{p}^k_j$ are defined by:
\begin{align*}
     \tilde{p}^k_j \alpha_0 \wedge (\Omega_0)^{n-1} = d\big(\alpha_0 \wedge d\alpha^j_0 \wedge (\Omega_0)^{n-1-k-j}\wedge \eta \wedge d\eta^{k-1}\big)
\end{align*}
In analogy with the previous step we make the following definition:
\begin{align}
    \hat{\gamma}_{k,j}:= \alpha_0 \wedge d\alpha^j_0 \wedge (\Omega_0)^{n-1-k-j}\wedge \eta \wedge d\eta^{k-1}
\end{align}
Now, it is clear from \eqref{est_bounds_basis_1} that there exists a real $d_1 >0$ such that:
\begin{align}
    \vert\vert\hat{\gamma}_{k,j}\vert\vert_{C^1} < d_1c^k
\end{align}
So, we now conclude that the following holds for each $1 \leq k \leq n-2$ and $0\leq j \leq n-1-k$:
\begin{align}
\vert\vert \tilde{p}^k_j\vert\vert_{C^0} \leq  \vert\vert\hat{\gamma}_{k,j}\vert\vert_{C^1} <  d_1c^{k}
\end{align}
\\
\textbf{\underline{Step 3: The bounds on $\tilde{Q}^{k,j}_1$}}
\\
First recall that the functions $\tilde{Q}^{k,j}_1$ are defined by the equations:
\begin{align}
         \tilde{Q}^{k,j}_1 \cdot \alpha_0 \wedge (\Omega_0)^{n-1} = d\hat{\mathcal{F}}[\alpha_0 \wedge \eta \wedge (\Omega_0)^{n-1-k-j} \wedge (d\alpha_0)^{j} \wedge d\eta^{k-1}]
\end{align}
So, if we write $\kappa_{k,j}:= \alpha_0 \wedge \eta \wedge (\Omega_0)^{n-1-k-j} \wedge (d\alpha_0)^{j} \wedge d\eta^{k-1}$. Then, by \eqref{est_bounds_basis_1} there exists a real constant $d_2>0$ such that the following hold:
\begin{align}
    \vert\vert\kappa_{k,j}\vert\vert_{C^0} \leq d_2 \delta^{k} \; \;  \& \;\; \vert\vert\kappa_{k,j}\vert\vert_{C^1} \leq d_2 c^{k}
\end{align}
This allows us to conclude as in step 1 above that there exists a $d_3>0$ such that:
\begin{align}
    \vert\vert\tilde{Q}^{k,j}_1\vert\vert_{C^0} \leq \vert\vert\hat{\mathcal{F}}\vert\vert_{C^0}\cdot \vert\vert\kappa_{k,j}\vert\vert_{C^1} + \vert\vert\hat{\mathcal{F}}\vert\vert_{C^1}\cdot \vert\vert\kappa_{k,j}\vert\vert_{C^0} \leq d_3(\delta^{k+1} + c^{k+1})
\end{align}
\\
\textbf{\underline{Step 4: The bounds on $\Tilde{Q}^{k,j}_2$}}
\\
\\
First, recall that for each $k$ and $j$ the functions $\Tilde{Q}^{k,j}_2$ are defined by the equations:
\begin{align}
        \tilde{Q}^{k,j}_2 \cdot \alpha_0 \wedge (\Omega_0)^{n-1} = d( \alpha_0 \wedge \eta  \wedge (\Omega_0)^{n-1-k-j} \wedge (d\alpha_0)^j \wedge  (d\eta)^{k})
    \end{align}

\begin{align}
    \vert\vert\tilde{Q}^{k,j}_2\vert\vert_{C^0} \leq a_1c^{k+1}
\end{align}
So, by choosing $c$ in \eqref{est_bounds_basis_1} small enough we can guarantee that for a given $\epsilon>0$ and for each $k$ and $j$ the functions $\Tilde{Q}^{k,j}_2$ satisfy the bounds:
\begin{align}
    \vert\vert\tilde{Q}^{k,j}_2\vert\vert_{C^0} < \epsilon
\end{align}
Since $D(x,s)$ is a polynomial in $s$ with coefficients given by $\tilde{P}$, $Q_1$ and $Q_2$ we see that if $\delta$ and $c$ in \eqref{est_bounds_basis_1} are small enough then $\vert\vert\frac{\partial D}{\partial s}\vert\vert_{C^0} < \epsilon$ as claimed. 
\end{proof}
\section{Hamiltonian dynamics on hypersurfaces}\label{dynamics_on_hyp_sec_main}
\subsection{Proof of Theorem \ref{systolic_ineq_for_hyp_sur}}\label{proof_of_hyp_sys_ineq_section}
Given a symplectic manifold $(W,\omega)$, the $2n$-form $\omega^{2n}$ defines an orientation on $W$, which in turn induces a natural orientation on an embedded hypersurface $\Sigma \hookrightarrow W$. The characteristic foliation of the odd-symplectic form $\omega\vert_{\Sigma}$ will be called the characteristic foliation of $\Sigma$. In this setting, we suppose that the symplectic manifold $(W,\omega)$ admits a Zoll hypersurface $\Sigma_0 \hookrightarrow W$, to this Zoll hypersurface, one can associate a smooth function $H \in C^{\infty}(A)$ where $A \subset W$ is a neighborhood of $\Sigma_0$ such that the following holds:
\begin{align}
    H^{-1}(0)= \Sigma_0
\end{align}
Given another embedded hypersurface $\tilde \Sigma \hookrightarrow W$, there similarly exists a smooth function $\tilde H \in C^{\infty}(\tilde A)$ where $\tilde A \subset W$ is a neighborhood of $\tilde \Sigma$. Recall that in the introduction, we defined $\tilde \Sigma$ as being $C^2$-close to $\Sigma$ if the two functions $H$ and $\tilde H$ are $C^2$-close; this definition is motivated by the observation \eqref{eta_taylor} below.\\
\\
We now fix a connected symplectic manifold $(W,\omega)$ of dimension $2n$ for $n\geq 2$ admitting a Zoll hypersurface $\Sigma_0 \hookrightarrow W$ and pick another closed, connected and embedded hypersurface $\Sigma \hookrightarrow W$ that is $C^2$-close to $\Sigma_0$. In this setting, there exists a smooth function $H_0 \in C^{\infty}(A)$ defined in a neighborhood $A \subset W$ of the union $\Sigma \cup \Sigma_0$ such that for some $a \in \mathbb{R}$ the following hold:
\begin{align}
    H^{-1}_0(0) = \Sigma_0 \qquad H^{-1}(a) = \Sigma
\end{align}
Since $\Sigma$ is assumed to be $C^2$-close to $\Sigma_0$, we can assume without any loss of generality that $H_0$ was chosen so that $a \in (-\epsilon_0,\epsilon_0)$ where $\epsilon_0>0$ was chosen so that the interval $(-\epsilon_0,\epsilon_0)$ contains no critical values of $H_0$. In this setting, we make the following definition:
\begin{align}\label{D(Sigma)_region_def}
  \tilde D(\Sigma):= \{x \in W\;\big|\; H(x)\in[0,a] \}
\end{align} 
For $ t \in (-\epsilon_0,\epsilon_0)$ we will denote by $\varphi^{t}$ the flow of the following vector field
\begin{align}
    X:= \frac{\nabla H}{\vert\vert \nabla H\vert\vert^2}
\end{align}
In particular, $H(\varphi^t)=t$ for $t \in (-\epsilon_0,\epsilon_0)$. This allows us to define  the following diffeomorphism:
\begin{align}\label{hyp_def_of_Phi}
   \Phi:  \Sigma_0 \times  [0,a] \longrightarrow \tilde D(\Sigma)\\\nonumber
    (x,t)\longmapsto \varphi^t(x)
\end{align}
Upon denoting by $\phi^t$ the pullback $\Phi^*(\varphi^t)$, we see that $\phi^t$ satisfies the following properties:
\begin{align}
    \phi^0\equiv \text{Id} \qquad \phi^t: \Sigma_0 \times\{0\} \rightarrow \Sigma_0\times \{t\}
\end{align}
It then follows that the following pullback is cohomologous with $\Omega_0:= \omega\big\vert_{\Sigma_0}$
\begin{align}
    \Omega_t:= \big(\phi^t\big)^*(\Phi^*\omega) \qquad \text{for each $t \in [0,a]$}
\end{align}
where by a slight abuse of notation, we denote by $\omega$ the restriction $\tilde D(\Sigma)\big\vert_{\omega}$. Indeed, a simple computation using Cartan's magic formula and \cite[Lemma 2.2.1]{ Geiges_book} shows that the following holds for each $t \in [0,a]$:
\begin{align}\label{eta_t_hyp_def_eqn}
    \Omega_t -\Omega_0 = \Phi^*\bigg( d\Big(\int^t_0 (\varphi^s)^* (\iota_{X} \omega) \; ds\Big)\bigg) = d\Big(\int^t_0 \eta_s \;ds\Big)
\end{align}
The above equation motivates us to define the following family of $1$-forms on $\Sigma_0$:
\begin{align}
   \alpha_t:= \int^t_0 \eta  \qquad \text{for $t \in [0,a]$}
\end{align}
So that \eqref{eta_t_hyp_def_eqn} can be expressed in the following succinct way:
\begin{align}
    \Omega_t -\Omega_0 = d\alpha_t
\end{align}
The following proposition shows that the Volume of the region $\tilde D(\Sigma)$ has a convenient expression in terms of the Volume functional defined in \eqref{Vol_function_def_eqn}:
\begin{proposition}\label{hyp_vol_prop}
    Let $(W,\omega)$ be a connected symplectic manifold of dimension $2n$ admitting a Zoll hypersurface $\Sigma_0 \hookrightarrow W$. Given another closed, connected and embedded hypersurface $\Sigma \hookrightarrow W$ that is $C^2$-close to $\Sigma_0$, the following holds:
    \begin{align}
        \int_{\Sigma_0 \times [0,a]} (\Phi^*\omega)^n =  nVol_{\Omega_0}(\Omega_a)
    \end{align}
    where $\Omega_0:= \omega\vert_{\Sigma_0}$, $\Omega_a = (\phi^a)^*\tilde \Omega_a$ where $\tilde \Omega_a:= \omega\vert_{\Sigma}$ and $\Phi$ is the diffeomrophism defined in \eqref{hyp_def_of_Phi}.
\end{proposition}
\begin{proof}
   On the product manifold $\Sigma_0 \times [0,a]$, the $2$-form 
   $\Phi^*\omega$ has the following expression:
   \begin{align}\label{Phi_omega_decomp}
       \Phi^*\omega = dt\wedge \eta_t+ \Omega_t
   \end{align}
As a consequence, the $2n$-form $\big(\Phi^*\omega\big)^n$ admits the following binomal expansion:
\begin{align}\label{phi_omega_n_binom}
    \big(\Phi^*\omega\big)^n = \big(dt\wedge \eta_t\wedge\Omega_t\big)^n = \sum^n_{k=0}\binom{n}{k}\big(dt \wedge \eta_t\big)^k \wedge \big(\Omega_t\big)^{n-k}
\end{align}
Since the $1$-form $dt \in \Omega^1([0,a])$, for any $k>1$, the following holds:
\begin{align}
    \big(dt\wedge\eta_t\big)^{k} =0
\end{align}
And since for each $t \in [0,a]$, the $2$-form $\Omega_t \in \Omega^2(\Sigma_0)$, the $2n$-form $\Omega_t^n =0$ because $\Sigma_0$ has dimension $2n-1$. Therefore, we can rewrite \eqref{phi_omega_n_binom} as the following:
\begin{align}
     \big(\Phi^*\omega\big)^n =n dt \wedge \eta_t\wedge \Omega^{n-1}_t
\end{align}
The required integral then has the following expression:
\begin{align}\label{hyp_vol_int_1}
    \int_{\Sigma_0\times [0,a]} \big(\Phi^*\omega\big)^n = \int_{\Sigma_0\times [0,a]} n dt \wedge \eta_t\wedge \Omega^{n-1}_t = n\int^a_0\int_{\Sigma_0}  \eta_t\wedge\Omega^{n-1}_t\;dt 
\end{align}
Computing further, we obtain the following:  
\begin{align}\label{hyp_vol_int_2}
 n\int^a_0\int_{\Sigma_0}  \eta_t\wedge \Omega^{n-1}_t\;dt  = n\int^a_0\frac{d}{dt}Vol_{\Omega_0}(\Omega_t) \; &dt=   n \Big(Vol_{\Omega_0}(\Omega_a) -Vol_{\Omega_0}(\Omega_0)\Big)\\\nonumber = &nVol_{\Omega_0}(\Omega_a)
  \end{align}
  Where the first equality follows form the fact that by \cite{B-K-Odd_Symp} the following is true:
  \begin{align}
      d_{\alpha_a}Vol_{\Omega_0}(\Omega_a)\cdot \beta = \int_{\Sigma_0} \beta \wedge \Omega_a^{n-1} 
  \end{align}
\end{proof}
We now define the action and systole of an arbitrary regular level set of $H$.  \\
\\
The Hamiltonian vector field of $H_0$ is the unique vector field satisfying the following differential equation:
\begin{align}
    \iota_{X_{H_0}}\omega = dH_0
\end{align}
We denote by $\tilde X_H$ the pullback: $\Phi^*(X_H)$. When restricted to $\Sigma_0 \times \{t\}$ for some $t \in [0,1] $, the flow of $\tilde X_H$ generates the characteristic distribution of $\tilde \Omega_t:= \Phi^*\omega\big\vert_{\Sigma_0 \times \{t\}}$.\\
\\
We now pick a closed leaf $\gamma$ of the characteristic foliation of $\tilde \Omega_a$ and since $\Sigma$ is assumed to be $C^2$-close to $\Sigma_0$, we can find a leaf $ \gamma_0$ of the characteristic foliation of $\Omega_0= \Phi^*\omega\big\vert_{\Sigma_0\times \{0\}}$ that is $C^1$-close to $\gamma$. In such a setting, we will denote by $\Gamma_{\gamma}$ a cylinder parameterizing a homotopy between $\gamma$ and $\gamma_0$ and we define the action of $\gamma$ to be the following:
\begin{align}
    \mathcal{A}_{\Sigma}(\gamma):= \int_{\Gamma_{\gamma}}\Phi^*\omega
\end{align}
Since $\omega$ is closed, the above integral is invariant under homotopies with a fixed boundary. We therefore choose a short cylinder in the following way: let $\tilde \Gamma_{\gamma}$ be a projection of $\Gamma_{\gamma}$ to $\Sigma_0$ and consider the cylinder defined to be the concatenation $\tilde \Gamma_{\gamma} \# (\phi^t)^*\gamma$. Using \eqref{Phi_omega_decomp}, we observe that the above integral has the following expression:
\begin{align}\label{hyp_action_decomp_1}
    \int_{\Gamma_{\gamma}}\Phi^*\omega = \int_{\Gamma_{\gamma}} dt \wedge \eta_t +  \int_{\Gamma_{\gamma}} \tilde \Omega_t   =\int^a_0\int_{(\phi^t)^*\gamma} \eta_t \; dt + \int_{\tilde \Gamma_{\gamma} \# (\phi^t)^*\gamma} \Omega_t
\end{align}
By definition of the $1$-form $\alpha_t$, the first factor in \eqref{hyp_action_decomp_1} can be expressed in the following form:
\begin{align}\label{hyp_action_decomp_2}
\int^a_0\int_{(\phi^t)^*\gamma} \eta_t \; dt = \int_{(\phi^a)^*\gamma} \alpha_a
\end{align}
Putting \eqref{hyp_action_decomp_1} and \eqref{hyp_action_decomp_2} together, we obtain the following expression for $\mathcal{A}_{\Sigma}(\gamma)$:
\begin{align}
    \mathcal{A}_{\Sigma}(\gamma) = \int_{\Gamma_{\gamma}}\Phi^*\omega = \int_{(\phi^a)^*\gamma} \alpha_a + \int_{\tilde \Gamma_{\gamma} \# (\phi^t)^*\gamma} \Omega_0 = \mathcal{A}_{\Omega_a}((\phi^a)^*\gamma)
\end{align}
where $\mathcal{A}_{\Omega_a}$ is the action defined in Section \ref{action_section}. Denoting by $\mathfrak{h} \in \pi_1(\Sigma_0)$ the homotopy class represented by the leaves of the characteristic foliation of $\Sigma_0$, we recall that the set of short loops in $\Sigma_0$ were denoted by $\Lambda_{\mathfrak{h}}(\Sigma_0)$ (c.f Section \ref{action_section}). We can now conveniently express the systole of $\Sigma$ in the following way:
\begin{align}\label{sys_in_odd-symp}
    \text{sys}(\Sigma)= \inf_{\bar \gamma \in \Lambda_{\mathfrak{h}}(\Sigma_0)} \mathcal{A}_{\Omega_a}(\bar \gamma)
\end{align}
Since the pair $(\Sigma_0,\Omega_0)$ is a Zoll odd-symplectic manifold, there exists an oriented $S^1$-bundle $p: \Sigma_0 \rightarrow M$ whose leaves are the leaves of characteristic foliation of $\Omega_0$. Given a connection $1$-form $\alpha_0 \in \Omega^1(\Sigma_0)$ for such an $S^1$-bundle, it is easy see that the following realtions are automatically satisfied
\begin{align}
  \alpha_0 \in \mathcal{C}(\Sigma_0,\Omega_0) \;\;\; \text{and} \;\;\;  \ker \Omega_0 \subset \ker d\alpha_0
\end{align}
Therefore, the odd-symplectic form $\Omega_0$ and $1$-form $\alpha_0$ determine a stable Hamiltonian structure on $\Sigma_0$ in the sense of \cite{CieliebakFrauenfelderPaternain2010}. As a consequence, there exists an $\tilde \epsilon_0 >0$ such that the neighborhood $\Sigma_0 \times (-\tilde\epsilon_0,\tilde \epsilon_0)$ is foliated by Zoll hypersurfaces. By re-normalizing if necessary we can parametrize the leaves of this foliation in the following way: we denote the leaves by $\hat{\Sigma}_t$ for $t \in (-\tilde\epsilon_0,\tilde \epsilon_0)$ such that the following holds for every $t \in  (-\tilde\epsilon_0,\tilde \epsilon_0)$:
\begin{align}
    \text{sys}(\hat{\Sigma}_t) = t = \mathcal{A}_{\hat{\Sigma}_t}(\tilde\gamma_t) \;\;\; \text{\small for any leaf $\tilde \gamma_t$ of the characteristic foliation of $\hat{\Sigma}_t$}
\end{align}
We can also assume with any loss of generality that $\tilde \epsilon_0$ was chosen so that $\text{sys}(\Sigma) \in [0,\tilde\epsilon_0)$. In such a situation, since the hypersurface $\hat{\Sigma}_{\text{sys}(\Sigma)}$ is Zoll, by Proposition \ref{hyp_vol_prop} the following holds:
\begin{align}
\text{Vol}\big(\tilde D(\hat{\Sigma}_{\text{sys}(\Sigma)})\big) = P\big(\text{sys}(\Sigma)\big)
\end{align}
where the volume of the manifold $\tilde D(\hat{\Sigma}_{\text{sys}(\Sigma)})$ is computed with respect to $\omega$. The $1$-form $\eta_a$ defined in \eqref{eta_t_hyp_def_eqn} has the following Taylor expansion:
\begin{align}\label{eta_taylor}
    \eta_a = \eta_0 + a\dot\eta_0 + \frac{a^2}{2}\ddot \eta_0 + \frac{a^3}{6} \partial_a\ddot \eta_0 + \mathcal{O}(a^4)
\end{align}
Therefore, by \eqref{eta_t_hyp_def_eqn}, the odd-symplectic form $\Omega_a$ will be $C^2$-close to $\Omega_0$ if $a>0$ chosen so that $a^3$ is close to zero. So, we can now use Theorem \ref{systolic_ineq} to obtain the following inequality:
\begin{align}
    Vol\big(\tilde D(\hat{\Sigma}_{\text{sys}(\Sigma)})\big) = P\big(\text{sys}(\Sigma)\big) \leq Vol_{\Omega_0}\big(\Omega_a\big) = Vol\big(\tilde D(\Sigma)\big)
\end{align}
with equality holding if and only if the action of every leaf of the characteristic foliation of $\Omega_a$ equals $\text{sys}(\Sigma)$, which happens if and only if $\Sigma$ is a Zoll hypersurface as required.  

\section{Magnetic dynamics on Riemannian manifolds}\label{mag_dyn_riem_section}
We begin this section with a justification of the fact that odd-symplectic forms model the dynamics of a magnetic system on a regular energy level.\\
\\
Observe that given a magnetic pair $(g,\sigma)$ on a Riemannian manifold $(W,g)$ of dimension $n\geq 3$, the trajectories of the associated magnetic geodesic flow $\Phi^t_{g,\sigma}$ on a regular level set $H^{-1}_{g}(k)$ can be viewed as the leaves of the characteristic foliation of the following odd-symplectic form defined on the manifold $S_g^*W \cong H^{-1}_{g}(\frac{1}{2})$:
\begin{align}\label{omega_s_def}
    \Omega^s_{g,\sigma}:= d\lambda-\pi^*\sigma_s \;\;\; \text{where $\sigma_s:= s\sigma$ and $s = \frac{1}{\sqrt{2k}}$}
\end{align}
The above odd-symplectic form is called the odd-symplectic form determined by the magnetic pair $(g,\sigma)$. It is also easy to see from \eqref{omega_s_def} that the odd-symplectic form defined above will be Zoll if and only if the magnetic system determined by the pair $(g,\sigma)$ is Zoll at strength $s$. In the sequel, if the level set of the metric Hamiltonian is obvious, in the interest of a more concise presentation of the formulas, we will drop the constant $s$ from the notation.\\
\\
In this section, we intend to compare the dynamics of the two different magnetic systems on a fixed odd-symplectic manifold. To do so, we begin by picking two magnetic pairs $(g,\sigma) , (\tilde g,\tilde \sigma) \in \mathcal{M}(W) \times \Omega^2(W)$ on the Riemannian manifold $(W,g)$, and define the following function:
\begin{align}\label{def_of_the_function_f}
      f_{\tilde g}:\;&T^*W \rightarrow \mathbb{R}\\\nonumber
      &(p,q)\mapsto \frac{\vert\vert p\vert\vert_{\tilde g}}{\vert\vert p\vert\vert_{g}}
\end{align}
On the unit sphere sub-bundle $S_g^*W$ of the Riemannian manifold $(W,g)$, the trajectories of the magentic geodesic flow $\Phi_{\tilde g,\tilde \sigma}$ are the leaves of the characteristic foliation of the following odd-symplectic form:
\begin{align}\label{odd-symp_determiend_by_triple}
  \tilde \Omega^{s}_{\tilde g,\tilde \sigma}:=  d(f_{\tilde g}\lambda)-s\pi^*\tilde\sigma \Big|_{S_g^*W}
\end{align}
On the manifold $S_g^*W$, the odd-symplectic form defined by the magnetic pair $(g,\sigma)$ has the form \eqref{omega_s_def} and that defined by the magnetic pair $(\tilde g,\tilde \sigma)$ has the form \eqref{odd-symp_determiend_by_triple}. Since the first summand in \eqref{odd-symp_determiend_by_triple} is exact, it is easy to see that for each fixed $s\in \mathbb{R}$, the two odd-symplectic forms determiend on the manifold $S_g^*W$ by the two magnetic pairs $(g,\sigma)$ and $(\tilde g,\tilde \sigma)$ are cohomologous if and only if $\sigma$ and $\tilde \sigma$ are.\\
\\
We now fix a closed and connected Riemannian manifold $(W,g)$ of dimension $n\geq 3$, a Zoll magnetic pair $(g,\sigma)$ and strength $s \in \mathbb{R}$ at which the pair is Zoll. We will also assume that the class $C_0 \in H^2(W,\mathbb{R})$ represented by $\sigma$ satisfies the following condition:
\begin{align}\label{cohomology_assumption_mag_systems}
   0=\pi^*\big[C^{n-1}_0\big] \in H^{2n-2}(S^*_gW,\mathbb{R})
\end{align}
We also fix an $S^1$-bundle associated with the Zoll odd-symplectic form determined by the fixed Zoll magnetic pair $(g,\sigma)$ of strength $s$ and denote by $\mathfrak{h} \in \pi_1(S_g^*W)$ the homotopy class represented by the fibers of this $S^1$-bundle.\\
\\
Moreover, observe that since $W$ has dimension $n\geq 3$, the fibers of the following bundle are contractible:
\begin{align}\label{S^m_fibration}
    S^{n-1}\hookrightarrow S_g^*W \xrightarrow{\pi} W
\end{align}
where by a slight abuse of notation, we denote by $\pi$ the restriction of foot-point projection to $S_g^*W$. As a consequence, using the long exact sequence for homotopy groups of the fibration \eqref{S^m_fibration}, the following holds:
\begin{align}
    \pi_1(S_g^*W) \cong \pi_1(W)
\end{align}
Motivated by the above observation, we will hence forth identify $ \pi_1(S_g^*W)$ with $ \pi_1(W)$.\\
\\
In this setting, using the odd-symplectic action defined in Section \ref{action_section}, the magnetic action of an arbitrary magnetic pair $(\tilde g,\tilde \sigma) \in \mathcal{M}(W) \times \Xi^2_{C_0}(W)$ is defined to be the following:
\begin{align}\label{mag_action_def_eqn}
\mathcal{A}^{\text{mag}}_{\tilde g,\tilde \sigma}:\; \Lambda_{\mathfrak{h}}(S^*&W) \rightarrow \mathbb{R}\\\nonumber
    &\gamma\mapsto \mathcal{A}_{\tilde \Omega_{\tilde g,\tilde \sigma}}(\gamma)
\end{align}
where $\tilde \Omega_{\tilde g,\tilde \sigma}:= d(f_{\tilde g}\lambda)-\pi^*\tilde\sigma\big|_{S_g^*W}$. We will also denote by $\tilde \chi^{s}(\tilde g, \tilde \sigma)$ the set of closed leaves of the characteristic foliation of $\tilde \Omega^{s}_{\tilde g,\tilde \sigma}$ that can be joined to the leaves of the characteristic foliation of the Zoll odd-symplectic form $\big(d\lambda-s\pi^*\sigma\big)\big\vert_{S_g^*W}$ by a short homotopy (c.f Section \ref{action_section}). Recall that in Section \ref{mag_sys_sec_intro}, we introduced the notation
$\chi^{s}(\tilde g,\tilde \sigma)$ for the set of closed magnetic geodesics of the
pair $(\tilde g,\tilde \sigma)$ of strength $s$ that are $C^2$-close to those of $(g,\sigma)$.
Since these geodesics are parametrized by arc length, there is a canonical lift map
\begin{align}
    L_{\tilde g,g}:\chi^{s}(\tilde g,\tilde \sigma)\longrightarrow
\tilde \chi^{s}(\tilde g,\tilde \sigma)
\end{align}
obtained by first lifting
\begin{align}
    \gamma \mapsto \big(t\mapsto \tilde g_{\gamma(t)}(\dot\gamma(t),\cdot)\big)
\end{align}
to $S^*_{\tilde g}W$ and then composing with the radial identification
$S^*_{\tilde g}W \rightarrow S^*_gW$ obtained via the function $f_{\tilde g}$. With this convention, the following holds:
\begin{align}
    \tilde \chi^{s}(\tilde g,\tilde \sigma)
=
L_{\tilde g,g}\big(\chi^{s}(\tilde g,\tilde \sigma)\big)
\end{align}
Thus, for $\gamma\in \chi^{s}(\tilde g,\tilde \sigma)$, we denote its \emph{canonical}
lift by $\tilde\gamma:= L_{\tilde g,g}(\gamma) \in \tilde \chi^{s}(\tilde g,\tilde \sigma)$. The justification of the fact that the set $\tilde \chi^{s}(\tilde g,\tilde \sigma)$ is non-empty when the magnetic pair $(\tilde g,\tilde \sigma)$ is sufficiently $C^3$-close to the pair fixed Zoll pair $(g,\sigma)$ is provided in Section \ref{mag_sys_ineq_general_proof}. Since the strength $s$ is fixed throughout and no ambiguity can arise, we will suppress it from the notation in the sequel. For the benefit of the reader, we recall from the introduction the following convention that was introduced: when picking a Zoll magnetic pair, we will always implicitly assume that its dynamics is being considered at a fixed strength where it is Zoll. In addition, in the sequel, we would like to compare the dynamics of such a pair with another appropriately chosen magnetic pair. In such a situation, we will always assume that both pairs are being considered at a fixed strength, where the Zoll magnetic pair is Zoll.\\
\\
In this context, the following proposition furnishes a convenient expression for the magnetic action.
\begin{proposition}\label{mag_sys_action_general_propsition}
    Let $(W,g)$ be a closed, connected Riemannian manifold of dimension at least three with the property that there exists a cohomology class $C_0 \in H^2(W,\mathbb{R})$ satisfying \eqref{cohomology_assumption_mag_systems} and admitting a representative $\sigma \in \Xi^2_{C_0}(W)$ such that the magnetic pair $(g,\sigma)$ is Zoll. Then, for any other magnetic pair $(\tilde g,\tilde\sigma) \in \mathcal{M}(W)\times \Xi_{C_0}^2(W)$ with the property that the set $\chi(\tilde g,\tilde\sigma)$ is non-empty, the following holds for any $\gamma \in \chi(\tilde g,\tilde\sigma)$:
    \begin{align}
\mathcal{A}^{\text{mag}}_{\tilde g,\tilde \sigma}(\gamma)=  \text{length}_{\tilde g}(\gamma) - \int_{\Gamma_{\gamma}} \tilde \sigma - \text{length}_{g}(\gamma_0) + \int_{\gamma_0}\eta
    \end{align}
    where $\gamma_0 \in \chi(g,\sigma)$ is a magnetic geodesic whose lift $\tilde \gamma_0 \in \tilde \chi(g,\sigma)$ is contained in the same good neighborhood of $\tilde \gamma$, $\Gamma_\gamma$ is a short homotopy between $\tilde \gamma$ and $\tilde \gamma_0$ and $\eta \in \Omega^1(W)$ is any $1$-form such that $\sigma_0 = \sigma+d\eta$.
\end{proposition}
\begin{proof}
    The odd-symplectic form defined on the manifold $S^*_gW$ by the magnetic pair $(g,\sigma)$ has the form \eqref{omega_s_def} and that defined by the magnetic pair $(\tilde g,\tilde \sigma)$ has the form \eqref{odd-symp_determiend_by_triple}.\\
    \\
    Let $(\tilde g, \tilde\sigma) \in \mathcal{M}(W) \times \Xi^2_{C_0}(W)$ be another magnetic pair such that the set $\chi(\tilde g,\tilde\sigma)$ is non-empty. The odd-symplectic form defined on the manifold $S_g^*W$ is of the form \eqref{odd-symp_determiend_by_triple}, using this fact and the definition of the magnetic action in \eqref{mag_action_def_eqn} and the discussion in Section \ref{action_section}, we see that the magnetic action $\mathcal{A}^{\text{mag}}_{\tilde g,\tilde\sigma}(\gamma)$ of a loop $\gamma \in \chi(\tilde g,\tilde\sigma)$ has the following expression:
\begin{align}\label{mag_action_proof_0}
    \mathcal{A}^{\text{mag}}_{\tilde g,\tilde\sigma}(\gamma) = \int_{\tilde {\gamma}_0} (f_{\tilde g}-1)\lambda - &\int_{\gamma_0}\pi^*\eta  - \int_{\Gamma_{\gamma}} d(f_{\tilde g}\lambda)-\pi^*\sigma \\\nonumber = \int_{\tilde {\gamma}_0} (f_{\tilde g}-1)\lambda - \int_{\gamma_0}\pi^*\eta + &\int_{\tilde {\gamma}_1}(f_{\tilde g}\lambda) - \int_{\tilde {\gamma}_0}(f_{\tilde g}\lambda) - \int_{\Gamma_{\gamma}} \pi^*\sigma \\\nonumber =\int_{\tilde {\gamma}_1}f_{\tilde g}\lambda - \int_{\Gamma_{\gamma}}\pi^* \sigma - &\int_{\tilde {\gamma}_0} \lambda -\int_{\gamma_0} \pi^*\eta  
   \end{align}
   where the second equality above follows from an application of Stokes' theorem and $\eta \in \Omega^1(W)$ is any $1$-from with the following property:
\begin{align}\label{def_of_eat_mag_sys_ineq_1}
\sigma = \sigma_0 + d\eta
   \end{align}
The following is the $\tilde g$-speed of $\gamma$
\begin{align}
    \big(f_{\tilde g}\lambda\big)(\tilde  \gamma)
\end{align}
Therefore, the following holds:
\begin{align}
   \int_{\tilde  \gamma} \big(f_{\tilde g}\lambda\big) = \text{length}_{\tilde g}(\gamma)
\end{align}
We can similarly conclude that the following holds:
\begin{align}
    \int_{\tilde \gamma_0} \lambda = \text{length}_{g}(\gamma_0)
\end{align}
To finish the proof, we need to show that the integral $\int_{\gamma_0}\pi^*\eta$ is independent of the choice of $1$-form satisfying \eqref{def_of_eat_mag_sys_ineq_1}. Towards this end, we begin by proving that the condition \eqref{cohomology_assumption_mag_systems} implies that the cohomology class $\tilde C_0:= \Big[d\lambda-\pi^*\sigma_0\big|_{S_g^*W}\Big] \in H^2(S^*_gW,\mathbb{R})$ satisfies the condition:
\begin{align}
    \tilde C^{n-1}_0 =0
\end{align}
Indeed, the following holds because the $2$-form $(d\lambda)|_{S_g^*W}$ is exact:
\begin{align}
    \tilde C_0 = \pi^*[\sigma]= \pi^*C_0
\end{align}
This inturn implies the following:
\begin{align}
    \Big[d\lambda-\pi^*\sigma_0\big|_{S_g^*W}\Big]^{n-1} = \big[\pi^* C_0\big]^{n-1} = \pi^*\big[ C_0\big]^{n-1} =0
\end{align}
where the final equality above follows from the condition \eqref{cohomology_assumption_mag_systems}.\\
\\
We now pick an arbitrary closed $1$-form $\theta \in \Omega^1(W)$, we can define $\eta':= \eta + \theta$ and observe that the following holds:
\begin{align}\label{int_eta_1}
       \int_{\tilde \gamma_0}\pi^*\eta' = \int_{\tilde \gamma_0} \pi^*\eta + \int_{\tilde \gamma_0}\pi^*\theta 
\end{align}
However, $\tilde  \gamma_0$ is the fiber of the $S^1$-bundle associated with the Zoll odd-symplectic form defined by the Zoll magnetic pair $(g,\sigma)$. And since $\tilde C^{n-1}_0 =0$, by \cite[Lemma 4.5]{B-K-Odd_Symp}, the homology class $[\tilde  \gamma_0] \in H_1(S^*_gW,\mathbb{R})$ vanishes. In particular, since the $1$-form $\theta$ in \eqref{int_eta_1} is closed, the following holds:
\begin{align}
    \int_{\tilde \gamma_0}\pi^*\theta = \langle [\tilde  \gamma_0], \theta \rangle =0
\end{align}
Therefore, the integral in \eqref{int_eta_1} is independent of the choice of closed $1$-form  $\theta$. This in turn implies that the last summand in the last line of \eqref{mag_action_proof_0} is independent of the choice of $\eta \in \Omega^1(W)$ satisfying \eqref{def_of_eat_mag_sys_ineq_1}.
\end{proof}
Before proceeding further, we recall some facts about the geometry of cotangent bundles of Riemannian manifolds. Given a Riemannian manifold $(W,g)$, a point $q \in W$ and some $p \in T_q^*W$, the tangent space $T_{p}TW$ can we written as the direct sum of two vector spaces $\mathcal{H}(p)$ and $\mathcal{V}(p)$ called the horizontal and vertical components respectively. This isomorphism is constructed by combining the following two endomorphisms:
\begin{align}
    pr_{\mathcal{H}}: &T_{p}T^*W \rightarrow T_q W \cong \mathcal{\mathcal{H}}(p)  \\\nonumber
   &\xi \mapsto \xi_{\mathcal{H}} := d_{p}\pi(\xi) 
\end{align}
\begin{align}
    pr_{\mathcal{V}}: &T_{p}T^*W \rightarrow T^*_q W \cong \mathcal{V}(p) \\\nonumber
   &\xi \mapsto \xi_\mathcal{V} := \nabla_{t}(Z(0)) 
\end{align}
where $Z:(-\epsilon, \epsilon) \rightarrow T^*W$ is a curve such that $Z(0) = p$ and $\Dot{Z}(0) = \xi$ and $\nabla_t$ represents the covariant derivative along the curve $(\pi \circ Z)$. We now define $\mathcal{H}(p)$ to be the set of all $\xi \in T_pT^*W$ such that $\xi_{\mathcal{V}} =0$ and $\mathcal{V}(p)$ to be the set of all $\xi \in T_{p}T^*W$ such that $\xi_{\mathcal{H}} =0$. We also have the following identifications given by the maps $pr_{\mathcal{H}}$ and $pr_{\mathcal{V}}$ respectively: $\mathcal{H}(p) \cong T_qW$ and $\mathcal{V}(p) \cong T^*_qW$. The Sasaki metric on $T^*W$ is defined with respect to this direct sum decomposition in the following way:
\begin{align}
    \Hat{g} : = g \oplus g
\end{align}
\begin{proposition}\label{vol_mag_sys_general}
        Let $(W,g)$ be a closed, connected Riemannain manifold of dimension at least three, with the property that there exists a cohomology class $C_0 \in H^2(W,\mathbb{R})$ satisfying \eqref{cohomology_assumption_mag_systems} and admitting a representative $\sigma \in \Xi^2_{C_0}(W)$ such that the magnetic pair $(g,\sigma)$ is Zoll. Then, for any other magnetic pair $(\tilde g,\tilde \sigma) \in \mathcal{M}(W) \times \Xi^2_{C_0}(W)$, the following holds:
        \begin{align}
            Vol_{\Omega_{g,\sigma}}(\tilde \Omega_{\tilde g, \tilde \sigma}) =     n\text{Vol}(B^n) \Big( \text{Vol}_{\tilde g}(W) - \text{Vol}_{g}(W) \Big)
        \end{align}
        where $\Omega_{g,\sigma}$ is the odd-symplectic form defined on the manifold $S_g^*W$ by the magnetic pair $(g,\sigma)$ (c.f \eqref{omega_s_def}) and $\tilde \Omega_{\tilde g,\tilde\sigma}$ is the odd-symplectic form on the same manifold defined by the magnetic pair $(\tilde g,\tilde \sigma)$ (c.f \eqref{odd-symp_determiend_by_triple}). In particular if $\text{Vol}_{\tilde g}(W) = \text{Vol}_{g}(W)$ then, the following holds: \begin{align}
            Vol_{\Omega_{g,\sigma}}(\tilde \Omega_{\tilde g, \tilde \sigma}) =0
        \end{align}
    \end{proposition}
    \begin{proof}
        We begin the proof by fixing some notation. We will denote by $n \in \mathbb{N}$ the dimension of the Riemannian manifold $(W,g)$ and we will fix a $1$-form $\eta \in \Omega^1(W)$ satisfying \eqref{def_of_eat_mag_sys_ineq_1}.\\
        \\
To avoid overloading the notation in this proof, we will suppress pullbacks by the foot-point
projection $\pi:S^*W\rightarrow W$ throughout the proof; in particular, forms on $W$
such as $\sigma$, $\tilde \sigma$ and $\eta$ are implicitly identified with their pullbacks to
$S^*W$.\\
\\
The integral being computed here is the following:
\begin{align}\label{kahler_vol_int_1}
\int^1_0\int_{S^*_gW} \big((f_{\tilde g}-1)\lambda - \eta\big) \wedge\bigg(\Omega_{g,\sigma} + d(rf_{\tilde g}-1)\lambda - rd\eta \bigg)^{n-1} \; dr
\end{align}
where $\eta \in $ $\Omega_{g,\sigma}$ is the following odd-symplectic form:
\begin{align}
   \Omega_{g,\sigma}:= d\lambda -\pi^*\sigma_0\Big|_{S_g^*W}
\end{align}
The expression for the $2$-form $\Omega_{g,\sigma} + d(rf_{\tilde g}-1)\lambda - rd\eta$ can be simplified in the following way when treating $r \in \mathbb{R}$ as a constant:
\begin{align}
    \Omega_{g,\sigma} + d\big((&rf_{\tilde g}-1)\lambda\big) - rd\eta \; = \; d\lambda -\sigma + d(r(f_{\tilde g}-1)\lambda - r\eta ) \\\nonumber = &d\big(\big((1-r)+rf_{\tilde g} \big)\lambda \big) - (rd\eta +\sigma) = d(f^r_{\tilde g}\lambda) + \sigma_r
\end{align}
where $f^r_{\tilde g}:= (1-r) + rf_{\tilde g}$ and $\sigma_r := -(\sigma + rd\eta)$. Using this simplified expression, the integral in \eqref{kahler_vol_int_1} takes the following form:
\begin{align}
    \int^1_0\int_{S^*_gW} \big((f_{\tilde g} -1)\lambda +\eta \big)\wedge \big(d(f^r_{\tilde g}\lambda) +\sigma_r \big)^{n-1}
\end{align}
Using the binomial theorem, we see that the $2n-1$-form $\big((f_{\tilde g} -1)\lambda +\eta \big)\wedge \big(d(f^r_{\tilde g}\lambda) +\sigma_r \big)^{n-1}$ has the following expression:
\begin{align}\label{kahler_vol_int_2}
    \sum^{n-1}_{k=0}  &\binom{n-1}{k} (f_{\tilde g}-1)\lambda \wedge \big(d(f^r_{\tilde g}\lambda)\big)^k \wedge (\sigma_r)^{n-1-k} \\\nonumber &+ \sum^{n-1}_{k=0} \binom{n-1}{k} \eta \wedge \big(d(f^r_{\tilde g}\lambda)\big)^k \wedge (\sigma_r)^{n-1-k}
\end{align}
We now claim that the integral of the second summand in \eqref{kahler_vol_int_2} over $S^*_gW$ is zero. We begin by considering the case $k=0$ in \eqref{kahler_vol_int_2}. In this case, the second summand can be written in the following form
\begin{align}\label{eta_wedge_vol_1}
    \eta \wedge \big(\sigma_r \big)^{n-1}
\end{align}
Since both $\eta$ and $\sigma_r= -(\sigma + rd\eta)$ are pullbacks to $S_g^*W$ of $2$-forms on $W$, they evaluate as non-zero only on vectors in the horizontal distribution of $TS_g^*W$. However, $TS^*_gW$ has only $2n$ linearly independent horizontal directions. Therefore, the above $2n-1$-form must vanish. Suppose now that $0< k< 2n-1$ then, the all of the vertical components in the following integral are supplied by the $2k$-form $\big(d(f^r_{\tilde g}\lambda)\big)^k$:
\begin{align}
    \binom{n-1}{k} \int^1_0\int_{S^*_gW} \eta \wedge \big(d(f^r_{\tilde g}\lambda)\big)^k \wedge (\sigma_r)^{n-1-k}
\end{align}
However, the tangent space $TS^*_gW$ has $2n-1$ linearly independent vertical directions and since we are only looking at the case where $k<2n-1$, the above integral does not have sufficient vertical terms to evaluate as non-zero, therefore, it must vanish. Finally, if $k=2n-1$ then, the required integral can be written in the following form using the fact that $d^2=0$ for the exterior derivative:
\begin{align}
    \int^1_0\int_{S^*_gW} \eta \wedge \big(d(f^r_{\tilde g}\lambda)\big)^{n-1}=
\int_{S^*_gW}
\eta\wedge
d\Bigl(
f^{r}_{\tilde g}\lambda\wedge d(f^{r}_{\tilde g}\lambda)^{n-2}
\Bigr)
\end{align}
Since $S^*_gW$ is closed, integrating by parts, we obtain the following:
\begin{align}\label{mag_vol_general_1}
    \int_{S_g^*W}
\eta\wedge d(f^{r}_{\tilde g}\lambda)^{n-1}
=
\int_{S_g^*W}
d\eta\wedge f^{r}_{\tilde g}\lambda\wedge d(f^{r}_{\tilde g}\lambda)^{n-2}
\end{align}
The form $d\eta$ only evaluates as non-zero on horizontal vectors in $TS^*_gW$. Moreover, the following $2n-3$-form has vertical degree at most $2n-4$, since $\lambda$ is purely horizontal, that is, it evaluates as non-zero only on horizontal vectors in $TS^*_gW$:
\begin{align}
    f^{r}_{\tilde g}\lambda\wedge d(f^{r}_{\tilde g}\lambda)^{n-2}
\end{align}
Therefore the integrand in \eqref{mag_vol_general_1} has vertical degree at most $2n-4$, which is again too small to contribute to a non-zero integral over the manifold $S_g^*W$. Consequently, the integral in \eqref{mag_vol_general_1} vanishes, as required.\\
\\
We can now conclude that the following integral vanishes for every $k$:
\begin{align}
    \binom{n-1}{k}\int_{S^*_gW} \eta \wedge \big(d(f_g^r\lambda) \big)^k \wedge (\sigma_r)^{n-1-k}  = 0 \;\; \forall k
\end{align}
Therefore, only the first summand in \eqref{kahler_vol_int_2} contributes to the integral \eqref{kahler_vol_int_1} i.e, the following holds:
\begin{align}\label{kahler_vol_3}
&\int^1_0\int_{S^*_gW} \big((f_{\tilde g}-1)\lambda - \eta\big) \wedge\bigg(\Omega_{g,\sigma} + d(rf_{\tilde g}-1)\lambda - rd\eta \bigg)^{n-1} \; dr \\\nonumber &= \sum^{n-1}_{k=0}  \binom{n-1}{k} \int^1_0\int_{S^*_gW} (f_{\tilde g}-1)\lambda \wedge \big(d(f^r_{\tilde g}\lambda)\big)^k \wedge (\sigma_r)^{n-1-k} \; dr
\end{align}
The expression for the $3$-form $\lambda \wedge \big(d(f^r_{\tilde g}\lambda)\big)$ can be simplified in the following way:
\begin{align}
   \lambda \wedge d\bigl(f_{\tilde g}^{\,r}\lambda\bigr)
=
\lambda \wedge
\left(df_{\tilde g}^{\,r}\wedge \lambda
+
f_{\tilde g}^{\,r}d\lambda\right)
=
f_{\tilde g}^{\,r}\lambda\wedge d\lambda.
\end{align}
where the final equality above follows from the fact that $\lambda \wedge \lambda =0$. As a consequence of the above observation, the following holds:
\begin{align}
    \lambda \wedge \bigl(d(f_{\tilde g}^{\,r}\lambda)\bigr)^{k}
=
\bigl(f_{\tilde g}^{\,r}\bigr)^{k}
\lambda \wedge (d\lambda)^{k}
\end{align}
To evaluate as non-zero, the $2$-form $d\lambda$ must be evaluated on one vector in the vertical distribution and one in the horizontal distribution. As a consequence, the $2k+1$-form $ \lambda \wedge \bigl(d(f_{\tilde g}^{\,r}\lambda)\bigr)^{k}$ needs to be evaluated on exactly $k$ vertical and $k+1$ horizontal vectors to be non-zero. This observation combined with the facts that $TS^*_gW$ has exactly $2n$ horizontal and $2n-1$ vertical linearly independent directions and the $2$-form $\sigma_r$ evaluates as non-zero only on two horizontal vectors, imply that the following integral is non-zero if and only if $k=n-1$:
\begin{align}
    \int_{S^*_gW} (f_{\tilde g}-1)\lambda \wedge \big(d(f^r_{\tilde g}\lambda)\big)^k \wedge (\sigma_r)^{n-1-k}
\end{align}
Therefore, the integral in \eqref{kahler_vol_3} simplifies to the following:
\begin{align}\label{Kahler_vol_4}
\int^1_0\int_{S^*_gW} \big((f_{\tilde g}-1)\lambda - \eta\big) \wedge\bigg(\Omega_{g,\sigma} + d(rf_{\tilde g}-1)\lambda - rd\eta \bigg)^{n-1} \; dr\\\nonumber = \int^1_0\int_{S^*_gW} (f_{\tilde g}-1)\lambda \wedge \big(d(f_{\tilde g}^r\lambda) \big)^{n-1}\; dr = \frac{1}{n} \int_{S^*_gW} (f^{n}_{\tilde g}-1)\lambda \wedge d\lambda^{n-1}
\end{align}
The evaluation of the integral \eqref{Kahler_vol_4} is now straightforward. By the Fubini-Tonelli theorem, it is the product of its horizontal and vertical components (c.f \cite[Pg.3]{Abbondandolo-Benedetti}). So, we obtain the following:
\begin{align}
    \int_{S^*_gW}
f_{\tilde g}^{\,n}\lambda\wedge(d\lambda)^{n-1}
=
\frac{n\pi^{n/2}}{\Gamma\left(\frac n2+1\right)}
\text{Vol}_{\tilde g}(W)
\end{align}
where $\text{Vol}_{\tilde g}(W)$ is the $\tilde g$-volume of $W$. Equivalently, we can write the above integral in the following form:
\begin{align}
    \int_{S^*_gW}
f_{\tilde g}^{\,n}\lambda\wedge(d\lambda)^{n-1}
=
n\,\text{Vol}(B^n)\text{Vol}_{\tilde g}(W)
\end{align}
where $\text{Vol}(B^n)$ is the euclidean volume of the unit ball of dimension $n$. Indeed, the following holds:
\begin{align}
    \text{Vol}(B^n)
=
\frac{\pi^{n/2}}{\Gamma\left(\frac n2+1\right)}
\end{align}
Similarly, we obtain:
\begin{align}
     \int_{S^*_gW}
\lambda\wedge(d\lambda)^{n-1}
=
\frac{n\pi^{n/2}}{\Gamma\left(\frac n2+1\right)}
\text{Vol}_{g}(W) = n \text{Vol}(B^n)\text{Vol}_g(W)
\end{align}
We thus obtain the following, as required:
\begin{align}
Vol_{\Omega_{g,\sigma}}(\tilde \Omega_{\tilde g, \tilde \sigma}) =     n\text{Vol}(B^n) \Big( \text{Vol}_{\tilde g}(W) - \text{Vol}_{g}(W) \Big)
\end{align}
    \end{proof}
\subsection{Proof of Theorem \ref{mag_sys_ineq_general}}\label{mag_sys_ineq_general_proof}
We begin by recalling the geometric setup of Theorem \ref{mag_sys_ineq_general}: $(W,g)$ is a closed and connected Riemannian manifold of dimension $n\geq 3$ and $C_0 \in H^2(W,\mathbb{R})$ is a cohomology class satisfying condition \eqref{cohomology_assumption_mag_systems} and admitting a representative $\sigma \in \Xi^2_{C_0}(W)$ with the property that the magnetic system determined by the pair $(g,\sigma)$ is Zoll; we now consider the following Zoll odd-symplectic form on the manifold $S_g^*W$, defined by the pair $(g,\sigma)$:
\begin{align*}
    \Omega_{g,\sigma} := d\lambda - \pi^*\sigma\big\vert_{S^*_gW}
\end{align*}
The odd-symplectic form defined on the manifold $S^*_gW$ by another magnetic pair $(\tilde g,\tilde \sigma) \in \mathcal{M}(W) \times \Xi^2_{C_0}(W)$ has the following expression:
\begin{align}
    \tilde \Omega_{\tilde g,\tilde \sigma}= d(f_{\tilde g}\lambda)-\pi^*\tilde\sigma \big\vert_{S_g^*W}
\end{align}
where $f_{\tilde g}$ is the function defined in \eqref{def_of_the_function_f}.\\
\\
We now set $\delta>0$ to be the following:
\begin{align}\label{def_delta_mag_ineq_gen}
\delta:=  \max\{\vert \vert g - \tilde g\vert\vert_{C^3}, \vert\vert \sigma -\tilde \sigma\vert\vert_{C^3}\}
\end{align}
And claim that the following holds, for some real constant $C>0$:
\begin{align}\label{mag_ineq_gen_1}
    \big\vert\big\vert \Omega^{s}_{g,\sigma} -  \tilde \Omega^{s}_{\tilde g,\tilde \sigma}\big\vert\big\vert_{C^2} < C\delta
\end{align}
To prove the above claim, we begin by considering the following difference:
\begin{align}
     \Omega_{g,\sigma} -  \tilde \Omega_{\tilde g,\tilde \sigma} = -d\big( (f_{\tilde g} -1) \lambda\big) - \pi^*\big(\sigma -\tilde \sigma\big)
\end{align}
Taking the $C^2$-norms, we obtain the following:
\begin{align}
    \big\vert\big\vert \Omega_{g,\sigma} -  \tilde \Omega_{\tilde g,\tilde \sigma}\big\vert\big\vert_{C^2} \leq C_1 \big\vert\big\vert f_{\tilde g}-1\big\vert\big\vert_{C^3} + C_2 \big\vert\big\vert \sigma -\tilde \sigma\big\vert\big\vert_{C^2}
\end{align}
The loss of one derivative above comes from applying the exterior derivative $d$ to the the $1$-form $(f_{\tilde g} -1) \lambda$. It remains to estimate the difference $f_{\tilde g} -1$ in terms of $\delta$. Towards this end, we observe that on $S_g^*W$, $\vert\vert p\vert\vert^2_g =1$, therefore, the following holds locally:
\begin{align}
    f_{\tilde g}(q,p)= \sqrt{\tilde g^{ij}(q)p_ip_j }
\end{align}
And the above map is smooth on compact subsets of $W$. Hence, by the mean value theorem in the $C^3$-norm, the following holds for some real constant $C_3>0$:
\begin{align}
    \big\vert\big\vert  f_{\tilde g} -1\big\vert\big\vert_{C^3} \leq  C_3\;  \big\vert\big\vert g -\tilde g\big\vert\big\vert_{C^3}
\end{align}
Therefore, we can conclude that there exists a positive real constant $C'>0$ such that the following holds:
\begin{align}
    \big\vert\big\vert \Omega_{g,\sigma} -  \tilde \Omega_{\tilde g,\tilde \sigma}\big\vert\big\vert_{C^2} \leq C'\big(  \big\vert\big\vert \tilde g-g \big\vert\big\vert_{C^3} + \big\vert\big\vert \sigma -\tilde \sigma\big\vert\big\vert_{C^3}\big)
\end{align}
By the definition of $\delta$ in \eqref{def_delta_mag_ineq_gen}, we now obtain the following:
\begin{align}
    \big\vert\big\vert \Omega_{g,\sigma} -  \tilde \Omega_{\tilde g,\tilde \sigma}\big\vert\big\vert_{C^2} \leq 2C'\delta 
\end{align}
setting $C:= 2C'$, we obtain bound claimed in \eqref{mag_ineq_gen_1}.\\
\\
Therefore, we can conclude that there exists a $C^3$-neighborhood $\tilde{\mathcal{U}} \subset \mathcal{M}(W) \times \Xi^2_{C_0}(W)$ of the pair $(g,\sigma)$ with the property that the odd-symplectic form defined by any magnetic pair $(\tilde g,\tilde \sigma) \in \tilde{\mathcal{U}}$ is contained in the neighborhood $\mathcal{U} \subset \Xi^2_{\tilde C_0}(S_g^*W)$ of $\Omega_{g,\sigma}$ satisfying the conditions of Theorem \ref{systolic_ineq} where $\tilde C_0 \in H^2(S_g^*W,\mathbb{R})$ is the cohomology class represented by the odd-symplectic form $\Omega_{g,\sigma} \in \Omega^2(S_g^*W)$ (c.f the proof of Proposition \ref{mag_sys_action_general_propsition}).\\
\\
We now pick a pair $(\tilde g,\tilde \sigma) \in \tilde{\mathcal{U}}$ such that the following holds:
\begin{align}\label{mag_sys_gen_vol_proof_1}
    \text{Vol}_g(W) = \text{Vol}_{\tilde g}(W)
\end{align}
And denote by $\tilde \Omega_{\tilde g,\tilde \sigma}$ the odd-symplectic form defined on $S^*_gW$ by the magnetic pair $(\tilde g,\tilde \sigma)$.\\
\\
Using the assumption \eqref{mag_sys_gen_vol_proof_1}, in Proposition \ref{vol_mag_sys_general}, we see that the following holds:
\begin{align}\label{mag_sys_gen_vol_proof_2}
    \text{Vol}_{\Omega_{g,\sigma}}(\tilde \Omega_{\tilde g,\tilde \sigma})=0
\end{align}
And since $(\tilde g,\tilde \sigma) \in \tilde{\mathcal{U}}$, by the discussion above, $\tilde \Omega_{\tilde g,\tilde \sigma} \in \mathcal{U}$. Using \eqref{mag_sys_gen_vol_proof_2}, we can conclude from Theorem \ref{systolic_ineq} that the following holds:
\begin{align}
    \inf\mathcal{A}_{\tilde \Omega_{\tilde g,\tilde \sigma}} \leq 0
\end{align}
By Proposition \ref{mag_sys_action_general_propsition}, the action of a loop $\gamma \in \chi(\tilde g,\tilde \sigma)$ has the following expression:
\begin{align}
    \text{length}_{\tilde g}(\gamma) - \int_{\pi_*\Gamma_{ \gamma}} \tilde\sigma - \text{length}_{g}(\gamma) + \int_{\gamma_0}\eta
\end{align}
where $\eta \in \Omega^1(W)$ is any $1$-form such that: $\sigma= \tilde \sigma + d\eta$.\\
\\
Taking the infumum over all loops in $\chi(\tilde g,\tilde \sigma)$, we obtain the following, as required:
\begin{align}\label{mag_sys_general_proof_1}
\inf_{\gamma \in \chi(\tilde g,\tilde \sigma)}\Big(\text{length}_{\tilde g}(\gamma)-\int_{\pi_*\Gamma_{\gamma}} \tilde \sigma\Big) \leq \text{length}_{g}(\gamma_0) - \int_{\gamma_0}\eta
\end{align}
with equality holding if and only if the pair $(\tilde g,\tilde \sigma)$ is Zoll.
\section{Magnetic dynamics on K\"ahler manifolds}\label{mag_dynamics_on_kahler_section}
To make this work self-contained, we begin with some relevant preliminaries about K\"ahler manifolds. A K\"ahler manifold is defined to be a triple $(M, \sigma_0, J)$ where $M$ is a smooth manifold, $\sigma_0 \in \Omega^2(M)$ a symplectic form and $J$ an integrable almost complex structure that is compatible with $\sigma_0$. Recall that an almost complex structure $J$ is said to be integrable if it induces complex charts on $M$ and it is compatible with $\sigma_0$ if the pair $(\sigma_0, J)$ determines a Riemannian metric on $M$ in the following way:
\begin{align}\label{kahler_metric}
 g_0(u,v):=   \sigma_0(u,Jv)  \;\;\; \forall\; u,v \in TM.
\end{align}
The Riemannian metric determined by the symplectic form and complex structure corresponding to a K\"ahler manifold will be called the associated K\"ahler metric. The holomorphic sectional curvature of the K\"ahler metric $g_0$ associated with a given K\"ahler manifold $(M,\sigma_0,J)$ is defined to be:
\begin{align}\label{def_of_hol_sectional_curvature}
&\text{Hol}: TM \backslash \{0\} \rightarrow \mathbb{R}\\\nonumber
    (q,v) &\longmapsto \frac{g_0(v, R(v,Jv)Jv)}{\vert\vert v\vert\vert_{g_0}^4}
\end{align}
where $R$ is the Riemannian curvature tensor of $g_0$. We observe here that if $u,v \in TM$ span the same complex line, then $\text{Hol}(u)= \text{Hol}(v)$. Therefore, the holomorphic sectional curvature can be considered a function from the complex projectivization $\mathbb{P}(TM)$ to the real line. In this section, when working with a K\"ahler manifold, we will always make use of the convention of denoting its dimension by $2n$.\\
\\
A recent result of J. Bimmermann \cite{johanna_negative_curvature} asserts the existence of Zoll magnetic systems on K\"ahler manifolds of constant holomorphic sectional curvature (c.f Theorem \ref{zollness_of_mag_systems_kahler}). We fix some notational conventions before we state the relevant version of Bimmermann's result here. Given a K\"ahler manifold $(M, \sigma_0, J)$ the following vector field generates the fiber-wise rotation on $T^*M$:
\begin{align}\label{angular_vf_cpn_def}
V_{(q,p)}:= (Jp)^{\mathcal{V}}
\end{align}
where the $\mathcal{V}$ in the superscript above denotes the lift of the relevant vector to the vertical distribution in $TT^*M$. We will denote by $\tau$ the dual to the above vector field with respect to the Sasaki metric. 
 \begin{align}
        \mathbb{D}_rW:=\{(q,p) \in TW\;\big\vert \;\ \vert\vert p\vert\vert^2_{g} = r\} \;\; \text{for $r \in \mathbb{R}_{+}$}
        \end{align}

In addition to this, to stress the strength of the magnetic system being considered, given some $s \in \mathbb{R}$, we will denote by $X^{s}_{g_0,\sigma} $ the vector field generating the magnetic geodesic flow of the magnetic pair $(g_0,\sigma_0)$ on the Riemannain manifold $(M,g_0)$ of strength $s$. The relevant version of the result of Bimmermann referenced above is the following:
\begin{theorem}\cite[Theorem A \& Theorem B]{johanna_negative_curvature}\label{zollness_of_mag_systems_kahler}
    Let $(M,\sigma_0, J)$ be a K\"ahler manifold of constant holomorphic sectional curvature $\kappa$. Then, there exists a monotonically decreasing function $a:\mathbb{R} \rightarrow \mathbb{R}_{>0}$ and the following symplectomorphism $\Psi$: 
\begin{align}\label{Bimmermann_symplectomorphism}
        \Psi: \bigg(\mathbb{D}_{1}M, d\lambda-s\pi^*\sigma_0\bigg) \rightarrow \bigg(\mathbb{D}_{ a(1)}M, \frac{d\tau}{2} -s\pi^*\sigma_0\bigg) \;\;\; \text{for all $s$ such that $s^2+ \kappa >0$}
    \end{align}
    whose composition with the inclusion $\iota: \mathbb{D}_{a(1)} \hookrightarrow \mathbb{D}_{1}$ is isotopic to the identity and the vector field $\Psi_*(X^{s}_{g_0,\sigma_0})$ is generated by the fiber-wise rotation: \\$e^{2\pi Jt}: T^*M \rightarrow T^*M$ up to time reparametrization. The $1$-form $\tau$ is the Sasaki-dual to the vector field defined in \eqref{angular_vf_cpn_def} and the function $a$ has the following form:
\begin{align}\label{def_of_a_mag_zoll}
        a(1) = \begin{cases}
            \bigg(\frac{2}{\vert\kappa\vert} \big(\sqrt{s^2+\kappa}-s \big)  \bigg)^{\frac{1}{2}}\;\;\; \text{if $\kappa \neq 0$} \\
            \\
            \frac{1}{\sqrt{s}} \;\;\; \text{if $\kappa =0$}
            \end{cases}
    \end{align}
\end{theorem}
In the theorem above, if the K\"ahler manifold $M$ has constant negative holomorphic sectional curvature, then there exists a finite $s_c \in \mathbb{R}$ such that $s_c^2= -\kappa$ \cite{johanna_negative_curvature}. It is well known that at this strength, the magnetic geodesic flow can be conjugated to a Horocycle flow, see for instance \cite{johanna_negative_curvature}, \cite{CieliebakFrauenfelderPaternain2010} and \cite{paternain_horocycle}.\\
\\
By the discussion in Section \ref{mag_dyn_riem_section}, the odd-symplectic forms defined on the manifold $S^*_{g_0}M$ by Bimmerman's Zoll magnetic pairs $(g_0,\sigma_0)$ in Theorem \ref{zollness_of_mag_systems_kahler} above are Zoll. Therefore, by Theorem \ref{systolic_ineq}, these odd-symplectic forms satisfy a local systolic inequality. The purpose of this section is calculate this systolic inequality. We begin by fixing some notational conventions,
for the remainder of this section, we will fix a closed K\"ahler manifold $(M, \sigma_0,J)$ of dimension $2n$ for $n\geq 1$, with constant holomorphic sectional curvature $\kappa$ and denote the associated K\"ahler metric by $g_0$. As observed earlier, the pair $(M,g_0)$ determines a Riemannian manifold; we will use this fact in the sequel without any further qualification. In this setting, when talking about a Zoll magnetic pair $(g_0,\sigma_0)$ we will always assume that the resulting magnetic system is being considered at a fixed strength $s$ such that $s^2+\kappa >0$.\\
\\ 
By Bimmermann's theorem, on a K\"ahler manifold $(M,\sigma_0,J)$ of constant holomorphic sectional curvature $\kappa$, the magnetic geodesic flow of the magnetic pair $(g_0,\sigma_0)$ can be conjugated to the fiber-wise rotation $e^{2\pi Jt}: T^*M \rightarrow T^*M$ at appropriate strengths. We now claim that this fiber-wise rotation is an isometry for the Sasaki metric on $T^*M$. Indeed, observe that upon fixing a $q_0 \in M$ and $p_0 \in T_{q_0}^*M$, the horizontal lift of an arbitrary $u \in T^*M$ is defined by:
\begin{align}
   u^{\mathcal{H}}:= \frac{d}{dt}\Bigg|_{t=0} (q(t),v(t))
\end{align}
where $(q(t),v(t))$ is a path in $T^*M$ satisfying the following conditions:
\begin{align}
    q(0) = q_0& \;\;\;\;\;\; \dot{q}(0)=u \\\nonumber v(0) =p_0& \;\;\;\; \nabla_{t}v(0) =0
\end{align}
and $\nabla_{t}$ denotes the covarinat derivative of the K\"ahler metric $g_0$ along the curve $\pi \circ v(t)$. From the definition of the curve $(q(t),v(t))$, it is clear that 
$u^{\mathcal{H}} \in \mathcal{H}(p_0)$. Moreover, using the explicit 
expression for the flow $\Phi_{g_0,\sigma_0}^s$, we compute:
\begin{align}
    d \Phi_{g_0,\sigma_0}^s\big|_{(q(0),v(0)}  u^{\mathcal{H}} = \frac{d}{dt}\Bigg|_{(q(0),v(0))} \Phi_{g_0,\sigma_0}^s(q(t),v(t)) = \frac{d}{dt}\Bigg|_{(q(0),v(0))} (q(t), e^{2\pi J s}v(t))
\end{align}
Since the complex structure $J$ of a K\"ahler manifold is parallel with respect 
to the K\"ahler metric, that is, $\nabla J \equiv 0$, the endomorphism 
$e^{2\pi J s}$ is also parallel. As a consequence, the following holds:
\begin{align}
    \nabla_{t}\Big(e^{2\pi J s}v(t)\Big)
    =
    \Big(\nabla_t e^{2\pi J s}\Big)v(t)
    +
    e^{2\pi J s}\nabla_t v(t)
    =
    0.
\end{align}
Therefore, the curve 
\begin{align}
    t \mapsto (q(t), e^{2\pi J s}v(t))
\end{align}
represents a horizontal lift at the point 
$\Phi_{g_0,\sigma_0}^s\vert_{q(0),v(0)}$. Hence,
\begin{align}
    d \Phi_{g_0,\sigma_0}^s\big|_{(q(0),v(0))}  u^{\mathcal{H}} 
    \in 
    \mathcal{H}\big({\Phi_{g_0,\sigma_0}^s(q(0),v(0))}\big)
\end{align}
In particular, the flow $\Phi_{g_0,\sigma_0}^s$ preserves the horizontal distribution. Therefore, with respect to the orthogonal splitting $TT^*M \cong \mathcal{H} \bigoplus \mathcal{V}$ the map $d\Phi_{g_0,\sigma_0}^s$ can be written in the following way:
\begin{align}\label{rotation_is_killing}
    d\Phi_{g_0,\sigma_0}^s = \begin{bmatrix}
    Id & 0\\
    0 & e^{2\pi J s}
\end{bmatrix}
\end{align}
It is worth noting here that in \cite[Proposition 2.5]{johanna_negative_curvature} it is observed that upon denoting by $\eta:= J^*\lambda$, the $2$-form $d\tau$ has the following expression on $T^*M$:
\begin{align}\label{dtau_decomop}
    d\tau = +2\sigma_{0}^{\mathcal{V}} - \frac{\kappa}{2}\Bigg[\lambda \wedge \eta +  2H_{g_0}\pi^*(\sigma_{0})\Bigg]
\end{align}
where $\sigma_{\text{0}}^{\mathcal{V}}$ is the vertical lift of $\sigma_{\text{0}}$ and $H_{g_0}: T^*M \rightarrow \mathbb{R}$ is the metric Hamiltonian. We note here that the formula in \eqref{dtau_decomop} differs from the one appearing in \cite{johanna_negative_curvature} by a sign due to a difference in conventions with respect to the definition of curvature used.\\
\\ 
Since the Zoll magnetic pair $(g_0,\sigma_0)$ on the Riemannian manifold $(M,g_0)$ defines a Zoll odd-symplectic form on the manifold $S_{g_0}^*M$, after a picking an $S^1$-bundle associated with this Zoll odd-symplectic form, we can define a Zoll polynomial as in Section \ref{zoll_poly_section}. We will refer to the resulting Zoll polynomial as the Zoll polynomial of the magnetic pair $(g_0,\sigma_0)$. The following proposition furnishes a connivent way to compute this Zoll polynomial in Bimmermann's examples: 
 \begin{proposition}\label{Zoll_poly_cpn_proposition}
Let $(M, J, \sigma_0)$ be a closed, connected K\"ahler manifold such that the associated K\"ahler metric $g_0$ has constant holomorphic sectional curvature $\kappa$. Then, the Zoll polynomial of the magnetic pair $(g_0,\sigma_0)$ has the following expression:
\begin{align}\label{cpn_Zoll_polynomial_final_form}
 P(\tilde A) = \begin{cases}
   n!\pi^{n-1}[a(1)]^{2n-2}
    \;\text{Vol}_{g_0}(M)\; \big(A(\tilde A)C(\tilde A)\big)^{n-1}\bigg(C(\tilde A) -\frac{B(\tilde A)}{n} \bigg) \;\;\;&\text{if $\kappa \neq 0$} 
    \\
    \\
       (-s)^n A(\tilde A)^{n-1}\; n!\pi^{n-1}[a(1)]^{2n-2}
    \;\text{Vol}_{g_0}(M)\;\;\; &\text{if $\kappa =0$}
 \end{cases}  
 \end{align}
 where $s \in \mathbb{R}$ is any constant satisfying the condition of Theorem \ref{zollness_of_mag_systems_kahler}, $a$ is the function defined in that theorem and :
 \begin{align}\label{def_of_K_tilde}
   A(\tilde A)=1+\frac{\tilde A}{\pi a^2(1)},\qquad
B(\tilde A)=\frac{\kappa\bigl(\tilde A+\pi a^2(1)\bigr)}{4\pi a^2(1)},\qquad
C(\tilde A)=\frac{\kappa\bigl(\tilde A+\pi a^2(1)\bigr)}{4\pi}-s
 \end{align}
\end{proposition}
\begin{proof}
Since the Zoll polynomial is invariant under diffeomorphisms isotopic to the identity, due to Theorem \ref{zollness_of_mag_systems_kahler}, we consider the following Zoll odd-symplectic form:
\begin{align}\label{zoll_poly_mag_os_def}
    \Omega_0 = \bigg(\frac{d\tau}{2} - s\pi^*\sigma_0 \bigg) \bigg|_{\partial(\mathbb{D}_{a(1)}M)}
\end{align}
In the remainder of this proof, the strength $s$ of the magnetic system determined by the pair $(g_0,\sigma_0)$ is fixed to be one that satisfies the condition $s^2+\kappa>0$ from Theorem \ref{zollness_of_mag_systems_kahler}.\\
\\
By Theorem \ref{zollness_of_mag_systems_kahler}, the leaves of the characteristic foliation of the above odd-symplectic form are the fibers of the following $S^1$-bundle:
\begin{align}\label{s1_in_zoll_computation_def}
    \partial(\mathbb{D}_{a(1)}M) \xrightarrow{p} \mathbb{P}(T^*M)
\end{align}
It is also clear from Theorem \ref{zollness_of_mag_systems_kahler} the fibers of the above 
$S^1$-bundle are the fibers of a liner rescaling of the vector field \eqref{angular_vf_cpn_def}. In addition, as shown in \eqref{rotation_is_killing} the flow of this vector field leaves the Sasaki metric invariant. As a result, the connection $1$-form for the $S^1$-bundle \eqref{s1_in_zoll_computation_def} is the following:
\begin{align}
    \frac{\tau}{2\pi a^2(1)}
\end{align}
Therefore, there exists a closed $2$-form $\nu \in \Omega^2(\mathbb{P}(T^*M))$ such that $p^*\nu =  \frac{d\tau}{2\pi a^2(1)}$. The cohomology class $e_0 \in H^2(\mathbb{P}(T^*M), \mathbb{R})$ represented by $\nu$ is $-1$ times the real Euler class of the bundle \eqref{s1_in_zoll_computation_def}. Moreover, there exists a symplectic form $\omega_0 \in \Omega^2(\mathbb{P}(T^*M))$ that is uniquely determined by the relation $p^*\omega_0 = \Omega_0$ and we will denote by $c_0 \in H^2(\mathbb{P}(T^*M), \mathbb{R})$ the cohomology class represented by $\omega_0$.\\
\\
By \eqref{zoll_poly_mag_os_def}, the following cohomological equation, holds:
\begin{align}\label{c_0_cpn_1}
    c_0 = \pi a^2(1) e_0-{s}\pi^*[\sigma_{0}]
\end{align}
where $\pi:\mathbb{P}(T^*M) \rightarrow M$ is the standard foot-point projection and $[\sigma_{0}] \in H^2(M,\mathbb{R})$ is the cohomology class represented by the K\"ahler form $\sigma_0$.\\
\\
The Zoll polynomial we are trying to compute is then the following integral:
\begin{align}\label{cpn_kronocker_sum_0}
    P(\tilde A) = \int^{\tilde A}_0  \langle (c_0+te_0)^{2n-1},[\mathbb{P}(T^*M)]\rangle\; dt \;\;\; \text{for some $\tilde A \in \mathbb{R}$.} 
        \end{align}
where the orientation on the manifold $\mathbb{P}(T^*M)$ is that induced by the symplectic form $\omega_0$.\\
\\
By \eqref{c_0_cpn_1}, the sum $c_0 + te_0$ takes the form: 
\begin{align}\label{cpn_kronocker_sum_1}
    (t+a^2(1)\pi)e_0 - {s}\pi^*[\sigma_{0}]
\end{align}
For each $p \in \mathbb{P}(T^*M)$ the tangent space $T_p\mathbb{P}(T^*M)$ has the following direct sum decomposition:
\begin{align}
    T_p\mathbb{P}(T^*M) \cong \mathcal{H}(p) \bigoplus \tilde{\mathcal{V}}(p)
\end{align}
where $\mathcal{H}(p) \cong T_pM \subset \partial(\mathbb{D}_{a(1)}M)$ is (the restriction of) the horizontal distribution of $T_pT^*M$ and $\tilde{\mathcal{V}}(p)$ is the quotient of (the restriction of) the vertical distribution of $T_pT^*M$ by the relevant $S^1$-action. Using this information and \eqref{dtau_decomop}, we see that the cohomological equation \eqref{cpn_kronocker_sum_1} has a representative $\Delta$ on $\mathbb{P}(T^*M)$ given by:
\begin{align}\label{cpn_kronocker_sum_2}
 \Delta(t) =  \Bigg(\frac{t}{\pi a^2(1)}+1\Bigg)\sigma^{\mathcal{V}}_{0} -\Bigg(\frac{\kappa(t+a^2(1)\pi)}{4\pi a^2(1)} \Bigg) \lambda\wedge \eta + \Bigg(\frac{\kappa(t+a^2(1)\pi)}{4\pi}-s \Bigg)\sigma^{\mathcal{H}}_0
\end{align}
where $\sigma^{\mathcal{H}}_{0}$ is the horizontal lift of $\sigma_{0}$ that is, $\sigma^{\mathcal{H}}_{0} := \pi^*\sigma_{0}$. In order to see that formula  \eqref{cpn_kronocker_sum_2} follows from \eqref{dtau_decomop} and \eqref{cpn_kronocker_sum_1}, we observe that the following holds:
\begin{align}
    H^{-1}_{g_0}\Big( \frac{a^2(1)}{2}\Big) = \mathbb{D}_{a(1)}M
\end{align}
In the sequel, we compute the integral \eqref{cpn_kronocker_sum_0}. We will do so in two separate parts: by first considering the case $\kappa \neq 0$ and then looking at the situation with $\kappa =0$.\\
\\
\textbf{\underline{Case 1: $\kappa \neq 0$}}\\
\\
We want to evaluate $\Delta^{2n-1}$ with $\kappa \neq
 0$ which has the following binomial expansion:
\begin{align}\label{Z_2n-1_cpn}
   \Delta^{2n-1} (t) = \sum^{2n-1}_{i=0} & \binom{2n-1}{i} \Big( C(t) \sigma^{\mathcal{H}}_0 - B(t)\lambda\wedge\eta  \Big)^{i}  \wedge A(t)^{2n-1-i} \big(\sigma^{\mathcal{V}}_{0}\big)^{2n-1-i}
\end{align}
where 
\begin{align}
    A(t)=1+\frac{t}{\pi a^2(1)},\qquad
B(t)=\frac{\kappa\bigl(t+\pi a^2(1)\bigr)}{4\pi a^2(1)},\qquad
C(t)=\frac{\kappa\bigl(t+\pi a^2(1)\bigr)}{4\pi}-s
\end{align}
The first two summands in  \eqref{Z_2n-1_cpn} act only on the horizontal distribution in $T(\mathbb{P}(T^*M))$ and the third only on the vertical one, therefore the only summand in \eqref{Z_2n-1_cpn} that is non-zero is the one with $i=n$ because $\mathbb{P}(T^*M)$ has exactly $2n$ horizontal and $2n-1$ vertical directions. Therefore, the $4n-2$-form \eqref{Z_2n-1_cpn} takes the form: \begin{align}\label{cpn_kronocker_sum_3}
 \Delta^{2n-1} (t) = \sum^{2n-1}_{i=0} & \binom{2n-1}{i} \Big( C(t) \sigma^{\mathcal{H}}_0 - B(t)\lambda\wedge\eta  \Big)^{n}  \wedge A(t)^{n-1} \big(\sigma^{\mathcal{V}}_{0}\big)^{2n-1-i}
\end{align}
We begin by considering only the horizontal part and make use of the following notation:
\begin{align}\label{H_def_without_power}
    \mathfrak{H}:=  C(t) \sigma^{\mathcal{H}}_0 - B(t)\lambda\wedge\eta  
\end{align}
Therefore the $2n$-form $\mathfrak{H}^n$ can be written as:
\begin{align}
    \mathfrak{H}^n = \sum^n_{i=0} (-1)^{i}\; C(t)^{n-i}\;B(t)^{i}\;(\lambda\wedge\eta)^i \wedge \Big(\sigma^{\mathcal{H}}_0\Big)^{n-i}
\end{align}
Since $\lambda \wedge \lambda =0$, only the terms $i=0,1$ survive in the above expression. Therefore, we obtain:
\begin{align}\label{CPn_zoll_poly_Hn_def}
     \mathfrak{H}^n = C(t)^n &\big(\sigma^{\mathcal{H}}_0\big)^{n} - C(t)^{n-1}\;B(t)\;(\lambda\wedge\eta) \wedge \big(\sigma^{\mathcal{H}}_0\big)^{n-1}  \\\nonumber = 
     &C(t)^{n-1} \big(\sigma^{\mathcal{H}}_0\big)^{n} \bigg(C(t) -\frac{B(t)}{n} \bigg)
\end{align}
where the final equality is a consequence of the following fact: given some point $(q,p) \in \mathbb{P}(T^* M)$ and a unitary basis $\{v_i, Jv_i\}^{n}_{i=1}$ for $H(p) \cong T_qM$ with the property that $v_1$ is the metric dual to $p$, the following is true: 
\begin{align}
    \big(\lambda \wedge \eta \wedge (\sigma^{\mathcal{H}}_{0})^{n-1}\big)|_{(q,p)}(&v_1,Jv_1,\cdots,v_n,Jv_n) = \frac{(n-1)!}{n!}(\sigma^{\mathcal{H}}_{0})|_{(q,p)}^n(v_1,Jv_1,\cdots,v_n,Jv_n) \\\nonumber = &\frac{1}{n}(\sigma^{\mathcal{H}}_{0})|_{(q,p)}^n(v_1,Jv_1,\cdots,v_n,Jv_n)
\end{align}
the first equality above is a consequence of the fact that given the unitary basis $\{u_i\}^{2n}_{i=0}$ of $ T\mathbb{P}(T^* M)$, $[\lambda \wedge \eta](u_i,u_j) \neq 0$ if and only if $\{u_i,u_j\} = \{u_1, Ju_1\}$.\\
\\
Using the Fubini-Tonelli theorem, the Kronocker pairing $\langle (c_0+te_0)^{2n-1},[\mathbb{P}(T^*M)]\rangle$ in \eqref{cpn_kronocker_sum_0} takes the form:
\begin{align}\label{cpn_second_last_int}
  \big(A(t)C(t)\big)^{n-1}\bigg(C(t) -\frac{B(t)}{n} \bigg)  \int_{\mathbb{P}(T^*M)} (\sigma^{\mathcal{H}}_{0}\big)^n \wedge (\sigma^{\mathcal{V}}_{0})^{n-1}
\end{align}

Upon integration, \eqref{cpn_second_last_int} then becomes:
\begin{align}\label{zoll_poly_computation_final_step}
  \langle (c_0+te_0)^{2n-1},[\mathbb{P}(T^*M)]\rangle = n!\pi^{n-1}[a(1)]^{2n-2}
    \;\text{Vol}_{g_0}(M)\; \big(A(t)C(t)\big)^{n-1}\bigg(C(t) -\frac{B(t)}{n} \bigg) 
\end{align}
From which we then obtain the following expression for the Zoll polynomial $P(\tilde A)$:
\begin{align}
    P(\tilde A)= n!\pi^{n-1}[a(1)]^{2n-2}
    \;\text{Vol}_{g_0}(M)\; \big(A(\tilde A)C(\tilde A)\big)^{n-1}\bigg(C(\tilde A) -\frac{B(\tilde A)}{n} \bigg) 
\end{align}
\textbf{\underline{Case 2: $\kappa =0$}}\\
\\
Taking $\kappa =0$ in \eqref{cpn_kronocker_sum_2} we obtain:
\begin{align}
    \Delta(t) =  A(t)\sigma^{\mathcal{V}}_{0} -s \sigma^{\mathcal{H}}_0
\end{align}
Therefore, $\Delta^{2n-1}(t)$ has the following expression:
\begin{align}
    \Delta^{2n-1}(t) = (-s)^n A(t)^{n-1}\big(\sigma^{\mathcal{V}}_{0}\big)^{n-1} \wedge \big(\sigma^{\mathcal{H}}_{0}\big)^{n}
\end{align}
From which we obtain the following: 
\begin{align}
    P(\tilde A)
 =  (-s)^n A(\tilde A)^{n-1}\; n!\pi^{n-1}[a(1)]^{2n-2}
    \;\text{Vol}_{g_0}(M)
\end{align}
\end{proof}
Recall that the magnetic length of a magnetic geodesic $\gamma_0$ of the magnetic system determined by the Zoll pair $(g_0,\sigma_0)$ was defined in \eqref{mag_length_functional_def} to be the following:
\begin{align}\label{mag_length_kahler_def_Zoll}
    l^{g_0,\sigma_0}_{\text{mag}}(\gamma_{0}) = \text{length}_{g_0}(\gamma_{0}) - \int_{\mathbb{D}_{\gamma_{0}} } \sigma_{0} 
\end{align}
We now describe in more detail how the capping surface $\mathbb{D}_{\gamma_0}$ is obtained. We begin by observing that the magnetic geodesic flow of the Zoll magnetic pair $(g_0,\sigma_0)$ is contractible in the disk bundle $\mathbb{D}_1M$ by Theorem \ref{zollness_of_mag_systems_kahler}. And since the magnetic geodesics of the Zoll pair $(g_0,\sigma_0)$ are precisely the curves obtained by projecting to $M$ these contractible trajectories in $\mathbb{D}_1M$, the magnetic geodesics of the Zoll pair $(g_0,\sigma_0)$ are also null-homotopic in $M$. We now fix a universal covering $\tilde p: \tilde M \rightarrow M$ and pick a magnetic geodesic $\gamma_0 \in \chi(g_0,\sigma_0)$. Let $\tilde x \in \tilde M$ be a point such that $\tilde p(\tilde x) =\gamma_0(0)$. The magnetic geodesic $\gamma_0$ then has a unique lift $\tilde \gamma_0 \subset \tilde M$ such that $\tilde \gamma_0(0)=\tilde x$ and since the magnetic geodesic $\gamma_0$ is contractible, the curve $\tilde \gamma_0$ is closed. In addition, the homology class $[\tilde \gamma_0] \in H_1(\tilde M,\mathbb{Z})$ represented by the curve $\tilde \gamma_0$ is the trivial or zero class. This is because the universal cover $\tilde M$ is simply connected and by the Hurewicz theorem, the following holds:
\begin{align}
     H_1(\tilde M,\mathbb{Z}) \cong \pi_1(\tilde M)_{\text{ab}} \cong \{e\}
\end{align}
Therefore, by \cite[Proposition 3.4.8]{Geiges_book}, the curve $\tilde \gamma_0$ admits a capping disk $\mathbb{D}_{\tilde \gamma_0}$. The capping surface $\mathbb{D}_{\gamma_0}$ in \eqref{mag_length_kahler_def_Zoll} is then obtained by just projecting the disk $\mathbb{D}_{\tilde \gamma_0}$ to $M$. If $\kappa \leq 0$ then, $\pi_2(\tilde M) = \{e\}$ as a result, the following integral is independent of the choice of capping disk $\mathbb{D}_{\tilde \gamma_0}$
\begin{align}
    \int_{\mathbb{D}_{\tilde \gamma_0}} \tilde p^*\sigma_0
\end{align}
However if $\kappa >0$ then $\pi_2(\tilde M) \cong \mathbb{Z}$. So, in this case, we need to specify which disk we choose as a capping disk for the the lifted curve $\tilde \gamma_0$, the disk we wish to use is obtained in the following way: upon lifting a given magnetic geodesic $\gamma_0$ to the universal cover $\Tilde{M} \cong \mathbb{CP}^n$ of $M$ in the way specified above, as observed in \cite[Pg. 27]{johanna_projective}, there exists an immersion $\mathbb{CP}^1 \hookrightarrow \mathbb{CP}^n$ with the property that this projectivised complex line completely contains the lift of the chosen magnetic geodesic $\gamma_0$. In addition to this, there is a well-defined normal direction to the lifted magnetic geodesic $\tilde \gamma_0$ determined by the orientation on $\mathbb{CP}^1$ induced by the lift of $\sigma_0$. We now define the capping disk $\mathbb{D}_{\tilde \gamma_0}$ of the lifted curve $\tilde \gamma_0$ to be the one for which this normal vector points inward.\\
\\
The following proposition furnishes a convenient expression for the magnetic length considered in \eqref{mag_length_kahler_def_Zoll}.
\begin{proposition}\label{action_of_fubini_study_mag_system}
Let $(M,\sigma_0,J)$ be a closed and connected K\"ahler manifold of constant holomorphic sectional curvature $\kappa$ with associated K\"ahler metric $g_0$. Then, given any magnetic geodesic $\gamma_0$ of the Zoll magnetic pair $(g_0,\sigma_0)$, its magnetic length $l^{g_0,\sigma_0}_{\text{mag}}(\gamma_0)$ has the following expression:
\begin{align}
   l^{g_0,\sigma_0}_{\text{mag}}(\gamma_{0}) =  \pi a(1)^2
    \end{align}
    where $a$ is the function defined in
\eqref{def_of_a_mag_zoll}. 
\end{proposition}
\begin{proof}
\textbf{\underline{Case 1: $\kappa >0$}:}\\
\\
As observed earlier, we know from \cite[Pg. 27]{johanna_projective} that given a magnetic geodesic $\gamma_{0} \in \chi(g_{0}, \sigma_{0})$, there exists an immersion $\mathbb{CP}^1 \hookrightarrow M$ with the property that $\gamma_{0} \subset \mathbb{CP}^1$. Therefore, for any given magnetic geodesic $\gamma_{0}$, we can carry out the computations in the proof completely within such an embedded $\mathbb{CP}^1$.\\
\\
From \eqref{mag_length_functional_def}, the magnetic length of $\gamma_0$ is the following:
\begin{align}
    l^{g_0,\sigma_0}_{\text{mag}}(\gamma_{0}) = \text{length}_{g_0}(\gamma_{0}) - \int_{\mathbb{D}_{\gamma_{0}} } s\sigma_{0} 
\end{align}
The constant $s$ in the above equation is omitted in \eqref{mag_length_functional_def} and \eqref{mag_length_kahler_def_Zoll}, to make for a better presentation of the formula, since the strength of the magnetic system being considered there is clear. However, we reintroduce it here since it will play a crucial role in the calculations that follow.\\
\\
In \cite[Pg. 27]{johanna_projective}, the following equality was established:
\begin{align}\label{paradigm_1_lenght_FS}
\text{length}_{g_0}(\gamma_{0}) = \frac{2\pi}{\sqrt{s^2+ \kappa}}
\end{align}
To compute the integral $\int_{\mathbb{D}_{\gamma_0}} s\sigma_{\text{0}}$, we fix the universal cover $\tilde p: \mathbb{CP}^n \rightarrow M$ of $M$ and observe that the following holds:
\begin{align}
    \int_{\mathbb{D}_{\gamma_0}} s\sigma_{\text{0}} =\int_{\mathbb{D}_{\tilde \gamma_0}} s\tilde p^*\sigma_{\text{0}}
\end{align}
So, we loose no information by working on the universal cover. We will use polar coordinates $(\theta, \phi)$ on $\mathbb{CP}^1 \cong S^2 \hookrightarrow \mathbb{R}^3$. These polar coordinates can be described as follows: $ \theta \in (0, \pi)$ is the angle a vector in the standard $\mathbb{R}^3$ makes with the vertical $\text{z}$-axis and $\phi \in \mathbb{R} \backslash 2\pi \mathbb{Z}$ is the angular coordinate of the projection of the same vector to the $(x,y)$-plane. In these coordinates, the magnetic geodesic $\gamma_0$ has the following expression:
\begin{align}\label{mag_geodesic_cpn_eqn}
    \gamma_0(\theta) =\frac{s}{\kappa} \sin(\sqrt{\kappa}\theta) 
\end{align}
The K\"ahler form $\sigma_0$ has the following expression:
\begin{align}\label{kahler_form_pos_curvature}
    \sigma_0 = \frac{1}{\kappa}\sin{(\sqrt{\kappa}\theta)} d\theta \wedge d\phi
\end{align}
Therefore, the desired integral has the following expression:
\begin{align}\label{paradigm_1_action_2}
    s\int_{\mathbb{D}_{\gamma_0}} \sigma_0 =  \frac{s}{\kappa}\int^{2\pi}_0 \Big(\int^{\theta}_0  \sin (\sqrt{\kappa}t) \; dt \Big) d\phi = \frac{2\pi s}{\kappa}[1-\cos (\sqrt{\kappa}\theta)]
\end{align}
We know from the proof of \cite[Lemma 5.4]{johanna_projective} that the magnetic geodesic $\gamma_0 \in \chi(g_0,\sigma_0)$ has constant geodesic curvature equal to the strength $s$. In addition, it was proven in \cite[(10)]{johanna_projective} that the geodesic curvature $\kappa_{g_0}$ satisfies the following:
\begin{align}
    \kappa_{g_0} = \frac{\sqrt{\kappa}}{\tan (\sqrt{\kappa}\theta)}
\end{align}
Combining these two facts, we obtain the following:
\begin{align}\label{kahler_action_zoll_tan}
    \tan (\sqrt{\kappa}\theta) = \frac{\sqrt{\kappa}}{s}
\end{align}
combining \eqref{kahler_action_zoll_tan} with the identity $\cos(\sqrt{\kappa}\theta) = \frac{1}{\sqrt{1+\tan^2(\sqrt{\kappa}\theta)}}$ we see that we can rewrite \eqref{paradigm_1_action_2} in the following form:
\begin{align}\label{paradigm_1action_3}
s\int_{\mathbb{D}_{\gamma_0}} \sigma_{\text{0}} = \frac{2\pi}{\kappa}\Bigg[s -\frac{s^2}{\sqrt{s^2+\kappa}}\Bigg]
\end{align}  
Subtracting \eqref{paradigm_1action_3} from \eqref{paradigm_1_lenght_FS}, we obtain the following expression for $l^{g_0,\sigma_0}_{\text{mag}}(\gamma_{0})$:
\begin{align}\label{paradigm_1action_4}
    2\pi\Bigg[ \frac{1}{\sqrt{s^{2}+\kappa}} - \frac{s}{\kappa} + \frac{s^2}{\kappa\sqrt{s^{2}+\kappa}} \Bigg] = \pi \Bigg[ \frac{2}{\pi}\Big( \sqrt{\kappa +s^2} -s \Big)\Bigg] = \pi a^2(1)
\end{align}
The final equality in \eqref{paradigm_1action_4} above follows directly from the definition of the function $a$ in \eqref{def_of_a_mag_zoll}.\\
\\
\textbf{\underline{Case 2: $\kappa <0$}}\\
\\
By \cite{johanna_negative_curvature}, we know that the magnetic geodesic $\gamma_0 \in \chi(g_0,\sigma_0)$ is completely contained in an immersed surface $\Sigma_{g_0} \hookrightarrow M$ such that $\Sigma_{g_0}$ has constant Gaussian curvature equal to $\kappa$. As a consequence, we can carry out the following computation of the magnetic action evaluated at the magnetic geodesic $\gamma_0$ inside this surface $\Sigma_{g_0}$.\\
\\
The following formula was established in \cite{johanna_negative_curvature}:
\begin{align}\label{length_hyperbolic}
    \text{length}_{g_0}(\gamma_{0}) = \frac{2\pi}{\sqrt{s^2+ \kappa }}
\end{align}
To compute the integral $\int_{\mathbb{D}_{\gamma_0}} s\sigma_0$ we begin by fixing the universal cover $\tilde p: \mathbb{CH}^n \rightarrow M$ and pick a point $\tilde x \in \mathbb{CH}^n$ such that $\tilde p(\tilde x) = \gamma_0(0)$. The curve $\gamma_0$ then admits a unique lift $\tilde \gamma_0 \subset \mathbb{CH}^n$ such that $\tilde \gamma_0(0)=\tilde x$. From \cite{johanna_negative_curvature}, we know that $\mathbb{CH}^n$ is foliated by fully geodesic hyperbolic disks that is to say, there exists an embedding of a Poincar\'e disk  $\mathbb{D}_{\tilde \gamma_0} \hookrightarrow \mathbb{CH}^n$ of Gaussian curvature $\kappa$ with the property that $\mathbb{D}_{\tilde \gamma_0}$ is a capping disk for $\tilde \gamma_0$ and the surface $\Sigma_{g_0}$ mentioned in the previous paragraph is the projection of $\mathbb{D}_{\tilde \gamma_0}$ to $M$. In addition, the following holds:
\begin{align}
    \int_{\mathbb{D}_{\gamma_0}} s\sigma_{\text{0}} =\int_{\mathbb{D}_{\tilde \gamma_0}} s\tilde p^*\sigma_{\text{0}}
\end{align}
So, we once again loose no information by working on the universal cover $\mathbb{CH}^n$; we now compute the required integral on the Poincar\'e disk $\mathbb{D}_{\tilde \gamma_0} \subset \mathbb{CH}^n$. In the $(\theta, \phi)$ coordinates used in the case with $\kappa>0$ above, the magnetic geodesic $\gamma_0$ has the following expression:
\begin{align}\label{mag_geodesic_hyp_eqn}
    \gamma_0(\theta) =\frac{s}{\sqrt{-\kappa}} \sinh(\sqrt{-\kappa}\theta) 
\end{align}
And the metric induced on $\mathbb{D}_{\tilde \gamma_0}$ by the lift of $g_0$ to $\mathbb{CH}^n$ has the following expression:
\begin{align}\label{hyperboic_metric_equation}
    d\theta^2 + \frac{1}{-\kappa} \sinh^2(\sqrt{-\kappa}\theta) \; d\phi^2 =d\theta^2 + \big(\gamma_0(\theta)\big)^2 \; d\phi^2
\end{align}
The lift $\tilde p^* \sigma_0$ of the K\"ahler form $\sigma_0$ has the following expression:
\begin{align}\label{kahler_form_hyp_eqn}
    \tilde p^* \sigma_0 = \frac{1}{-\kappa} \sinh{(\sqrt{-\kappa}\theta)} \; d\theta \wedge d\phi
\end{align}
Using the above information, we obtain the following expression for the required integral: 
\begin{align}\label{disk_int_hyp}
   s\int_{\mathbb{D}_{\gamma_0}} \tilde p^*\sigma_0  = \frac{s}{-\kappa}\int^{2\pi}_0 \Big( \int^{\theta}_0 \sinh(\sqrt{-\kappa}t)\; dt \Big)d\phi \; = \; \frac{2s\pi}{-\kappa}[\cosh{\sqrt{-\kappa}\theta}-1] 
\end{align}
Using \eqref{mag_geodesic_hyp_eqn} and \eqref{hyperboic_metric_equation} we obtain the following formula for the geodesic curvature $\kappa_{g_0}$:
\begin{align}\label{hyp_geodesic_curvature}
    \kappa_{g_0} = \frac{\frac{d \gamma_0}{dt}}{\gamma_0} = \frac{\sqrt{-\kappa} \;\cosh(\sqrt{-\kappa}\theta)}{\sinh(\sqrt{-\kappa} \theta)} = \frac{\sqrt{-\kappa}}{\tanh{(\sqrt{-\kappa} \theta)}}
\end{align}
Combining \eqref{hyp_geodesic_curvature} with the equation $\kappa_{g_0} = s$ obtained from \cite[(10)]{johanna_projective} we obtain the following:
\begin{align}\label{tanh_final}
    \tanh{(-\kappa \theta)} = \frac{\sqrt{-\kappa} }{s}
\end{align}
Plugging \eqref{tanh_final} into  the identity $\cosh x= \frac{1}{\sqrt{1-\tanh^2x}}$ we obtain the following expression for  \eqref{disk_int_hyp}:
\begin{align}\label{disk_int_hyp_final}
s\int_{\mathbb{D}_{\gamma_0}} \tilde p^*\sigma_0 =  \frac{2\pi}{-\kappa}  \Bigg[s- \frac{s^2}{\sqrt{s^2+\kappa}}  \Bigg]
\end{align}
Subtracting \eqref{disk_int_hyp_final} from \eqref{length_hyperbolic} we obtain the following:
\begin{align}
    2\pi\Bigg[ \frac{1}{\sqrt{s^2+\kappa}} + \frac{s}{\kappa} - \frac{s^2}{\kappa\sqrt{s^2+\kappa}} \Bigg] = \pi \Bigg[ \frac{2}{-\kappa}\Big( \sqrt{\kappa +s^2} -s \Big)\Bigg] = \pi a^2(1)
\end{align}
The final equality above follows from the definition of the function $a$ in \eqref{def_of_a_mag_zoll} above.
\\
\\
\\
\textbf{\underline{Case 3: $\kappa =0$}}\\
\\
We begin with the observation that the right-hand side of the equation \eqref{length_hyperbolic} is continuous in the variable $\kappa$ and its Taylor expansion about $\kappa=0$ with $s,\kappa > 0$ is the following:
\begin{align}
 \frac{2\pi}{\sqrt{s^2+ \kappa }} =   \frac{2\pi}{s} \left[
1 - \frac{1}{2} \frac{\kappa}{s^2}
+ \frac{3}{8} \bigg( \frac{\kappa}{s^2} \bigg)^2
- \frac{5}{16} \bigg( \frac{\kappa}{s^2} \bigg)^3
+ \cdots
\right]
\end{align}
Evaluating the above expression at $\kappa=0$, we see that the length of a magnetic geodesic $\gamma_0 \in \chi(g_0,\sigma_0)$ has the following expression:
\begin{align}\label{torus_length}
    \text{length}_{g_0}(\gamma_0) = \frac{2\pi}{s}
\end{align}
Similarly, the right hand side of \eqref{disk_int_hyp_final} is a continuous function of $\kappa$ and it's Taylor expansion about $\kappa =0$ with $s,\kappa > 0$ has the following expression:
\begin{align}
 \frac{2\pi}{\kappa}  \Bigg[s - \frac{s^2}{\sqrt{s^2+\kappa}} \Bigg] =     \frac{\pi}{s} - \frac{3\pi}{4s^3} \kappa + \frac{5\pi}{8s^5} \kappa^2 - \cdots
\end{align}
Evaluating the above expression at $\kappa=0$ we see that, given any capping surface $\mathbb{D}_{\gamma_0}$ of the magnetic geodesic $\gamma_0$ we obtain the following expression for the integral of $\sigma_0$ over this surface:
\begin{align}\label{torus_disk_int}
    \int_{\mathbb{D}_{\gamma_0}} s\sigma_0 = \frac{\pi}{s}
\end{align}
Subtracting \eqref{torus_disk_int} from \eqref{torus_length} we obtain:
\begin{align}
    \frac{2\pi}{s} - \frac{\pi}{s} = \frac{\pi}{s} = \pi a^2(1)
\end{align}
\end{proof}
We will denote by $B_0 \in H^2(M,\mathbb{R})$ the cohomology class represented by the K\"ahler form $\sigma_0 \in \Omega^2(M)$ and we now pick an arbitrary pair $(g,\sigma) \in \mathcal{M}(M)\times \Xi^2_{B_0}(M)$. Recall that after fixing the Zoll magnetic pair $(g_0,\sigma_0)$ on the Riemannian manifold $(M,g_0)$, the set $\chi(g,\sigma)$ was introduced earlier to denote the set of closed magnetic geodesics of the magnetic pair $(g,\sigma)$ that are $C^2$-close to those of the Zoll pair $(g_0,\sigma_0)$.\\
\\
As explained in Section \ref{mag_dyn_riem_section}, the two magnetic pairs $(g_0,\sigma_0)$ and $(g,\sigma)$ define odd-symplectic forms on the manifold $(S^*_{g_0}M)$. These odd-sympelctic forms are the following:
\begin{align}
  \Omega_{g_0,\sigma_0}= \big(d\lambda-\pi^*\sigma_0\big)\vert_{S^*_{g_0}M} 
\end{align}
And,
\begin{align}
    \Omega_{g,\sigma}:= \big(d(f_g\lambda)-\pi^*\sigma\big)\vert_{S^*_{g_0}M} 
\end{align}
where $f_g$ is the function defined in \eqref{def_of_the_function_f}.\\
\\
Since the first summand in the definition of $\Omega_{g,\sigma}$ and $\Omega_{g_0,\sigma_0}$ is an exact $2$-form and $\sigma$ was chosen from the set $\Xi^2_{B_0}(M)$, we see that the odd-symplectic forms $\Omega_{g,\sigma}$ and $\Omega_{g_0,\sigma_0}$ and cohomologous. We can therefore define their magnetic action as in \eqref{mag_action_def_eqn}. The following proposition furnishes a connivent expression for the magnetic action $\mathcal{A}^{\text{mag}}_{{g,\sigma}}$ in terms of the magnetic length functional defined in \eqref{mag_length_functional_def}.
\begin{proposition}\label{cpn_action_paradigm_2}
Let $(M,g_0,\sigma_0)$ be a closed, connected K\"ahler manifold such that the associated K\"ahler metric $g_0$ has constant holomorphic sectional curvature $\kappa$ and denote by $B_0 \in H^2(M,\mathbb{R})$ the cohomology class represented by $\sigma_0$. Given a pair $(g,\sigma) \in \mathcal{M}(M) \times \Xi^2_{B_0}(M)$ such that the set $\chi(g,\sigma)$ is non-empty, its magnetic action $\mathcal{A}^{\text{mag}}_{{g,\sigma}}$ has the following expression:
\begin{align}
\mathcal{A}_{\text{mag}}^{{g,\sigma}}(\gamma) = l^{g, \sigma}_{\text{mag}}(\gamma) - \pi a^2(1) \qquad \text{for $\gamma \in \chi(g,\sigma)$}
\end{align}
where $a$ is the function defined in \eqref{def_of_a_mag_zoll}.
 \end{proposition}
\begin{proof}
Since $\sigma$ and $\sigma_{0}$ are cohomologous, there exists an $\eta \in \Omega^1(M)$ with the property that $\sigma = \sigma_{0} + d\eta$. The following is a direct consequence of this observation:
\begin{align}
    \Omega_{g,\sigma} = \Omega_{g_0,\sigma_0} + d\big((f_g-1)\lambda - s\pi^*\eta\big)
\end{align}
Since the strength $s$ is fixed, we will drop it from the notation in the sequel.\\
\\
Before we proceed further, we observe that a magnetic geodesic $\gamma \in \chi(\tilde g, \tilde\sigma)$ has a lift to the manifold $(S^*_{g_0}M)$ such that this lifted curve is a leaf of the characteristic foliation of the odd-symplectic form $\Omega_{\tilde g,\tilde \sigma}$. We will denote one such lift by $\tilde \gamma$. Since $\gamma$ is assumed to be $C^2$-close to a magnetic geodesic $\gamma_0$ of the Zoll magnetic pair $(g_0,\sigma_0)$, the curve $\tilde \gamma$ is contained in the same good chart on $S^*_{g_0}M$ as the lift $(\gamma_0,\dot \gamma_0)$; we will denote by $\Gamma_{\tilde \gamma}$ a short homotopy between these two lifts. In this setting, by \eqref{action_functional_eq}, \eqref{mag_action_def_eqn}, the magnetic action has the following expression:
\begin{align}\label{paradigm_2_1}
  \mathcal{A}_{\text{mag}}^{{\tilde g,\tilde \sigma}}(\gamma)   = \int_{\gamma_{0}}\big( (f_{\tilde g}-1)\lambda - \pi^*\eta\big) + \int_{\Gamma_{\tilde \gamma}} \Omega_{g_0,\sigma_0} + d((f_{\tilde g}-1)\lambda - \pi^*\eta) 
\end{align}
The integral \eqref{paradigm_2_1} then splits into the following two parts:
\begin{align}
 \text{I}: =   \int_{\gamma_{0}} \big((f_{\tilde g}-1)\lambda - \pi^*\eta\big) + \int_{\Gamma_{\tilde\gamma}} d((f_{\tilde g}-1)\lambda - \pi^*\eta) 
\end{align}
\begin{align}
    \text{II}:=  \int_{\Gamma_{\tilde \gamma}} d\lambda - \pi^*\sigma_{0}
\end{align}
By Stokes' theorem, the integral $\text{I}$ has the form:
\begin{align}
  \text{I} = \int_{\tilde \gamma} (f_{\tilde g}-1)\lambda - \int_{\mathbb{D}_{\gamma}} d\eta =  \text{length}_{\tilde g}(\gamma) - \text{length}_{g_0}(\gamma) - \int_{\mathbb{D}_{\gamma}} d\eta \\\nonumber 
\end{align}
where $\mathbb{D}_\gamma$ is a capping surface for $\gamma$ obtained by concatenating to the capping surface $\mathbb{D}_{\gamma_0}$ the projection of $\Gamma_{\tilde \gamma}$ to $M$ via the restriction of the foot-point projection.\\
\\
Using Stokes' theorem once again, we see that the integral $\text{II}$ has the following form:
\begin{align}
  &\text{II}=  \int_{\tilde \gamma}  \lambda - \int_{\dot \gamma_{0}}  \lambda - \int_{\Gamma_{\tilde \gamma}}  \pi^*\sigma_{0}\\\nonumber = 
 \text{length}_{g_0}&(\gamma) -  \text{length}_{g_0}(\gamma_{0}) - \int_{\mathbb{D}_{\gamma}} \sigma_{0} + \int_{\mathbb{D}_{\gamma_{0}}} \sigma_{0}  
\end{align}
Since $\sigma = \sigma_{0} + d\eta$, the integral $\text{I} + \text{II}$ can now be written in the form:
\begin{align}\label{action_para_2_1}
 \text{length}_{\tilde g}&(\gamma) - \int_{\mathbb{D}_{\gamma}} \sigma -  \text{length}_{g_0}(\gamma_{0}) + \int_{\mathbb{D}_{\gamma_0}} \sigma_{0} \\\nonumber
= l^{\tilde g,\tilde \sigma}_{\text{mag}}&(\gamma) - \text{length}_{g_0}(\gamma_{0}) + \int_{\mathbb{D}_{\gamma_{0}}} \sigma_{0}  
\end{align}
 In Proposition \ref{action_of_fubini_study_mag_system}, the following identity was established:
\begin{align}\label{action_para_2_2}
 \text{length}_{g_0}(\gamma_{0}) + s_0\int_{\mathbb{D}_{\gamma_{0}}} \sigma_{0} = \pi a^2(1)
\end{align}
Plugging \eqref{action_para_2_2} into \eqref{action_para_2_1}, we obtain the required expression for the action. 
\end{proof}
\subsection{Proof of Theorem \ref{mag_sys_ineq}}\label{mag_sys_ineq_proof}
Let $(M,\sigma_0,J)$ be a closed connected K\"ahler manifold such that the associated K\"ahler metric $g_0$ has constant holomorphic sectional curvature $\kappa$. By Theorem \ref{zollness_of_mag_systems_kahler}, the following is a Zoll odd-symplectic form when $s^2+\kappa>0$:
\begin{align}\label{kahler_sys_ineq_proof_os_1_def}
    \Omega_{g_0,\sigma_0}:= \big(d\lambda -s\pi^*\sigma_0\big)\big\vert_{S^*_{g_0}M}
\end{align}
In the sequel, we will fix one $s\in\mathbb{R}$ satisfying the condition $s^2+\kappa>0$ and drop it from the notation.\\
\\
Let $\tilde B_0 \in H^2(S_{g_0}^*M,\mathbb{R})$ be the cohomology class represented by the odd-symplectic form \eqref{kahler_sys_ineq_proof_os_1_def}. Then, by Theorem \ref{systolic_ineq}, there exists a $C^2$-neighborhood $\mathcal{U}\subset \Xi^2_{\tilde B_0}(S_{g_0}^*M)$ with the property that for any odd-symplectic form $\Omega \in \mathcal{U}$, a local systolic inequality of the form \eqref{systolic_ineq_eqn} holds.\\
\\
We now denote by $B_0 \in H^2(M,\mathbb{R})$ the cohomology class represented by the K\"ahler form $\sigma_0$, and pick a pair $(\tilde g,\tilde \sigma) \in \mathcal{M}(M) \times \Xi^2_{B_0}(M)$. As discussed in Section \ref{mag_dyn_riem_section}, the pair $(\tilde g,\tilde \sigma)$ defines the following odd-symplectic form on the manifold $S_{g_0}^*M$:
\begin{align}
    \tilde \Omega_{\tilde g,\tilde \sigma}:= \big(d(f_{\tilde g}\lambda) -\pi^*\tilde \sigma\big)\big\vert_{S^*_{g_0}M}
\end{align}
In Section \ref{mag_sys_ineq_general_proof}, it is shown that there exists a $C^3$-neighborhood $\tilde{\mathcal{U}} \subset \mathcal{M}(M) \times \Xi^2_{B_0}(M)$ of the pair $(g_0,\sigma_0)$ such that the odd-symplectic form defined by any pair $(\tilde g,\tilde \sigma) \in \tilde{\mathcal{U}}$ is contained in the neighborhood $\mathcal{U} \subset \Xi^2_{\tilde B_0}(S^*_{g_0}M)$. We now fix one such pair $(\tilde g,\tilde \sigma) \in \tilde{\mathcal{U}}$ and denote by $ \tilde \Omega_{\tilde g,\tilde \sigma}$ the odd-symplectic form it defines on the manifold $(S_{g_0}^*M)$. By Theorem \ref{systolic_ineq}, the following holds:
\begin{align}\label{mag_final_ineq_1}
    P\big(\inf_{\gamma \in \chi(\tilde g,\tilde \sigma)}\mathcal{A}^{\text{mag}}_{\tilde g,\tilde \sigma}(\gamma)\big) \leq \text{Vol}_{\Omega_{g_0,\sigma_0}}(\tilde \Omega_{\tilde g,\tilde \sigma})
\end{align}
with equality holding if and only if the magnetic system determined by the pair $(\tilde g,\tilde \sigma)$ is Zoll at strength $s$; where $\mathcal{A}^{\text{mag}}_{\tilde g,\tilde \sigma}$ is the magnetic action defined in \eqref{mag_action_def_eqn}. By Proposition \ref{cpn_action_paradigm_2}, the following holds:
\begin{align}\label{kahler_proof_action_final}
    \inf_{\gamma \in \chi(\tilde g,\tilde \sigma)}\mathcal{A}^{\text{mag}}_{\tilde g,\tilde \sigma}(\gamma) = \inf_{\gamma \in \chi(\tilde g,\tilde \sigma)} l_{\text{mag}}^{\tilde g,\tilde \sigma}(\gamma) - a^2(1)\pi  =l_{\text{min}}^{\tilde g,\tilde \sigma} - a^2(1)\pi
\end{align}
where $l_{\text{mag}}^{\tilde g,\tilde \sigma}$ is the magnetic length functional defined in \eqref{mag_length_functional_def}.\\
\\
By Proposition \ref{Zoll_poly_cpn_proposition}, the Zoll polynomial of the magnetic pair $(g_0,\sigma_0)$ has the following expression:
\begin{align}\label{zoll_poly_final_expression}
 P(\tilde A) = \begin{cases}
  n!\pi^{n-1}[a(1)]^{2n-2}
    \;\text{Vol}_{g_0}(M)\; \big(A(\tilde A)C(\tilde A)\big)^{n-1}\bigg(C(\tilde A) -\frac{B(\tilde A)}{n} \bigg) \;\;\;&\text{if $\kappa \neq 0$} 
    \\
    \\
       (-s)^n A(\tilde A)^{n-1}\; n!\pi^{n-1}[a(1)]^{2n-2}
    \;\text{Vol}_{g_0}(M)\;\;\; &\text{if $\kappa =0$}
 \end{cases}  
 \end{align}
 where $s \in \mathbb{R}$ is any constant satisfying the condition of Theorem \ref{zollness_of_mag_systems_kahler}, $a$ is the function defined in that theorem and :
 \begin{align}\label{def_of_K_tilde_final_proof}
   A(\tilde A)=1+\frac{\tilde A}{\pi a^2(1)},\qquad
B(\tilde A)=\frac{\kappa\bigl(\tilde A+\pi a^2(1)\bigr)}{4\pi a^2(1)},\qquad
C(\tilde A)=\frac{\kappa\bigl(\tilde A+\pi a^2(1)\bigr)}{4\pi}-s
 \end{align}
 We now abbreviate $\inf_{\gamma \in \chi(\tilde g,\tilde \sigma)}\mathcal{A}^{\text{mag}}_{\tilde g,\tilde \sigma}(\gamma)$ as $\inf \mathcal{A}^{\text{mag}}_{\tilde g,\tilde \sigma}$ and plug \eqref{kahler_proof_action_final} into \eqref{def_of_K_tilde_final_proof} to obtain the following:
 \begin{align}\label{ABC_final_proof}
     A(\inf \mathcal{A}^{\text{mag}}_{\tilde g,\tilde \sigma}) = \frac{l_{\text{min}}^{\tilde g,\tilde \sigma}}{\pi a^2(1)}, \qquad B(\inf \mathcal{A}^{\text{mag}}_{\tilde g,\tilde \sigma})= \frac{\kappa\big(l_{\text{min}}^{\tilde g,\tilde \sigma}\bigr)}{4\pi a^2(1)},\qquad C(\inf \mathcal{A}^{\text{mag}}_{\tilde g,\tilde \sigma})= \frac{\kappa\big(l_{\text{min}}^{\tilde g,\tilde \sigma}\bigr)}{4\pi} -s
 \end{align}
Plugging \eqref{ABC_final_proof} into \eqref{zoll_poly_final_expression}, we obtain the following expression for $P(\inf \mathcal{A}^{\text{mag}}_{\tilde g,\tilde \sigma})$ when $\kappa \neq 0$:
\begin{align}\label{zoll_poly_final_1}
    \frac{n!}{4^n\pi^n a^2(1)}
    \Bigg(
        l_{\text{min}}^{\tilde g,\tilde \sigma}
        \Big(
            \kappa l_{\text{min}}^{\tilde g,\tilde \sigma}
            -4\pi^2s a^2(1)
        \Big)
    \Bigg)^{n-1}
    \Bigg(
        \kappa l_{\text{min}}^{\tilde g,\tilde \sigma}
        \big(na^2(1)-1\big)
        -4\pi s a^2(1)
    \Bigg)
    \;\text{Vol}_{g_0}(M).
\end{align}
Similarly, when $\kappa =0$, we obtain the following expression for $P(\inf \mathcal{A}^{\text{mag}}_{\tilde g,\tilde \sigma})$:
\begin{align}
    (-s)^n \Bigg(&\frac{l_{\text{min}}^{\tilde g,\tilde \sigma}}{\pi a^2(1)} \Bigg)^{n-1}n!\pi^{n-1}[a(1)]^{2n-2}
    \;\text{Vol}_{g_0}(M) \\\nonumber = &(-s)^n \Big(l_{\text{min}}^{\tilde g,\tilde \sigma}\Big)^{n-1}n!
    \;\text{Vol}_{g_0}(M)
\end{align}
From Proposition \ref{vol_mag_sys_general}, the right hand side of \eqref{mag_final_ineq_1} has the following expression:
\begin{align}\label{vol_kahler_final_1}
Vol_{\Omega_{g_0,\sigma_0}}(\tilde \Omega_{\tilde g, \tilde \sigma}) =     2n\text{Vol}(B^{2n}) \Big( \text{Vol}_{\tilde g}(M) - \text{Vol}_{g_0}(M) \Big)
\end{align}
where $\text{Vol}(B^{2n})$ is the Euclidean volume of the $2n$-dimensional ball. It has the following expression:
\begin{align}
    \text{Vol}(B^{2n})
=
\frac{\pi^{n}}{(2n-1)!}
\end{align}
Plugging the above expression into \eqref{vol_kahler_final_1} and plugging the resulting expression into \eqref{mag_final_ineq_1} along with \eqref{zoll_poly_final_1}, we obtain the following systolic inequality when $\kappa \neq 0$:
\begin{align}
    \Bigg(
        l_{\text{min}}^{\tilde g,\tilde \sigma}
        \Big(
            \kappa l_{\text{min}}^{\tilde g,\tilde \sigma}
            -&4\pi^2s a^2(1)
        \Big)
    \Bigg)^{n-1}
    \Bigg(
        \kappa l_{\text{min}}^{\tilde g,\tilde \sigma}
        \big(na^2(1)-1\big)
        -4\pi s a^2(1)
    \Bigg) \\\nonumber \leq& \frac{2^{2n+1}\pi^{2n}a^2(1)}{n!(2n-1)!} \Bigg(\frac{\text{Vol}_{\tilde g}(M)}{\text{Vol}_{g_0}(M)}-1\Bigg)
\end{align}
with equality holding if and only if the magnetic system determined by the pair $(\tilde g,\tilde \sigma)$ is Zoll at strength $s$.\\
\\
Similarly, when $\kappa=0$, we obtain the following systolic inequality:
\begin{align}
    \Big(l_{\text{min}}^{\tilde g,\tilde \sigma}\Big)^{n-1} \leq \frac{\pi^n}{(-s)^n\; n!\;(2n-1)! }\Bigg(\frac{\text{Vol}_{\tilde g}(M)}{\text{Vol}_{g_0}(M)}-1\Bigg)
\end{align}
We now observe that if $\text{Vol}_{g_0}(M) = \text{Vol}_{\tilde g}(M)$ then, by \eqref{vol_kahler_final_1}, the following holds:
\begin{align}
    Vol_{\Omega_{g_0,\sigma_0}}(\tilde \Omega_{\tilde g, \tilde \sigma}) =0
\end{align}
Theorem \ref{systolic_ineq} then implies that the local systolic inequality \eqref{mag_final_ineq_1} reduces to the following:
\begin{align}
    \inf \mathcal{A}^{\text{mag}}_{\tilde g,\tilde \sigma} \leq 0
\end{align}
Plugging \eqref{kahler_proof_action_final} into the above equation, we then obtain the following:
\begin{align}
    l_{\text{min}}^{\tilde g,\tilde \sigma} \leq  a^2(1)\pi
\end{align}
with equality holding if and only if the magnetic system determined by the pair $(\tilde g,\tilde \sigma)$ is Zoll at strength $s$.

\section{Appendix: Some facts about differential forms}
\subsection{The size of a $k$-form in terms of its differential}
\begin{proposition}\label{elliptic_bounds_on_1_forms}
Let $(M,g)$ be a smooth closed Riemannian manifold of dimension $n$ and let $\Tilde{\alpha} \in \Omega^k(M)$ for $k < n$. Then, we can find an $\alpha \in \Omega^1(M)$ and a sequence $\{L_{k}\}_{k \in \mathbb{N}} \in \mathbb{R}_{>0}$ such that the following holds:
\begin{enumerate}
    \item $d\alpha = d\Tilde{\alpha}$
    \item $||\alpha||_{C^k} \leq L_k ||d\Tilde{\alpha}||_{C^k}$ $\forall k \in \mathbb{N}_{\geq 1}$. 
\end{enumerate}
\end{proposition}
\begin{proof}
The idea of the proof is simple, we use the ellipticity of the Hodge Laplacian to obtain the required bounds after using the Hodge decomposition of $\Omega^k(M)$ to find the required $\alpha$.\\
\\
Towards that end, recall that the Hodge Laplacian is a formally self-adjoint elliptic operator defined in the following way(c.f \cite[10.1.22 \& 10.1.29]{Nicolaescu1996LecturesOT}:
\begin{align}\label{Hodge_lap_def}
    H_{\Delta,k}:= \delta_{k+1} \circ d_{k} + d_{k-1}\circ \delta_{k} : \Omega^k(M) \rightarrow \Omega^k(M) 
\end{align}
where $d_k: \Omega^k(M) \rightarrow \Omega^{k+1}(M)$ is the deRahm exterior derivative and $\delta_{k+1}: \Omega^{k+1}(M) \rightarrow \Omega^{k}M)$ it's formal adjoint with respect to the $L^2$-inner product. By the Hodge decomposition theorem \cite{Nicolaescu1996LecturesOT}[Theorem 10.4.30 (3)] we know that $\Omega^k(M)$ has the following orthogonal decomposition:
\begin{align*}
    \Omega^k(\Sigma)= \ker(H_{\Delta,k}) \bigoplus \Ima(d_{k-1}) \bigoplus \Ima (\delta_{k+1})
\end{align*}
We now define $\alpha$ to be the projection of $\Tilde{\alpha}$ onto the third factor in the above direct sum decomposition. This ensures that $d\alpha = d\Tilde{\alpha}$ because, clearly $\ker (H_{\Delta,k})= \ker (d_k) \cap \ker (\delta_{k})$. So, to finish the proof, we need only to establish the bounds in the statement of Proposition \ref{elliptic_bounds_on_1_forms}. \\
\\
Since $\alpha$ is orthogonal to the kernel of $H_{\Delta,k}$, the Poincar\'e inequality (\cite[Lemma 10.4.9]{Nicolaescu1996LecturesOT}) and Morrey's inequality \cite[10.2.36 (d)]{Nicolaescu1996LecturesOT} imply the following bounds:
\begin{align}\label{hodge_proof_chain_1}
   ||\alpha||_{C^q} \leq K_1 ||\alpha||_{W^{q+1,2}} \leq K_2 ||H_{\Delta,k}(\alpha)||_{W^{q+1,2}}
\end{align}
where $W^{q,2}$ is the corresponding Sobolev space and $q,K_1,K_2 \in \mathbb{R}_{>1}$ some constants. Using the Sobolev embedding theorem and the definition of the Hodge Laplacian in \eqref{Hodge_lap_def} we get the following chain of inequalities:
\begin{align}\label{hodge_proof_chian_2}
    ||H_{\Delta,k}(\alpha)||_{W^{q,2}} \leq C_1 ||\delta_{k-1}(d_k(\alpha))||_{C^q} \leq C_1||\Sigma^n_{i=1} \nabla_{e_i}\big(\iota_{e_i} d\alpha\big)||_{C^q} \leq C_2||d\alpha ||_{C^k} 
\end{align}
The last inequality in \eqref{hodge_proof_chian_2} follows from \cite[Chapter 4, Lemma 4.6]{lee2018introduction}. Combining \eqref{hodge_proof_chain_1} and \eqref{hodge_proof_chian_2} we obtained the required bounds on $\alpha$. 
\end{proof}

\subsection{The size of Reeb vector fields of odd-symplectic forms}
\begin{lemma}\label{Reeb_dist_lemma}
Let $(\Sigma,\Omega_0)$ be a closed and oriented odd-symplectic manifold. After picking an $\alpha_0 \in \mathcal{C}(\Sigma,\Omega_0)$, there exists a $C^1$-neighborhood $\mathcal{U} \subset \Omega^2(\Sigma)$ of $\Omega_0$ with the property that for any odd-symplectic form $\Omega \in \mathcal{U}$, $\alpha_0 \in \mathcal{C}(\Sigma,\Omega)$. For any such $\Omega \in \mathcal{U}$, there exists a monotonically increasing function $\sigma: \mathbb{R} \rightarrow \mathbb{R}_{>0}$ the following holds for all $k\geq 1$:
\begin{align}
     \vert\vert R_{\Omega} - R_{0}\vert\vert _{C^{1}} \leq \sigma(\vert\vert \Omega - \Omega_0\vert\vert _{C^{1}}) 
\end{align}
where $R_{\Omega}$ and $R_{0}$ are the Hamiltonian vector fields of $\Omega$ and $\Omega_0$ respectively defined with respect the $1$-form $\alpha_0$. 
\end{lemma}
\begin{proof}
In this proof, we will denote by $2n-1$ for $n\geq 2$ the dimension of $\Sigma$.\\
\\
Consider the following map:
\begin{align}\label{def_of_K_appendix}
    K_{\Omega_0}: TM \rightarrow \Omega^{2n-1}(M) \\\nonumber 
    K_{\Omega_0}(X) \mapsto \iota_X \big[\alpha_0  \wedge \Omega_0^{n-1} \big]
\end{align}
This map is smooth and invertible. We now use this inverse to define a vector field in the following way:
\begin{align*}
    \tilde{R}_{\Omega}:= K^{-1}_{\Omega_0}(\Omega^{n-1})
\end{align*}
We then obtain the following chain of equalities:
\begin{align*}
    \iota_{ \tilde{R}_{\Omega}} \Omega^{n-1} =  \iota_{ \tilde{R}_{\Omega}} K_{\Omega_0} (\tilde{R}_{\Omega}) = 
    \iota_{ \tilde{R}_{\Omega}}  \iota_{ \tilde{R}_{\Omega}} (\alpha_0 \wedge \Omega_0^{n-1}) =0
\end{align*}
The Hamiltonian vector field of the pair $(\alpha_0, \Omega)$ can be written in the following form:
\begin{align*}
    R_{\Omega} = \frac{\tilde{R}_{\Omega}}{\alpha_0(\tilde{R}_{\Omega_0})}  = \frac{K^{-1}(\Omega^{n-1})}{\alpha_0(\tilde{R}_{\Omega_0})}
\end{align*}
This in turn implies the following as required:
\begin{align*}
    \vert\vert R_{\Omega} - R_0\vert\vert_{C^1} = \Big|\Big|\frac{1}{\alpha_0}\Big|\Big|_{C^1} \vert\vert K^{-1}_{\Omega_0}(\Omega^{n-1}) - K^{-1}_{\Omega_0}(\Omega_0^{n-1})\vert\vert_{C^1} \leq \sigma(\vert\vert\Omega - \Omega_0\vert\vert_{C^{1}})
\end{align*}
where the last inequality follows from the smoothness of $K_{\Omega_0}$ defined in \eqref{def_of_K_appendix}.
\end{proof}

\bibliographystyle{alpha}
\bibliography{refs_mag_1}

\end{document}